%% file: Curvature.tex
\documentclass[10pt]{amsart}
\allowdisplaybreaks

\usepackage{amsfonts}
\usepackage{amsmath}
\usepackage{amssymb}
\usepackage{amsthm}
\usepackage{amsbsy}
\usepackage{kbordermatrix}
\usepackage{graphicx} \usepackage{enumerate} \usepackage{multicol}
\usepackage{mathrsfs} \usepackage[all,cmtip]{xy}
\usepackage{todonotes}
\usepackage{ytableau}
\usepackage{stmaryrd}

\renewcommand{\kbldelim}{(}
\renewcommand{\kbrdelim}{)}

\newcounter{dummy} \numberwithin{dummy}{section}
\newtheorem{theorem}[dummy]{Theorem}
\newtheorem{corollary}[dummy]{Corollary}
\newtheorem{lemma}[dummy]{Lemma}
\newtheorem{definition}[dummy]{Definition}
\newtheorem{proposition}[dummy]{Proposition}
\theoremstyle{remark}
\newtheorem{remark}[dummy]{Remark}
\newtheorem{example}[dummy]{Example}

\newcommand{\calL}{\mathcal{L}}
\newcommand{\calH}{\mathcal{H}}
\newcommand{\calE}{\mathcal{E}}
\newcommand{\calA}{\mathcal{A}}
\newcommand{\calR}{\mathcal{R}}
\newcommand{\calV}{\mathcal{V}}
\newcommand{\calZ}{\mathcal{Z}}

\DeclareMathOperator{\Ann}{Ann}
\DeclareMathOperator{\diam}{diam}
\DeclareMathOperator{\Sym}{Sym}
\DeclareMathOperator{\End}{End}
\DeclareMathOperator{\id}{id}

\DeclareMathOperator{\pr}{pr}
\DeclareMathOperator{\rank}{rank}
\DeclareMathOperator{\spn}{span}
\DeclareMathOperator{\tr}{tr}

\DeclareMathOperator{\Ric}{Ric}
\DeclareMathOperator{\vl}{vl}

\DeclareMathOperator{\Ad}{Ad}

\DeclareMathOperator{\SO}{SO}

\DeclareMathOperator{\so}{\mathfrak{so}}

\DeclareMathOperator{\LQ}{LQ}
\newcommand{\eul}{\mathrm{e}}
\newcommand{\bpi}{\boldsymbol \pi}
\newcommand{\bpsi}{\boldsymbol \psi}
\newcommand{\ve}{\varepsilon}
\newcommand{\ptr}{/\!  /}
\newcommand{\hptr}{/ \! \hat{/}}
\newcommand{\dvec}{\vec{\partial}}

\numberwithin{equation}{section}

\title[Affine connections and curvature in sub-Riemannian geometry]{Affine connections and curvature in sub-Riemannian geometry}
\author[E.~Grong]{Erlend Grong}

\address{University of Bergen, Department of Mathematics, P.O.~Box 7803, 5020 Bergen, Norway}
\email{erlend.grong@gmail.com}



\begin{document}

\begin{abstract}
We introduce a new approach for computing curvature of sub-Riemannian manifolds. Curvature is here meant as symplectic invariants of Jacobi curves of geodesics, as introduced by Zelenko and Li. We describe how they can be expressed using a compatible affine connection and induced tensors, without any restriction on our sub-Riemannian manifold or the choice of connection. In particular, we obtain a universal Bonnet-Myers theorem of Riemannian type. We also give universal formulas for the canonical horizontal frame along each geodesic and an algorithm for computing the curvature in general. Several examples are included to demonstrate the theory.
\end{abstract}

\maketitle


\input{1Intro}
\input{4Ham}
\input{2SubR}

\input{3JCurves}

\input{5JHam}

\input{6Cont}

\input{Dim3}

\input{7Step2}

\input{ModelStep2}
\appendix
\input{BConnect}

\bibliographystyle{habbrv}
\bibliography{Bibliography}

\end{document}

%% file: 1Intro.tex
\section{Introduction} \label{sec:Introduction}
Understanding curvature for a sub-Riemannian manifolds $(M, \calE, g)$ involves a number of intricacies which are not present for their Riemannian counterparts. Rather than a single flat model space for each dimension, there exists a wide range of flat models called Carnot groups which sub-Riemannian manifolds have as metric tangent cones \cite{Bel96}. The structure of constant curvature models is more complicated still, see \cite{Hug95,Mor08,Gro16,AMS19,BeGr19} for some results. When it comes to more general sub-Riemannian manifolds, there exists approaches from both the Lagrangian and Eulerian point of view. The Eulerian approach introduced in \cite{BaGa17} by Baudoin and Garofalo considers interactions of the sub-Riemannian heat flow and the sub-Riemannian metric. For most results, a non-canonical choice of taming Riemannian metric $\bar{g}$ of $g$ is a necessity. There is a wide class of sub-Riemannian manifolds for which the theory can be applied see e.g. \cite{BBG14,BKW14,GrTh16a,GrTh16b}, but results obtained are often not sharp.

This paper will focus on the Lagrangian approach, introduced by Zelenko and Li \cite{ZeLi07,ZeLi09}, and further developed in \cite{BaRi16,BaRi17,ABR18}. This approach considers curvature by looking at geodesic variations in a sub-Riemannian manifold. Of applications of this theory, we mention comparison theorems for conjugate points and diameter \cite{BaRi16,BaIv19}, sub-Laplacian and volume comparison theorems \cite{AgLe15,LeLi18,BaRi19}, measure contraction properties \cite{LLZ16} and interpolation inequalities \cite{BaRi19b,BaRi19}.
The concrete examples done so far are given in the codimension one case \cite{LiZe11,LiZh13}, contact geometry \cite{AgLe14,ABR17},
some Carnot group of rank~2 \cite{Isi17} and 3-Sasakian manifolds \cite{RiSi19}. We also mention papers \cite{BGKT19,BGMR19} which use a Lagrangian approach for Sasakian and H-type manifolds as well, while also relying on a taming metric. One of the main challenges for further developing the theory has been computational difficulties in finding the connection and curvatures associated to this formalism, see Section~\ref{sec:LagrangianGrassmannian} for details.

The objective of this paper is to provide general tools for computation of the connection and curvature in the Lagrangian approach. We give an algorithm for computing this canonical curvature from any choice of affine connection $\nabla$ on $TM$ compatible with the sub-Riemannian structure. For details of the final form of this algorithm, see Section~\ref{sec:Algorithm}. The key computational tool is the introduction of \emph{twist polynomials} to connect parallel transport of the connection $\nabla$ with that of its adjoint $\hat \nabla$, the latter parallel transport being the one determining the geodesics, see Proposition~\ref{pro:NormalGeodesics}. To show the effectiveness of the approach, we emphasize the following result, applicable to any sub-Riemannian manifold satisfying the completeness condition $(*)$.

Let $(M, \calE, g)$ be a sub-Riemannian manifold satisfying the following condition.
\begin{enumerate}
\item[$(*)$] \label{item:Complete} Assume that $(M , \calE, g)$ is complete and that for some $x \in M$ there is a dense subset of $M$ that can be reached from $x$ by a normal, minimizing, ample, equiregular geodesic.
\end{enumerate}
For definition of ample and equiregular, see Section~\ref{sec:NormalGeodesics}. In order to make the result more presentable, we will use a specific choice of compatible connection and refer to Theorem~\ref{th:BMFinalBox} for the result using any compatible connection.  Let $\sharp: T^*M \to \calE$ be the map corresponding to the sub-Riemannian structure $(\calE, g)$ defined by $p(v) = \langle \sharp p, v \rangle_g$ for any $p \in T^*M$, $v \in \calE$. Write $\Ann(\calE) = \ker \sharp$ for the subbundle of covectors vanishing on $\calE$. A Riemannian metric $\bar{g}$ is said to \emph{tame} $g$ if $\bar{g}| \calE = g$. Let $\bar{g}$ be an arbitrary such metric with orthogonal complement $(\calE)^\perp = \calA$ and Levi-Civita connection $\nabla^{\bar{g}}$. Let $\nabla$ be the connection defined by
\begin{equation} \label{NiceNabla}
\nabla_X Y = \left\{ \begin{array}{ll}
\pr_{\calE} \nabla_X^{\bar{g}} Y & \text{if $X,Y \in \Gamma(\calE)$,} \\
\pr_{\calE} [X, Y] + \frac{1}{2} \sharp (\calL_X \pr_{\calE} g)(Y, \, \cdot \,) & \text{if $X \in \Gamma(\calA)$, $Y \in \Gamma(\calE)$,} \\
\pr_{\calA} [X, Y] & \text{if $X \in \Gamma(\calE)$, $Y \in \Gamma(\calA)$,} \\
\pr_{\calA} \nabla_X^{\bar{g}} Y & \text{if $X,Y \in \Gamma(\calA)$,} 
\end{array} \right. 
\end{equation}
Denote the torsion and curvature of $\nabla$ by respectively $T$ and $R$.
\begin{theorem}[Universal Bonnet-Myers theorem of Riemannian type] \label{th:Universal}
For any $p \in T^* M \setminus \Ann(\calE)$, define
$$\underline{\Box}_p = \{ v \in \calE_x \, : \, T(\sharp p, v) = 0\},$$
with orthogonal complement $\underline{\Box}_p^\perp = \ker \underline{\pr}$ in $\calE_x$. Then this subspace is independent of taming metric $\bar{g}$. Furthermore, for every $p \in T^* M \setminus \Ann(\calE)$, there is a unique linear map $C_p: \underline{\Box}_p \to \underline{\Box}_p^\perp$ satisfying
$$T(\sharp p, C_p u) =  (\nabla_{\sharp p} T)(\sharp p, u)- \tr_{\calE} p (T(\sharp p, \times)) T( \times, u) , \qquad u \in \underline{\Box}_p .$$
Let $\underline{\tr} = \underline{\tr}_p$ denote the trace over $\underline{\Box}_p$. Define a map $\underline{\Ric}: T^*M \setminus \Ann(\calE) \to \mathbb{R}$ by
\begin{align*}
\underline{\Ric} & = \underline{\tr} \langle R(\sharp p, \times) \times ,\sharp p \rangle_g +  \underline{\tr} \, p( (\nabla_{\times} T)(\sharp p, \times)) + \frac{1}{4} | p (T(\underline{\pr} \, \cdot \, ,\underline{\pr} \, \cdot \,)) |_{g^* \otimes g^*}^2 \\ 
& \qquad -  | C_p \underline{\pr} \, \cdot \, |_{g^* \otimes g}^2 -  \underline{\tr} \, p(T(\times , C_p \times )) .
\end{align*}
Assume that condition $(*)$ holds and that for any $p \in T^*M \setminus \Ann(\calE) $, $\rank \Box_p >1$ and
$$\frac{1}{\rank \Box_p -1}\underline{\Ric}(p) \geq k_1 |\sharp p|^2,$$
for some $k_1 >0$ independent of $p$. Then $M$ is compact, with finite fundamental group and of diameter bounded by $\frac{\pi}{\sqrt{k_1}}$.
\end{theorem}
We will give the proof in Section~\ref{sec:ProofThUniversal}. The above statement is sharp for the Hopf fibriation $S^1 \to S^{2n+1} \to S^{2n}$ when $n \geq 2$, see \cite{LeLi18,BGKT19}, and generalizes results given in~\cite{BGMR19}.

\

\paragraph{\it Structure of the paper and main results} In Section~\ref{sec:ConHam} we will present some important observations regarding affine connections and Hamiltonian systems. In particular, Lie derivatives with respect to Hamiltonian vector field are described in Lemma~\ref{lemma:HamBracket} in terms of covariant derivatives and curvature. In Section~\ref{sec:SRgeometry} we give some preliminaries of sub-Riemannian geometry, and relate compatible connections with normal geodesics and abnormal curves. In Section~\ref{sec:JacCurves} we describe the theory of Jacobi curves and curvature of sub-Riemannian manifolds as presented in \cite{ZeLi09}. We rewrite this theory using a chosen compatible connection affine connection $\nabla$ in Section~\ref{sec:JacCon}, by introducing the twist polynomials from which we can determine the Young diagram corresponding to a normal equiregular geodesic and a canonical decomposition of the horizontal bundle $\calE$ along a geodesic. We also give a universal formula for the horizontal part of the canonical frame along any ample and equiregular geodesic in any sub-Riemannian manifold in Theorem~\ref{th:HorizontalFrame}. We emphasize that the twist polynomials are global objects, allowing in Section~\ref{sec:Global} to express the curvature in terms of global tensors rather than along individual geodesic. In particular, an explicit algorithm for computation of the canonical connection and curvature is given in Section~\ref{sec:Algorithm}.

The three next sections consist of examples all satisfying condition $(*)$ which show our theory in practice. In Section~\ref{sec:23}, our methods are applied to the simplest non-trivial case, namely sub-Riemannian manifolds with growth vector $(2,3)$, to give a frame of comparison to results already found in \cite[Section~7.5]{ABR18}, \cite{ABR17} and \cite[Section~17]{ABB19}. In Section~\ref{sec:Step2}, we consider fat sub-Riemannian manifolds. For the sake of avoiding long computations, we will only find the canonical horizontal frame for the general case, and complete the computation of connection and curvature only the special case of H-type sub-Riemannian manifolds \cite{BGMR18} with some geometric restrictions. In Section~\ref{sec:Model} we similarly give the canonical horizontal frame for all step two spaces with maximal isometry groups and do a complete description for the case of rank~3 sub-Riemannian structures.

Some theory and computations related to connections are found in Appendix~\ref{sec:Connections}. In particular, we give details on pull-back connections and curvature of non-linear connections.

\

\paragraph{\it Acknowledgement} The author would like to thank Andrei Agrachev, Pierre Pansu, Luca Rizzi and Igor Zelenko for helpful discussion. This project is supported in part by the Research Council of Norway (project number 249980/F20).

%% file: 4Ham.tex
\section{Connections and hamiltonian functions} \label{sec:ConHam}
\subsection{Affine connections and corresponding Ehresmann connections} \label{sec:AffineDef}
We review some general theory relating to affine connections on vector bundles and Ehresmann connections, and we refer to \cite[Chapter III]{KMS93} for details. Let $\pi: \calA \to M$ be an arbitrary vector bundle. We define $\calV = \ker \pi_* \subseteq T\calA$ as \emph{the vertical bundle}. For any element $a_1, a_2 \in \calA_x$, we define \emph{the vertical lift of $a_2$ to $a_1$} by
$$\vl_{a_1} a_2 = \frac{d}{dt} (a_1 + t a_2) |_{t=0} \in T_{a_1}\calA.$$
Clearly, we then have $\calV_{a_1} = \{ \vl_{a_1} a_2 : a_2 \in \calA_x\}$ for any $a_1 \in \calA_x$. If $A \in \Gamma(\calA)$ is a section, we can define its vertical lift $\vl A \in \Gamma(\calV)$ by $\vl A|_a = \vl_a A_{\pi(a)}$, $a \in \mathcal{A}$. 

Let $\nabla$ be an affine connection on $\calA$. Define $\calH = \calH^{\nabla} \subseteq T\calA$ by
$$\calH = \left\{ \frac{d}{dt} a(t) |_{t=0} \, : \, \text{$a(t)$ is $\nabla$-parallel along $\pi(a(t))$} \right\}.$$
We note that $T\calA = \calH \oplus \calV$, making $\calH$ an Ehresmann connection on $\pi$ and hence permitting us to define horizontal lifts with respect to $\calH$. For any $v \in T_x M$, $a \in \calA_x$, $x\in M$, let \emph{the horizontal lift} $h_a v$ be the unique element in $\calH$ that projects to $v$ by $\pi_*$. Similarly, if~$X$ is a vector field on $M$ we define its horizontal lift by $hX|_a = h_a X_{\pi(a)}$.

In conclusion, any element in $T\calA$ can be written as a sum of a vertical lift of an element in $\calA$ and a horizontal lift of an element in $TM$. We note these identities from the definitions of~$\calH$ and vertical lifts; for any vector fields $X, Y \in \Gamma(TM)$ and sections $A,B \in \Gamma(\calA)$,
\begin{equation} \label{BracketsHV} [hX, hY] = h[X,Y] - \vl R^{\nabla}(X,Y), \quad [hX, \vl A] = \vl \nabla_X A, \quad [\vl A, \vl B] =0. \end{equation}
Here, $\vl R^{\nabla}(X,Y)$ denotes the vector field $a \in \calA \mapsto \vl_a R^\nabla(X,Y)a$ corresponding to the curvature $R(X,Y) = \left[ \nabla_X, \nabla_Y \right] - \nabla_{[X,Y]}$. We will continue to use this notation in general, so if $E \in \Gamma(\pi^* \calA)$ and $F \in \Gamma(\pi^* TM)$ are sections of the pullback bundle, then we will also define
$$h E|_a = h_a E_a, \qquad \vl F|_a = \vl_a F_a, \qquad a \in \calA.$$
%

For the remainder of this section, we will only consider the case when $\mathcal{A} = T^*M$ is the cotangent bundle, and so consider horizontal lifts of vector fields and vertical lifts of one-forms.

\subsection{Symplectic complements and adjoint connections} \label{sec:HamCon}
Let $\pi$ be the canonical projection of the cotangent bundle $\pi:T^*M \to M$. Define \emph{the Liouville one-form} on $T^*M$ by
$$\vartheta(w) = p(\pi_* w), \qquad \text{ for $w \in T_p(T^*M)$, $p \in T^*M$},$$
and let $\omega = -d\vartheta$ be \emph{the canonical symplectic form} on $T^*M$. Let $\nabla$ be an affine connection on $TM$ with torsion $T(X,Y) = \nabla_XY - \nabla_Y X - [X,Y]$, $X,Y \in \Gamma(TM)$. We denote the induced connection on $T^*M$ and all other tensor bundles by the same symbol. Define horizontal lifts from $M$ to $T^*M$ with respect to $\nabla$. From the definition of $\vartheta$, for any $p \in T^*M$ 
\begin{equation} \label{omegaHV} \omega(h_pv, h_pw) = - pT(v,w), \quad \omega(h_pv, \vl_p \alpha) = \alpha(v), \quad \omega(\vl_p \alpha, \vl_p \beta) =0,\end{equation}
for any $v,w \in TM$ and $\alpha, \beta \in T^*M$.

Following the terminology of \cite{Dri92}, we define \emph{the adjoint connection} $\hat \nabla$ of $\nabla$ by
$$\hat \nabla_X Y = \nabla_X Y - T(X, Y) =: \nabla_X Y - T_X Y.$$
Let $\hat \calH = \calH^{\hat \nabla}$ be the corresponding Ehresmann connection. The corresponding horizontal lift $\hat h$ is given by
$$\hat h X = hX - \vl T_X^*,$$
with $\vl T^*_X |_p = \vl_p T^*_X p = \vl_p pT(X, \, \cdot \,)$. Using \eqref{omegaHV} we have the following relations regarding symplectic complements:
$$\calV^{\angle} = \calV, \qquad \calH^{\angle} = \hat \calH.$$
Hence, $\calV$ is always a Lagrangian subbundle, while $\calH$ is only Lagrangian if $\nabla$ is torsion free.

\subsection{Connections and Hamiltonian functions}
Let $f: T^*M \to \mathbb{R}$ be a smooth Hamiltonian function. To every such function, we define $\nabla f \in \Gamma(\pi^* T^*M)$ and $f' \in \Gamma(\pi^* TM)$ by the following relations
$$\alpha (f' |_p) = df(\vl_p \alpha), \qquad \nabla f|_p( v) = df(h_p v), \qquad p, \alpha \in T_x^* M, v \in T_x M, x \in M.$$
Since vertical lifts does not depend on any connection, the definition of $f'$ is also independent of this choice. We note the following observations which follow from \eqref{BracketsHV} and~\eqref{omegaHV}.
\begin{lemma}
Let $f,H \in C^\infty(T^*M)$ be two Hamiltonian functions
\begin{enumerate}[\rm (a)]
\item If $\vec f$ is the Hamiltonian vector field of $f$, then
$$\vec{f} = \hat h f' - \vl \nabla f = h f' - \vl \hat \nabla f .$$
\item The Poisson bracket is given by $\{ f, H\} = -\nabla H(f') + \hat \nabla f(H').$
\item For a vector field $X \in \Gamma(TM)$, let $H_X: T^*M \to \mathbb{R}$ denote the function $H_X(p) = p(X_{\pi(p)})$. Then
$$H'_X|_p =  X_{\pi(p)}, \qquad \nabla H_X|_p(v) = p(\nabla_v X), \qquad v\in T_x M, p \in T_x^*M, x \in M.$$
\end{enumerate}
\end{lemma}

We continue with the following definition.
\begin{definition}
We say that a Hamiltonian function $H$ is parallel with respect to~$\nabla$ if the following equivalent conditions are satisfied.
\begin{enumerate}[\rm (i)]
\item $\nabla H = 0$.
\item $\vec{H}$ takes values in $\hat \calH$.
\item $H$ is constant along any curve tangent to $\calH$.
\end{enumerate}
\end{definition}

For the next result, we introduce the symmetric bilinear map $H'' = \langle \, \cdot \, , \, \cdot \, \rangle_{H''} \in \Gamma(\pi^* \Sym^2 TM)$, defined as
$$\langle \lambda_1, \lambda_2 \rangle_{H''} = \frac{d^2}{ds_1 ds_2} H(p + s_1 \lambda_1 + s_2 \lambda_2)  |_{s_2=0} |_{s_1=0}, \qquad \lambda_1, \lambda_2 \in T^*_{\pi(p)}M.$$
We also define $\sharp^{H''}: \pi^*T^*M \to \pi^*TM$  by $\langle \lambda_1, \lambda_2 \rangle_{H''} = \lambda_1(\sharp^{H''} \lambda_2)$.
\begin{lemma} \label{lemma:HamBracket}
Assume that $H$ is parallel with respect to $\nabla$. Let $\hat h$ denote the horizontal lift with respect to $\hat \nabla$.
\begin{enumerate}[\rm (a)]
\item  Let $\gamma(t)$ be a curve in $M$ and let $p(t)$, $\lambda_1(t)$ and $\lambda_2(t)$ be $\nabla$-parallel forms along~$\gamma(t)$. Then
$$\frac{d}{dt} \lambda_1(t) (H'_{p(t)}) = 0, \qquad \frac{d}{dt} \langle \lambda_1(t), \lambda_2(t) \rangle_{H''_{p(t)}} =0.$$
\item Let $\hat R = R^{\hat \nabla}$ be the curvature of $\hat\nabla$. Let $X \in \Gamma(TM)$ and $\beta \in \Gamma(T^*M)$ be arbitrary. We then have the following brackets with respect to the Hamiltonian vector field;
\begin{eqnarray} \label{HhX}
{[\vec{H}, \hat{h} X]} |_p 
&= & \hat{h}  \hat{\nabla}_{H'_p} X  +\hat{h} \sharp^{H''_p} T_{X_{\pi(p)}}^* p - \vl \hat R(H'_p, X_{\pi(p)}),  \\ \label{HvlB}
{[\vec{H}, \vl \beta]} |_p & =& - \hat{h}_p \sharp^{H''_p} \beta + \vl_p \hat{\nabla}_{H'_p} \beta .
\end{eqnarray}
\end{enumerate}
\end{lemma}
\begin{proof}
\begin{enumerate}[\rm (a)]
\item The result follows from the fact that $p(t) + s_1 \lambda_1(t)$ and $p(t) + s_1 \lambda_1(t) + s_2 \lambda_2(t)$ are parallel forms along $\gamma(t)$ for any constants $s_1, s_2$.
\item Let $X_1, \dots, X_n$ be any local basis with corresponding coframe $\alpha_1, \dots, \alpha_n$. If we define $\phi_j (p) =\alpha_j(H'_p)$, then by the result in (a),
$$(h_p w) \phi_j= ( \nabla_w \alpha_j) (H'_p),\qquad (\vl_p \lambda) \phi_j = \langle \alpha_j, \lambda \rangle_{H''_p},$$
for any $w \in T_xM$, $\lambda, p \in T_xM$, $x \in M$. It follows that if we write $\vec{H} = \hat{h} H' =\sum_{j=1}^n \phi_j \hat{h} X_j$, then
$$[\vec{H}, \hat{h} X] = \sum_{j=1}^n \phi_j (\hat{h}  [X_j, X] - \vl \hat R(X_j, X)) - \sum_{j=1}^n ((hX - \vl T_X^*) \phi_j) \hat{h}X_j.$$
At any point $x \in M$, we can choose a frame $X_1, \dots, X_n$ as $\nabla$-parallel at $x$, which means that $\alpha_1, \dots, \alpha_n$ are parallel at $x$ as well. As a consequence, we have at $x \in M$ and for $p \in T_x M$,
\begin{align*}
[\vec{H}, \hat{h} X]|_p &=  \hat{h}  \nabla_{H'_p} X - \hat{h}T( H'_p, X_{\pi(p)}) - \vl \hat R(H'_p, X_{\pi(p)})  +\hat{h} \sharp^{H''_p} T_{X_{\pi(p)}}^* p \\
&=  \hat{h}  \hat{\nabla}_{H'_p} X  - \vl \hat R(H'_p, X_{\pi(p)}) +\hat{h} \sharp^{H''_p} T_{X_{\pi(p)}}^* p,
\end{align*}
and since the choice of $x$ was arbitrary, we have a formula for $[\vec{H}, \hat{h} X]$. The proof of \eqref{HvlB} is similar. \qedhere
\end{enumerate}
\end{proof}

%% file: 2SubR.tex
\section{Sub-Riemannian geometry} \label{sec:SRgeometry}
\subsection{Sub-Riemannian structures and minimizers} \label{sec:DefSR} We review  basic definitions and results relating to sub-Riemannian geometry. For details, see e.g.~\cite{Mon02,ABB19}. A sub-Riemannian manifold is a triple $(M, \calE, g)$ where $M$ is a connected manifold, $\calE$~is a subbundle of the tangent bundle $TM$ and $g$ is a metric tensor on $\calE$. The pair $(\calE, g)$ is called a \emph{sub-Riemannian structure}. A sub-Riemannian structure can equivalently be described in the following ways.
\begin{enumerate}[$\bullet$]
\item A vector bundle map $\sharp: T^*M \to TM$ with kernel of constant rank that also satisfies $p_2(\sharp p_1) = p_1(\sharp p_2)$  and $p_1(\sharp p_1) \geq 0$ for any $p_1, p_2 \in TM$.
\item A bilinear symmetric two-tensor $g^*$ on $T^*M$, the sub-Riemannian cometric, which is positive semi-definite and degenerate along a subbundle.
\end{enumerate}
The three definitions are related in the following way,
$$\langle p_1, p_2 \rangle_{g^*} = p_1(\sharp p_2) = \langle \sharp p_1, \sharp p_2 \rangle_{g},$$
with $\sharp$ having image $\calE$ and kernel
$$\Ann(\calE) = \{ p \in T^*_xM, x \in M \, : \, p(v) =0, v \in D_x \}.$$

The bundle $\calE$ is referred to as \emph{the horizontal bundle}. Write $\underline{\calE}^1 = \Gamma(\calE)$ and define iteratively $\underline{\calE}^{k+1} = \underline{\calE}^k + [\underline{\calE}^1, \underline{\calE}^k]$. For each $x \in M$, we write $\calE^k_x = \{ X|_x \, : \, X\in \underline{\calE}^k\} \subseteq T_x M$. We say that $\calE$ is \emph{bracket-generating} if for every $x\in M$, there is an integer $s$ such that $\calE^s_x = T_xM$. We call the minimal integer $s = s(x)$ satisfying this property \emph{the step} of $\calE$ at $x$. If $\mathfrak{G}_k(x) = \rank \calE^k_x$ then $\mathfrak{G}(x) = ( \mathfrak{G}_1(x), \dots, \mathfrak{G}_s(x))$ is called \emph{the growth vector} of $\calE$ at $x$. A point $x$ is called a \emph{regular point of $\calE$} is $\mathfrak{G}$ is locally constant at $x$. If $x$ is not regular, it is called a \emph{singular point of $\calE$}. We note that regular points form an open and dense set in $M$ \cite[Sect. 2.1.2, p. 21]{Jea14}. The subbundle $\calE$ is called \emph{equiregular} if all points in $M$ are regular. In this case, we have that $\{\calE_x^k\}_{x \in M}$ defines a vector bundle $\calE^k$ for each $k \geq 1$ and consequently $\underline{\calE}^k = \Gamma(\calE^k)$. 
For the rest of the paper, we will assume that $\calE$ is bracket-generating, but not necessarily equiregular.

A continuous curve $\gamma:[0,t_1] \to M$ is called \emph{horizontal} if it is absolutely continuous and satisfies $\dot \gamma \in \calE_{\gamma(t)}$ for almost every $t$. For such a curve, we can define its length by
$$\mathrm{Length}(\gamma) = \int_0^{t_1} | \dot \gamma(t) |_g^{1/2} \, dt, \qquad |\dot \gamma(t) | := \langle \dot \gamma(t), \dot \gamma(t) \rangle_g^{1/2} .$$
Since $\calE$ is bracket generating, any pair of points can be connected by a horizontal curve. Furthermore, if we define a distance $d_g$ by letting $d_g(x,y)$ denote the infimum of the length of all horizontal curves connecting $x$ and $y$, then the topology of $d_g$ equals the manifold topology.

On a sub-Riemannian manifold $(M, \calE, g)$, we define \emph{the sub-Riemannian Hamiltonian} $H: T^*M \to \mathbb{R}$ by
$$H(p) = \frac{1}{2} |p|_{g^*}^2 = \frac{1}{2} \langle p,p \rangle_{g^*}.$$
Let $\vec{H}$ be the corresponding Hamiltonian vector field. If $\pi:T^*M \to M$ is the canonical projection, we say that
$$\gamma(t) = \exp(tp) : = \pi(e^{t\vec{H}}(p)), \qquad p \in T^*M,$$
is \emph{the normal geodesic} with initial covector $p \in T^*M$ and that $\lambda(t) = e^{t\vec{H}}(p)$ is its \emph{extremal}. Normal geodesics are always local length minimizers. However, it is not in general true that all minimizers are of this type. Consider the Hilbert manifold $\mathrm{AC}^2(x)$ of horizontal curves defined on $[0,t_1]$ with initial value $x$ and with square integrable derivative. Define $\mathrm{End} :\mathrm{AC}^2(x) \to M$ as the endpoint map $\gamma \mapsto \gamma(t_1)$, which is a smooth map of manifolds. A curve $\gamma$ is then called \emph{abnormal} if it is a singular point of $\End$, i.e if
$$\End_{*,\gamma} : T_\gamma \mathrm{AC}^2(x) \to T_{\gamma(t_1)}M,$$
is not surjective. Abnormal curves do not depend on the metric $g$, only subbundle $\calE$. Any length minimizer of $d_g$ is either a normal geodesic or an abnormal curve, but abnormal curves need not be minimizers in general, even locally. Furthermore, a curve can both be a normal geodesic and an abnormal curve.

A normal geodesic $\gamma:[0,t_1] \to M$ is called \emph{strictly normal} if it is not abnormal and \emph{strongly normal} if $\gamma|_{[0,t_0]}$ is not abnormal for any $0 \leq t_0\leq t_1$. A minimizer is called \emph{strictly abnormal} if it is abnormal and not normal.

\subsection{Compatible connections and length minimizers} \label{sec:minimizers}
Let $(M, \calE, g)$ be a sub-Riemannian manifold and let $\nabla$ be a connection on $TM$. A connection $\nabla$ is said to be \emph{compatible} with the sub-Riemannian structure if the following equivalent conditions are satisfied.
\begin{enumerate}[$\bullet$]
\item Relative to $(\calE, g)$: $\calE$ is preserved under parallel transport and for $X_1, X_2 \in \Gamma(\calE)$ and $Y \in \Gamma(TM)$, we have that
$$Y \langle X_1, X_2 \rangle_g = \langle \nabla_Y X_1, X_2 \rangle_g + \langle X_1 , \nabla_Y X_2 \rangle_g.$$
\item Relative to $\sharp$: We have the commutation relation $\nabla \sharp = \sharp \nabla$.
\item Relative to $g^*$: $\nabla g^* = 0$.
\end{enumerate}
Any compatible connection to a sub-Riemannian structure $(\calE, g)$ will have torsion whenever $\calE$ is a proper subbundle and bracket-generating  \cite[Proposition~3.3]{GroTh18}. 

We have the following relations between normal geodesic and compatible connections, found in \cite[Proposition~2.1]{GoGr17}. Recall the definition of adjoint connection $\hat \nabla$ of $\nabla$ in Section~\ref{sec:HamCon}. We will denote their corresponding covariant derivatives along curves by respectively $\hat D$ and $D$.
\begin{proposition} \label{pro:NormalGeodesics}
Let $\nabla$ be a compatible connection with adjoint $\hat \nabla$. A curve $\gamma: [0,t_1] \to M$ is a normal geodesic if and only if there is a form $\lambda(t)$ along $\gamma(t)$ satisfying
$$\hat D_t \lambda = 0, \qquad \sharp \lambda(t) = \dot \gamma(t),$$
with $\hat D_t= (\gamma^* \hat \nabla)_{\frac{\partial}{\partial t}}$ being the covariant derivative along $\gamma$. Furthermore, $\lambda(t)$ is an extremal of $\gamma(t)$.
\end{proposition}
We show a similar relation for abnormal curves.
\begin{proposition} \label{prop:AbConnection}
Let $\nabla$ be any affine connection on $TM$ such that $\calE$ is preserved under parallel transport. Write its adjoint as $\hat \nabla$. Let $\gamma$ be a horizontal curve with square integrable derivative. Then $\gamma$ is an abnormal curve if and only is a non-zero one-form~$\lambda$ along a curve~$\gamma$, satisfying almost everywhere
$$\hat D_t \lambda =0, \qquad \sharp \lambda(t) =0,$$
with $\hat D_t = (\gamma^* \hat \nabla)_{\frac{\partial}{\partial t}}$ being the covariant derivative along $\gamma$. In particular, holds true if $\nabla$ is compatible with $(\calE, g)$. 
\end{proposition}
For the proof, we will need the following result from \cite{Hsu92} and \cite[Theorem~5.3]{Mon02}.
\begin{lemma}
Let $\omega$ be the canonical symplectic form on $T^*M$. Then $\gamma$ is abnormal if any only if there exists there is a non-zero one-form $\lambda(t)$ along a curve $\gamma(t)$, with square integrable derivative and with values in $\Ann(\calE)$ almost everywhere, satisfying
$$\omega\left(\dot \lambda(t), T_{\lambda(t)} \Ann(\calE)\right) = 0.$$
\end{lemma}

\begin{proof}[Proof of Proposition~\ref{prop:AbConnection}]
We note first that since $\nabla$ preserves $\calE$, it also preserves $\Ann(\calE)$. As a consequence, if $p \in \Ann(\calE)_x$, then
$$T_p \Ann(\calE) = \spn \{ h_p v, \vl_p \alpha \, : \, v \in T_xM, \alpha \in \Ann(\calE)_x \},$$
where $h$ is the horizontal lift with respect to $\calH = \calH^\nabla$.
By \eqref{omegaHV} the symplectic complement of $T_p \Ann(\calE)$ is
$$(T_p \Ann(\calE))^\angle = \{ \hat h_p v\, : \, v \in \calE_x \}.$$
Hence, $\dot \lambda(t)$ has to be a $\hat \nabla$-parallel form along a horizontal curve. If $\gamma$ has square integrable derivative, then so will $\lambda(t)$ as it is a horizontal lift.
\end{proof}
The following can now be easily deduced from the two descriptions.
\begin{corollary}
A normal geodesic $\gamma(t)$ is strictly normal if and only if it has a unique extremal $\lambda(t)$.
\end{corollary}

\subsection{Brackets of the the sub-Riemannian Hamitonian} \label{sec:SRHamBracket}
Consider $M$ with a sub-Riemannian structure $(\calE,g)$. Let $\nabla$ be a connection compatible with the sub-Riemannian structure and write $\hat \nabla$ for its adoint. If the corresponding Hamiltonian function $H$ is defined by by $H(p) = \frac{1}{2} | p|^2_{g^*}$, then in the notation of Section~\ref{sec:HamCon}
$$H' = \sharp, \qquad H'' = g^*.$$
As a consequence, Lemma~\ref{lemma:HamBracket} gives the following expressions for Lie brackets with the Hamiltonian vector field $\vec{H}$,
\begin{eqnarray} \label{SRbracketHh}
{[\vec{H}, \hat{h} X]} |_p 
&= & \hat{h}  \hat{\nabla}_{\sharp p} X  +\hat{h} \sharp T_X^* p - \vl \hat R(\sharp p, X),  \\ \label{SRbracketHvl}
{[\vec{H}, \vl \beta]} |_p & =& - \hat{h}_p \sharp \beta + \vl_p \hat{\nabla}_{\sharp p} \beta.
\end{eqnarray}

%% file: 3JCurves.tex
\section{Jacobi curves and canonical connections} \label{sec:JacCurves}
In this section we will review ideas of Jacobi fields, Jacobi curves and connections of sub-Riemannian manifolds introduced in \cite{ZeLi09,ABR18,BaRi16}.

\subsection{Jacobi fields and the Jacobi curve} Let $(M, \calE, g)$ be a sub-Riemannian manifold with sub-Riemannian Hamiltonian $H(p) = \frac{1}{2} | p|^2_{g^*}$, $p \in T^*M$. Let $\pi: T^*M \to M$ denote the natural projection and again write $\exp(tp) = \pi(e^{t\vec{H}}(p))$. Just like in the Riemannian case, we can consider Jacobi fields as a variation of (normal) geodesics. If $\eta(s)$ is a curve in $T^*M$, we can consider a Jacobi field as a result of computing
$$V(t) =  \frac{\partial}{\partial s} \exp(t\eta(s))|_{s=0} = \frac{\partial}{\partial s} \pi_*(e^{t\vec{H}}(\lambda(s))) = \pi_* e^{t\vec{H}}_* \partial_s \eta(0).$$
All Jacobi fields can hence be written as $V(t)= \pi_* \tilde V(t)$ with $\tilde V(t) = e^{t\vec{h}}_* u$ for some constant $u \in T(T_p^*M)$. In other words, they are solutions of the equation
\begin{equation} \label{JacobiEq}\frac{d}{dt} e^{-t\vec{H}}_* \tilde V = 0.\end{equation}
In what follows, for a vector field $\tilde X$ on $T^* M$, write $\Ad(e^{-t\vec{H}} \tilde X)|_p := e^{-t \vec{H}}_* \tilde X_{e^{t\vec{H}}(p)}$. We note that
$$\frac{d}{dt}\Ad(e^{-t\vec{H}})\tilde X = \Ad(e^{-t\vec{H}})[\vec{H},\tilde X].$$
\begin{example}[Riemannian manifolds]
For the case of Riemannian manifolds $\mathcal{E} = TM$, let $\nabla$ denote the Levi-Civita connection of $g$ which satisfies $\nabla H =0$. Since it is also torsion free, we know that $\hat \nabla = \nabla$. Let $\tilde V(t)$ be a vector field along the extremal $\lambda(t) = e^{t\vec{H}}(p)$ with projection $V(t)$ along~$\gamma(t)$. Then at least locally, there exist a vector field $X$ and a one-form $\alpha$ such that
$$\tilde V(t) = h_{\lambda(t)} X_{\gamma(t)} + \vl_{\lambda(t)} \alpha_{\gamma(t)}.$$
Using the formulas of \eqref{SRbracketHh} and \eqref{SRbracketHvl} in the Riemannian case,
\begin{align*}
\frac{d}{dt} e_*^{-t\vec{H}} \tilde V(t) & =  \frac{d}{dt} \Ad(e^{-t\vec{H}})(hX + \vl \alpha) |_{\lambda(0)}  =  \Ad(e^{-t\vec{H}})([\vec{H},hX + \vl \alpha]) |_{\lambda(0)}  \\
& = e^{-t\vec{H}}_*\left( h_{\lambda(t)} \nabla_{\dot \gamma(t)} X - \vl_{\lambda(t)} R(\dot \gamma, X) - h_{\lambda(t)} \sharp \alpha + \vl_{\lambda(t)} \nabla_{\dot \gamma(t)} \alpha \right), \\
& = e^{-t\vec{H}}_* \left( h_{\lambda(t)} D_t V - \vl_{\lambda(t)} R(\dot \gamma, V) - h_{\lambda(t)} \sharp \alpha + \vl_{\lambda(t)} \nabla_{\dot \gamma} \alpha \right).
\end{align*}
Hence, if $\frac{d}{dt} e_*^{-t\vec{H}} \tilde V(t) = 0$ then $\sharp \alpha|_{\gamma(t)} = D_t V$ and
$$\sharp \nabla_{\dot \gamma} \alpha|_{\dot \gamma(t)} = D_t D_t V(t) = \sharp  R(\gamma(t), V(t)) \lambda(t) = R(\gamma(t), V(t)) \dot \gamma(t). $$
As a result, we obtain the classical Jacobi equation $D_t^2 V - R(\dot \gamma, V) \dot \gamma =0$.
\end{example}

A geodesic $\gamma(t) = \exp(tp)$ is said to have the \emph{conjugate time} $t_0 >0$ if there is a Jacobi field $V(t)$ with $V(0) = 0$ and $V(t_0)=0$. This definition can be reformulated in terms of Jacobi curves. Write $\mathcal{V} = \ker \pi_*$ for the vertical bundle.
\begin{definition}
For a normal geodesic $\gamma$ with extremal $\lambda(t)$, the subspace $\Lambda(t) = e^{-t\vec{H}}_* \mathcal{V}_{\lambda(t)} \subseteq T_{\lambda(0)} T^*M$ is called the Jacobi curve of $\gamma$.
\end{definition}
By definition, $\Lambda(t)$ is a Lagrangian subspace for any $t\geq 0$. We see that $t_0$ is a conjugate time if and only if $\Lambda(0) \cap \Lambda(t_0) \neq 0$. Hence, we can study conjugate points by understanding the Jacobi curve $\Lambda(t)$. We note that for every $t$, $\Lambda(t)$ is a Lagrangian subspace of the symplectic vector space $T_{\lambda} T^*M$. We will hence give a description of symplectic invariants of curves of such subspaces.

\subsection{Curves in the Lagrangian Grassmannian} \label{sec:LagrangianGrassmannian}
Let $(W, \omega)$ be a symplectic vector space of dimension $2n$. Consider the Grassmann space $L(W)$ of Lagrangian subspaces, that is, $n$-dimensional subspaces that are their own symplectic complements. Let $\Lambda(t)$ be a smooth curve in $L(W)$ with $\Lambda(0) = \Lambda_0$. We can identify the vector $\dot \Lambda(0) \in T_{\Lambda_0}L(W)$ with the map
$$\dot \Lambda(0): \Lambda_0 \to W/\Lambda_0,$$
defined such that if $z(t) \in \Lambda(t)$ is a curve, then $\dot \Lambda(0) : z(0) \mapsto \dot z(0) \mod \Lambda_0$. One can verify that $\dot z(0)$ is independent of the choice of curve $z(t)$, making the map well defined. Using the symplectic form $\omega$, we can identify this map with a quadratic form on~$\Lambda_0$, by
$$\dot{\underline{\Lambda}}(0)(z) = \omega(\dot \Lambda(0) z, z) .$$
In fact, all quadratic forms on $\Lambda_0$ can be represented this way. We introduce the following notions for curves $\Lambda(t)$ in $L(W)$.

We define $\Lambda^{(0)}(t) = \Lambda(t)$ and
\begin{equation} \label{Lambdai}\Lambda^{(i)}(t) = \spn \left\{ \frac{d^j}{dt^j} z(t) \, : \, 0 \leq j \leq i, z(t) \in \Lambda(t) \right\},\end{equation}
and write $k_i(t) = \dim \Lambda^{(i)}(t)$.
\begin{enumerate}[\rm (i)]
\item We say that $\Lambda(t)$ is monotone increasing (resp. decreasing) at $t_0$ if $\dot{\underline{\Lambda}}(t_0)$ is a positive (resp. negative) semi-definite quadratic form. We say that it is strictly monotone increasing (resp. decreasing) if $\dot{\underline{\Lambda}}(t_0)$ is positive (resp. negative) definite.
\item We say that $\Lambda(t)$ is regular at $t_0$ if $\dot{\underline{\Lambda}}(t_0)$ is a non-degenerate quadratic form. It is called regular if it is regular at every point.
\item We define
$$s = \min \{ i \, : \, \Lambda^{(i)}(t_0) = \Lambda^{(j)}(t_0) \text{ for all $i \leq j$}\},$$
as \emph{the step} of $\Lambda(t)$ at~$t_0$.
\item $\Lambda(t)$ is called \emph{ample at $t_0 = 0$} if $\Lambda^{(i)}(t_0) = W$ for some $i \geq 1$.
\item $\Lambda(t)$ is called \emph{equiregular at $t_0 = 0$} if each $k_i(t)$ is constant at $t_0$.
\item $\Lambda(t)$ is called \emph{ample} or \emph{equiregular} if it is respectively ample or equiregular at every $t$.
\end{enumerate}

We remark the following property found in \cite[Chapter~3.1]{ABR18}, see also \cite{ZeLi09}.
\begin{lemma} \label{lemma:FStructure}
If $\Lambda(t)$ is equiregular of step $s$, then for any $j =1, \dots, s-1$, we have
$$d_j = k_j - k_{j-1} \geq d_{j+1} =k_{j+1} - k_j.$$
\end{lemma}
For every ample, equiregular curve $\Lambda(t)$ of step $s$, we introduce associated Young diagram~$\mathbb{Y}$ corresponding to the partition $(d_1, \dots, d_s)$. With the English notation, for a sequence of positive, non-increasing numbers $(d_1, \dots, d_s)$, we write $\mathbb{Y} = \mathbb{Y}(d_1, \dots, d_s)$ for the Young diagram with $s$-columns and with $d_i$-boxes in the $i$-th column.  For us, it will be practical to identify this Young diagram with the set
$$\mathbb{Y} = \{  (a,b) \in \mathbb{N} \times \mathbb{N} \, : \,  1 \leq b \leq s, 1 \leq a \leq d_b\},$$
where each number $(a,b)$ represents the box in the $a$-th row and $b$-th column.
For each $1 \leq a \leq d_1$, write $n_a$ for the maximal value such that $(a, n_a) \in \mathbb{Y}$. In other words, $n_a$ is the length of the $a$-th row.

From all possible values in $\{n_{a}\}_{a=1}^{d_1}$, write them as a descending sequence
$$\mathsf{n}_{1} > \dots > \mathsf{n}_{\mathsf{d}_1},$$
of distinct numbers. \emph{The reduced Young diagram} $\mathsf{Y}$ of $\mathbb{Y}$ is then given by
$$\mathsf{Y} = \{  (\mathsf{a},\mathsf{b}) \in \mathbb{N} \times \mathbb{N} \, : \,  1 \leq \mathsf{a} \leq \mathsf{d_1}, 1 \leq \mathsf{b} \leq \mathsf{n_{a}} \}.$$
In other words, $\mathsf{Y}$ can be considered as a result of collapsing all rows in $\mathbb{Y}$ of equal length into a single row. For any $(a,b) \in \mathbb{Y}$, we define $(\mathsf{a}, \mathsf{b}) = [a,b] \in \mathsf{Y}$ as the unique block satisfying $\mathsf{b} = b$ and $n_a = \mathsf{n_{a}}$. See Figure~\ref{fig:ex} for a concrete example.
\begin{figure}[h]
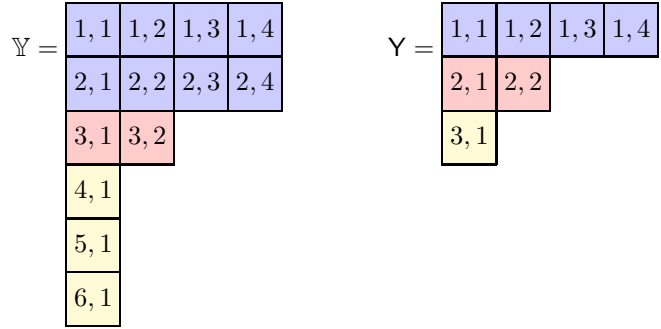
 \centering
$$\mathbb{Y} = \ytableausetup{mathmode, boxsize=2em}
\begin{ytableau}
*(blue!20) 1,1 & *(blue!20) 1,2 & *(blue!20) 1,3 & *(blue!20) 1,4 \\
*(blue!20) 2,1 & *(blue!20) 2,2 & *(blue!20) 2,3 & *(blue!20) 2,4 \\
*(red!20) 3,1 & *(red!20) 3,2 \\
*(yellow!20) 4,1 \\
*(yellow!20) 5,1 \\
*(yellow!20) 6,1
\end{ytableau}  \qquad \qquad 
\mathsf{Y} =
\begin{ytableau}
*(blue!20) 1,1 & *(blue!20) 1,2 & *(blue!20)1,3 &*(blue!20) 1,4 \\
*(red!20) 2,1& *(red!20) 2,2 \\
*(yellow!20) 3,1
\end{ytableau}$$
\caption{An example of a Young diagram with $(d_1, d_2, d_3, d_4) = (6, 3,2,2)$ and its corresponding reduced Young diagram. The rows of $\mathbb{Y}$ and the corresponding row in the reduced Young diagram $\mathsf{Y}$ are given the same color. For any any $(a,b) \in \mathbb{Y}$, $(\mathsf{a},\mathsf{b}) = [a,b]$ is the unique box in column $b$ with the same color as $(a,b)$.}
\label{fig:ex}
\end{figure}

For the geometry of such Jacobi curves, we have the following result found in \cite{ZeLi09}, giving us a canonical complement $\Gamma(t)$ of $\Lambda(t)$ and a determining set of symplectic invariants.
\begin{theorem}[Canonical frame: Original formulation] \label{th:ZLBasis}
Let $\Lambda(t)$ be an ample, equiregular, monotone curve with Young diagram $\mathbb{Y}$ and reduced Young diagram $\mathsf{Y}$. Let $d_1$ and $\mathsf{d}_1$ be the numbers of boxes in their respective first columns. We then have the following decomposition
\begin{equation} \label{Wdec} W= \Gamma(t) \oplus \Lambda(t),\end{equation}
\begin{equation*} \label{LambdaGamma} \Lambda(t) = \spn \{ E_{a,b}(t) \, : \, (a,b) \in \mathbb{Y} \}, \qquad \Gamma(t) = \spn \{ F_{a,b}(t) \, : \, (a,b) \in \mathbb{Y} \},\end{equation*}
into Lagrangian subspaces, where the basis $\{E_{a,b}, F_{a,b} \}_{ (a,b) \in \mathbb{Y}}$ of $W$ satisfies the following properties.
\begin{enumerate}[\rm (a) ]
\item It is a Darboux basis,
$$\omega(E_{a,b}, E_{i,j}) = \omega(F_{a,b}, F_{i,j}) = \omega(E_{a,b}, F_{i,j}) - \delta_{a,i} \delta_{b,j} = 0.$$
\item For any $1 \leq a \leq d_1$, $1 \leq b \leq n_a-1$, we have
\begin{align*}
\frac{d}{dt} E_{a,b+1} &= E_{a,b},  & \frac{d}{dt} F_{a,b} & =  - F_{a,b+1} + \sum_{(i,j)\in \mathbb{Y}} R^{a,b}_{i,j} E_{i,j}, \\
\frac{d}{dt} E_{a,1} &=  - F_{a,1}, &  \frac{d}{dt} F_{a,n_{a}} & =   \sum_{(i,j)\in \mathbb{Y}} R^{a,n_a}_{i,j} E_{i,j}.
\end{align*}
\item The coefficients $\{ R^{a,b}_{i,j} \, : \, (a,b), (i,j) \in \mathbb{Y}  \}$ satisfies the following: For any $(\mathsf{a},\mathsf{b})\in \mathsf{Y}$, define decomposition $\mathbb{R}^{\mathbb{Y}} = \spn \{ e_{a,b} \, :\, (a,b) \in \mathbb{Y} \} = \oplus_{(\mathsf{a}, \mathsf{b}) \in \mathsf{Y}}\Box(\mathsf{a}, \mathsf{b})$ by
\begin{equation} \label{BoxE} \Box( \mathsf{a}, \mathsf{b}) = \spn \{ e_{a,b} \, : \, [a,b] = (\mathsf{a}, \mathsf{b})\}.\end{equation}
Write the curvature operator $R: \mathbb{R}^\mathbb{Y} \to \mathbb{R}^\mathbb{Y}$ for the linear map with
$$Re_{a,b} = \sum_{i,j \in \mathbb{Y}} R^{a,b}_{i,j} e_{i,j}.$$
Let $R = (R^{(\mathsf{a}, \mathsf{b})}_{(\mathsf{i}, \mathsf{j})})$ denote its decomposition such that $R^{(\mathsf{a}, \mathsf{b})}_{(\mathsf{i}, \mathsf{j})}: \Box( \mathsf{a}, \mathsf{b}) \to \Box(\mathsf{i}, \mathsf{j})$.
Then
$$R_{(\mathsf{i}, \mathsf{j})}^{(\mathsf{a}, \mathsf{b})} = (R_{(\mathsf{a}, \mathsf{b})}^{(\mathsf{i}, \mathsf{j})})^\dagger$$
the transpose of $R^{(\mathsf{i}, \mathsf{j})}_{(\mathsf{a}, \mathsf{b})}$ in the basis $\{ e_{a,b}\}$ and furthermore,
\begin{enumerate}[\rm (i)] 
\item if $\mathsf{b} \neq \mathsf{n_{a}}$, then $R^{(\mathsf{a},\mathsf{b})}_{(\mathsf{a},\mathsf{b}+1)}$ is anti-symmetric;
\item if $\mathsf{a} = \mathsf{i}$ and $\mathsf{j} \neq \{\mathsf{b} -1, \mathsf{b}, \mathsf{b}+1\}$, then $R_{(\mathsf{a}, \mathsf{j})}^{(\mathsf{a},\mathsf{b})} =0$;
\item if $\mathsf{a} < \mathsf{i}$, $\mathsf{j} < \mathsf{n_i}$ and $\mathsf{j} \not \in \{ \mathsf{b}, \mathsf{b}+1\}$, then $R_{(\mathsf{i}, \mathsf{j})}^{(\mathsf{a},\mathsf{b})} =0$;
\item if $\mathsf{a} < \mathsf{i}$, $\mathsf{b} < \mathsf{n_i}-1$, then $R_{(\mathsf{i}, \mathsf{n_i})}^{(\mathsf{a},\mathsf{b})} =0$;
\item if $\mathsf{a} < \mathsf{i}$ and $\mathsf{n_{a}}- \mathsf{n_{i}} \geq \mathsf{b} + \mathsf{j} $, then  $R^{(\mathsf{a}, \mathsf{b})}_{ (\mathsf{i}, \mathsf{j})} =0$.
\end{enumerate}
\end{enumerate}
Furthermore, if $\{\tilde E_{a,b}, \tilde F_{a,b} \}$ is another basis satisfying the above conditions, there exist constant orthogonal matrices
$$O^{\mathsf{a}} = (O_{a,i}^\mathsf{a}) , \qquad [a,1] = [i,1] = (\mathsf{a},1), \qquad 1 \leq \mathsf{a} \leq \mathsf{d}_1,$$
such that
$$\tilde E_{a,b} = \sum_{[i,b] = [a,b] = (\mathsf{a}, \mathsf{b})} O_{a,i}^{\mathsf{a}} E_{i,b}, \qquad \tilde F_{a,b} = \sum_{[i,b] = [a,b] = (\mathsf{a}, \mathsf{b})} O_{a,i}^{\mathsf{a}} F_{i,b}.$$
In particular, the decomposition \eqref{Wdec} is independent of choice of basis.

Finally, if $\Lambda(t)$ and $\tilde \Lambda(t)$ are two ample, equiregular, monotone curves with Young diagram $\mathbb{Y}$ and with respective curvature operators $R$ and $\tilde R$, then they differ by a symplectic transformation if and only if $R(t)= \tilde R(t)$ for every $t$.
\end{theorem}

\begin{figure}[h]
\centering
$$
\ytableausetup
   {mathmode, boxsize=0.8cm}
\begin{ytableau}
{}_{( \mathsf{i}, \mathsf{j})}^{( \mathsf{a}, \mathsf{b})}& *(blue!20) 1,1 & *(blue!20) 1,2 & *(blue!20)1,3 &*(blue!20) 1,4 & *(red!20) 2,1& *(red!20) 2,2 & *(yellow!20) 3,1 \\
*(blue!20) 1,1 & \text{\tiny $R^{(1,1)}_{(1,1)}$}   \\
*(blue!20) 1,2 & *(gray!20) \text{\tiny $R^{(1,1)}_{(1,2)}$} & \text{\tiny $R^{(1,2)}_{(1,2)}$} \\
*(blue!20)1,3 & 0_{\rm (ii)} & *(gray!20) \text{\tiny $R^{(1,2)}_{(1,3)}$} & \text{\tiny $R^{(1,3)}_{(1,3)}$} \\
*(blue!20) 1,4 & 0_{\rm (ii)} & 0_{\rm (ii)} & *(gray!20)  \text{\tiny $R^{(1,3)}_{(1,4)}$} & \text{\tiny $R^{(1,4)}_{(1,4)}$}  \\
*(red!20) 2,1 & 0_{\rm (v)} & 0_{\rm (iv)} & 0_{\rm (iv)} & 0_{\rm (iv)} & \text{\tiny $R^{(2,1)}_{(2,1)}$}  \\
*(red!20) 2,2 & \text{\tiny $R^{(1,1)}_{(2,2)}$} & \text{\tiny $R^{(1,2)}_{(2,2)}$} & \text{\tiny $R^{(1,3)}_{(2,2)}$} & \text{\tiny $R^{(1,4)}_{(2,2)}$} & *(gray!20) \text{\tiny $R^{(2,1)}_{(2,2)}$} & \text{\tiny $R^{(2,2)}_{(2,2)}$} \\
*(yellow!20) 3,1 & 0_{\rm (v)} & 0_{\rm (v)} & \text{\tiny $R^{(1,3)}_{(3,1)}$} & \text{\tiny $R^{(1,4)}_{(3,1)}$} & \text{\tiny $R^{(2,1)}_{(3,1)}$} & \text{\tiny $R^{(2,2)}_{(3,1)}$} & \text{\tiny $R^{(3,1)}_{(3,1)}$}
\end{ytableau}$$
\caption{The curvature normalization condition for Jacobi curves with reduced Young diagram $\mathsf{Y} = \mathbb{Y}(3,2,2,1,1)$ as in Figure~\ref{fig:ex}. Only lower triangular values are shown, as the upper triangle is the transpose. Maps that are anti-symmetric by (i) are marked by gray squares. Maps that vanish are marked with a zero and with the condition as a subscript.}
\label{fig:zeros}
\end{figure}


We note that the conditions (ii)-(v) are a reformulation of normalization conditions in \cite{ZeLi09}, chosen so that we have unique connection. See Figure~\ref{fig:zeros} for an example of these curvature conditions applied to a specific Young diagram.

\subsection{Jacobi curves of normal geodesics} \label{sec:NormalGeodesics}
Consider the case when $\Lambda(t)$ is the Jacobi curve of a normal geodesic $\gamma: [0,t_1] \to M$ with extremal $\lambda(t)$, which is a curve in the Lagrangian Grassmanian of $(T_{\lambda(0)} T^*M, \omega_{\lambda(0)})$. Then $\Lambda(t)$ will always be monotone increasing, but only regular if $\calE = TM$. The normal geodesic $\gamma$ is called \emph{ample} and \emph{equiregular}, respectively if the same is true of $\Lambda(t)$. We note the following result of \cite[Chapter~5.2]{ABR18}.
\begin{lemma} \label{lemma:AmpleDense} Let $x_0$ be an arbitrary point.
\begin{enumerate}[\rm (a)]
\item For every $t_1 > 0$, there is at least one covector $p \in T_{x_0}^* M$ such that $\gamma(t) = \exp(tp)$ is ample at every $0 \leq t \leq t_1$.
\item The set of $p \in T_{x_0}^*M$ such that $\gamma(t) = \exp(tp)$ is ample at $t=0$ is Zarinski open.
\item If a normal geodesic $\gamma(t)$  is ample at $t =0$, then it is strongly normal.
\end{enumerate}
\end{lemma}

%
%

\subsection{Conjugate points, Ricci curvature and the Bonnet-Myers theorem} \label{sec:LQRic}
We review some material here regarding optimal control problems taken from \cite{BaRi16}.
\subsubsection{$\LQ$ optimal control problems}
We consider a class of optimal control problems called linear quadratic (LQ) optimal control problems. Let $A$, $B$ and $q$ be given constant matrices of size $n \times n$, $n \times k$ and $n \times n$. For a given $x_0 \in \mathbb{R}^n$, $t \geq 0$ and $u\in L^2([0,t], \mathbb{R}^k)$, we define $x = x_u$ as the solution of
$$\dot x = Ax + Bu.$$
Define cost functional
$$\Phi(u) = \frac{1}{2} \int_0^t \left( |u(s) |^2_{\mathbb{R}^k} + \langle x_u(s), q x_u(s) \rangle_{\mathbb{R}^n} \right) dt.$$
The $\LQ$-optimal control problem with respect to parameters $A$,$B$, $q$, $x_0$, $x_1$ and $t$ is the problem of finding $u$ such that $x_u(0) = x_0$, $x_u(t) = x_1$ and such that $u$ is minimal with respect to $\Phi(u)$.

Locally optimal trajectories $x(s)$ are given as projections of solutions $(p(s), x(s))$ of the Hamiltonian system
$$\begin{pmatrix} \dot p \\ \dot x \end{pmatrix} = \begin{pmatrix} - A^* & - q^* \\ B B^* & A \end{pmatrix} \begin{pmatrix} p \\ x \end{pmatrix} .$$
\begin{definition}
For $A$, $B$ and $q$ fixed, we say that $t$ is a conjugate time of the $\LQ$-optimal control problem if there exists a non-zero solution to the above Hamiltonian system with $x(0)= 0$ and $x(t) =0$.
\end{definition}

\subsubsection{Comparison with LQ-problems} Let $\mathbb{Y}$ be a given Young diagram and identify $\mathbb{R}^n$ with $\mathbb{R}^{\mathbb{Y}}$. Let $\{ e_{a,b}, (a,b) \in \mathbb{Y}\}$ be the standard basis. Define matrices $A_{\mathbb{Y}}$ and $B_{\mathbb{Y}}$ such that for $1 \leq a \leq d_1$ and $1 \leq b < n_a$, we have
$$ A_{\mathbb{Y}}e_{a,b} = e_{a,b+1}, \qquad A_{\mathbb{Y}} e_{a,n_a} = 0, \qquad B_{\mathbb{Y}}e_{a,b} = \delta_{b,1}e_{a,b}.$$
We let $\LQ(\mathbb{Y},q)$ denote the $\LQ$-optimal control problem relative to the matrices $A_{\mathbb{Y}}$,$B_{\mathbb{Y}}$ and $q$ write $t_c(\mathbb{Y},q)$ for the first conjugate time of the system.
\begin{lemma}[Sub-Riemannian comparison theorem]
Let $\gamma$ be a strongly normal, equiregular geodesic with Young diagram $\mathbb{Y}$. Let $R: \mathbb{R}^\mathbb{Y} \to \mathbb{R}^\mathbb{Y}$ denote the curvature operator of its Jacobi curve.
Let $0 < t_c \leq \infty$ be the first conjugate time of $\gamma$. If $R \geq q$ (resp. $R \leq q$) for all $t \geq 0$, then $t_c \leq t_c(\mathbb{Y}, q)$ (resp. $t_c \geq t_c(\mathbb{Y}, q)$).

In particular, if $R \leq 0$, then $t_c = \infty$.
\end{lemma}

\subsubsection{The general Bonnet-Myers theorem}
Let $\gamma(t)$ be an ample, equiregular geodesic with Young diagram $\mathbb{Y}$, reduced Young diagram $\mathsf{Y}$ and curvature $R(t) : \mathbb{R}^{\mathbb{Y}} \to \mathbb{R}^{\mathbb{Y}}$. Give $\Box(\mathsf{a}, \mathsf{b}) = \spn\{ e_{a,b} \, : [a,b] = (\mathsf{a}, \mathsf{b})\}$ an inner product by making the given basis orthonormal and define $\Ric(\mathsf{a},\mathsf{b})(t) = \tr R_{(\mathsf{a},\mathsf{b})}^{(\mathsf{a},\mathsf{b})}(t)$. For a given integer $s >0$, we write $\mathbb{Y}^s = \mathbb{Y}(1,\dots, 1)$ for the Young diagram with one row and $s$ columns. We define $q_{k_1, \dots, k_s} = \mathrm{diag}\{k_1, \dots, k_s\}$ as the diagonal matrix with coefficients $k_1, \dots, k_s$.

\begin{theorem}[The Bonnet-Myers theorem] \label{th:BMTheorem}
Let $M$ be a complete sub-Riemannian manifold such that from a given point, a dense set of $M$ can be reached by a minimizing geodesic with Young diagram $\mathbb{Y} = \mathbb{Y}_1(d_1,\dots, d_s)$ and reduced Young diagram $\mathsf{Y} = \mathbb{Y}(\mathsf{d}_1, \dots, \mathsf{d}_s)$. For a given $1\leq \mathsf{a} \leq \mathsf{d}_1$, assume that for any unit speed geodesic $\gamma(t)$ with Young diagram $\mathbb{Y}$, we have $\mathsf{r_a} = \rank \Box(\mathsf{a}, 1) > \delta_{\mathsf{a}, \mathsf{d}_1}$ and that there exists constants $k_1, \dots, k_{\mathsf{n_a}}$ with
$$\frac{1}{\mathsf{r_{a} - \delta_{\mathsf{a}, \mathsf{d}_1}}} \Ric(\mathsf{a}, \mathsf{b}) \geq k_{\mathsf{b}} ,  \qquad \mathsf{b} =1, \dots, \mathsf{n_{a}},$$
If the polynomial
$$\pi_{k_1, \dots, k_{\mathsf{n_a}}}(x) = x^{2\mathsf{n_a}} - \sum_{\mathsf{b} =0}^{\mathsf{n_a}-1} (-1)^{\mathsf{n_a}-\mathsf{b}} k_{\mathsf{n_a} - \mathsf{b}} x^{2\mathsf{b}},$$
has at least one simple purely imaginary root, then the manifold is compact with $\diam M \leq t_c(\mathbb{Y}^{\mathsf{n_a}}, q_{k_1, \dots, k_{\mathsf{n_a}}}) < \infty$. Moreover, its fundamental group is finite.
\end{theorem}
We note that for the lower step cases, for $k > 0$, we have
$$t_c(\mathbb{Y}^1, q_{k}) = \frac{\pi}{\sqrt{k}}, \qquad \text{and} \qquad t_c(\mathbb{Y}^2, q_{k,0}) = \frac{2\pi}{\sqrt{k}}.$$

%% file: 5JHam.tex
\section{Jacobi curves and affine connections} \label{sec:JacCon}
We want to give a more explicit description of the canonical complement $\Gamma(t)$ of the Jacobi curve $\Lambda(t)$ of a sub-Riemannian geodesic using compatible affine connections. Throughout this section, we let $\pi: T^*M \to M$ denote the canonical projection and write $H: T^*M \to \mathbb{R}$ be the sub-Riemannian Hamiltonian. We consider a given ample and equiregular geodesic $\gamma(t) = \exp(tp)$ with extremal $\lambda(t) = e^{t\vec{H}}(p)$ for some $p \in T^*M$. Let~$\nabla$ be a choice of compatible connection with adjoint~$\hat \nabla$. We let $D_t$ and $\hat D_t$ denote the corresponding covariant derivatives along $\gamma$ of respectively $\nabla$ and $\hat \nabla$. 

\subsection{Complements to the Jacobi curve} \label{sec:Complements} Consider the Jacobi curve $\Lambda(t) = e^{-t\vec{H}}_* \calV_{\lambda(t)}$. By definition, all curves in $\Lambda(t)$ are on the form
\begin{equation} \label{E} E(\alpha)(t) = e^{-t\vec{H}}_* \vl_{\lambda(t)} \alpha(t),\end{equation}
where $\alpha(t)$ is a one-form along the curve $\gamma(t)$. We first want to consider all choices of Lagrangian complements to $\Lambda(t)$.

The following anti-symmetric tensor will be important and will be used in the rest of the paper. Relative to the compatible connection $\nabla$, we define $A = A^\nabla \in \Gamma(\pi^* \wedge^2 T^*M)$ by
\begin{equation} \label{Amap} A_p(v,w) = \frac{1}{2} p(T(v,w)), \qquad p \in T_x^*M, v,w \in T_xM , x \in M.\end{equation}
Let $S \in \Gamma(\gamma^* \Sym^2 T^*M)$ be any symmetric two-tensor along the curve $\gamma(t)$. Relative to $S$ and for vector field $X(t)$ along the curve $\gamma(t)$, define the curve $F_{S}(X)(t)$ in $T_p (T^*M)$ by
\begin{equation} \label{FS} F_S(X)(t) = e^{-t\vec{H}}_* \left( \hat h_{\lambda(t)} X(t) +  \vl_{\lambda(t)} A_{\lambda(t)}(X(t)) - \vl_{\lambda(t)} S(t)(X(t)) \right),\end{equation}
with $A_{\lambda(t)}(X(t)) = A_{\lambda(t)}(X(t), \, \cdot \,)$ and $S(t)(X(t)) = S(t)(X(t), \, \cdot \,)$ being one-forms along $\gamma(t)$. 
Define
$$\Gamma_S(t) = \{ F_S(X)(t) \, : \, X \in \Gamma(\gamma^* TM) \}.$$ By \eqref{omegaHV}, we have that $\Gamma_S(t)$ is a Lagrangian complement to $\Lambda(t)$ and conversely all such complements correspond uniquely to a choice of symmetric map~$S$ along the curve. We note that for the case $S = 0$, 
$$v \mapsto  \hat h_{\lambda(t)} v +  \vl_{\lambda(t)} A_{\lambda(t)}(v) =\hat h_{\lambda(t)} v + \frac{1}{2} \vl_{\lambda(t)} T_v^* \lambda(t).$$ is the horizontal lift with respect to the torsion-free connection $\frac{1}{2} (\nabla + \hat \nabla)$.

Corresponding to a choice of $S$, introduce a covariant derivative $D_t^S$ of vector fields along $\gamma$,
$$D_t^S X= \hat D_t X  + A^\sharp_\lambda(X) + S^\sharp(X), \qquad A^\sharp(X) =  \sharp A(X, \, \cdot \,), \quad S^\sharp(X) = \sharp S(X, \, \cdot \,), $$
which also induces a covariant derivative of forms along $\gamma$,
$$D_t^S \alpha = \hat D_t \alpha  + A_\lambda(\sharp \alpha) -  S(\sharp \alpha),  \qquad A(X) = A(X, \, \cdot\,), \quad S(X) = S(X, \, \cdot \,).$$
We emphasize that these operators are only defined along the geodesic.
\begin{lemma} \label{lemma:EFGeneral}
For the curves $E(\alpha)$ as in \eqref{E} and $F_S(X)$ as in \eqref{FS}, we have
\begin{eqnarray*}
\frac{d}{dt}  E(\alpha) & = & E\left(D_t^S \alpha \right) -F_S(\sharp \alpha), \\
\frac{d}{dt} F_S(X) & = & F_S(D_t^S X )  + E( \mathfrak{R}^S_\gamma(X) ), \end{eqnarray*}
where $\mathfrak{R}^S_\gamma(t) \in \Sym^2 T_{\gamma(t)}M$ is the symmetric map determined by
\begin{align} \label{frakRS} \mathfrak{R}^S_\gamma(X,X) & = \lambda R(\dot \gamma, X)X + \lambda (\nabla_X T)(\dot \gamma, X) \\ \nonumber
&\qquad + \left|S^\sharp(X) + A^\sharp_\lambda(X) \right|^2_{g} -(\hat D_t S)(X,X) .\end{align}
and we write $\mathfrak{R}_\gamma^S(X) = \mathfrak{R}_\gamma^S(X, \, \cdot \,)$.
\end{lemma}
\begin{proof}
If $V(t)$ is a vector field along the extremal $\lambda(t)$ and if $Y$ is a vector field on $T^*M$ such that for $s \in (-\ve,\ve)$,
$$Y_{\lambda(t +s)} = Y_{e^{(t+s) \vec{H}}(p)} = V(t+s),$$
then
$$e_*^{t\vec{H}} \frac{d}{dt} e_*^{-t\vec{H}} V(t) = e^{t\vec{H}}_* \frac{d}{ds} e^{- s \vec{H}}_* e^{- t \vec{H}}_* Y_{e^{s \vec{H}} \circ e^{t\vec{H}}(p)} |_{s=0} = [\vec{H},Y]|_{e^{t\vec{H}}(p)}.$$
Using results of Section~\ref{sec:SRHamBracket}, we have
\begin{eqnarray*}
e_*^{t\vec{H}} \frac{d}{dt} E(\alpha) & = & \vl \hat D_t \alpha  - \hat h \sharp \alpha = e_*^{t\vec{H}} E\left(\hat D_t\alpha + A(\sharp \alpha )- S(\sharp \alpha)\right) - e_*^{t\vec{H}} F_S(\sharp \alpha) \\
&  = & e_*^{t\vec{H}} E\left( D_t^S\alpha \right) - e_*^{t\vec{H}} F_S(\sharp \alpha), \\
e_*^{t\vec{H}} \frac{d}{dt}  F_S(X) & = & \hat h \hat D_t X + \hat h \sharp T_X^* \lambda - \vl \hat R(\dot \gamma, X) \lambda \\
& & +  \vl A_{\lambda}(\hat D_t X) + \frac{1}{2} \vl (\hat \nabla_{\dot \gamma} T)_{X}^* \lambda - \hat h A^\sharp_\lambda(X) \\
& & - \vl (\hat D_t S)(X)  - \vl S (\hat D_t X ) +  \hat h S^\sharp (X)  \\
& = &  e_*^{t\vec{H}} F_S( \hat D_t^S X ) -  \vl A_\lambda(A^\sharp_\lambda(X)) +  \vl S(A_\lambda^\sharp(X))- \vl \hat R(\dot \gamma, X)\lambda \\
& &  + \frac{1}{2} \vl (\hat \nabla_{\dot \gamma} T)_{X}^* \lambda - \vl (\hat D_t S)(X)  -  \vl A_\lambda(S^\sharp(X)) + \vl S(S^\sharp(X)). \end{eqnarray*}
Hence we will have that $\frac{d}{dt} F_S(X) = F_S(D_t^S X )  + E( \mathfrak{R}^S_\gamma(X) )$ with
\begin{align*}
\mathfrak{R}^S(X,Y) & = -(\hat R(\dot \gamma, X) \lambda)( Y) -  A_\lambda(A^\sharp_\lambda(X),Y) +  S(A_\lambda^\sharp(X), Y) \\
&   \qquad + \frac{1}{2}  \lambda (\hat \nabla_{\dot \gamma} T)(X,Y) - (\hat D_t S)(X,Y)  -  A_\lambda(S^\sharp(X),Y) + S(S^\sharp(X),Y) \\
& = \lambda\left(R(\dot \gamma, X) Y - (\nabla_{\dot \gamma} T)(X,Y) + (\nabla_X T)(\dot \gamma, Y) - T(T(\dot \gamma, X), Y)\right) \\
& \qquad  + \lambda( T(\dot \gamma, T(X,Y)) - T(X, T(\dot \gamma, Y)))  - (\hat D_t S)(X,Y)  \\
& \qquad + \langle (A^\sharp_\lambda + S^\sharp)(X),(A_\lambda^\sharp+S^\sharp) (Y) \rangle_g + \frac{1}{2}  \lambda (\nabla_{\dot \gamma} T)(X,Y)  \\
&   \qquad - \frac{1}{2}  \lambda \Big(T(\dot \gamma, T(X,Y)) - T(T(\dot \gamma, X), Y) - T(X, T(\dot \gamma, Y))\Big) \\
& = \lambda\left(R(\dot \gamma, X) Y -  \frac{1}{2} (\nabla_{\dot \gamma} T)(X,Y) - (\nabla_X T)(Y, \dot \gamma) - \frac{1}{2} T(T(\dot \gamma, X), Y)\right) \\
& \qquad  - \frac{1}{2} \lambda( T(T(X,Y), \dot \gamma) + T( T(Y, \dot \gamma),X))  - (\hat D_t S)(X,Y)  \\
& \qquad + \langle (A^\sharp_\lambda + S^\sharp)(X),(A_\lambda^\sharp+S^\sharp) (Y) \rangle_g .
\end{align*}
We look at the anti-symmetric part. If $\circlearrowright$ denotes the cyclic sum over three elements, and using that $\lambda(R(X,Y) \dot \gamma) = 0$ from compatibility of the connection,
\begin{align*}
& \mathfrak{R}^S(X,Y)-  \mathfrak{R}^S(Y,X) \\
& = \lambda(\circlearrowright R(\dot \gamma, X) Y  - \circlearrowright (\nabla_{\dot \gamma} T)(X,Y)  -  \circlearrowright T( T(\dot \gamma, X) ,Y)) =0,
\end{align*}
from the first Bianchi identity of connections with torsion, see \eqref{Bianchi}, Appendix. Considering the symmetric part, the result follows.
\end{proof}

\begin{remark}[Jacobi fields] Let $S(t)$ be an arbitrary symmetric tensor along $\gamma(t)$. Considering  Jacobi fields as the projection of a solution of \eqref{JacobiEq}, we deduce from Lemma~\ref{lemma:EFGeneral} that $V(t)$ is a Jacobi field if and only if there exists a one-from $\alpha(t)$, such that
$$D_t^S V(t) = \sharp \alpha(t), \qquad D_t^S \alpha(t) = \mathfrak{R}^S_{\gamma}(V(t)).$$
\end{remark}

\subsection{Pullback sections and homogeneous sections}
Recall that $\pi: T^*M\to M$ denotes the canonical projection from the cotangent bundle. Let $\zeta: \calA \to M$ be any vector bundle with an affine connection $\nabla^\calA$. We then introduce the following operator $\dvec = \partial_{\vec{H},\calA} : \Gamma(\pi^* \calA) \to \Gamma(\pi^* \calA)$ defined by
\begin{equation} \label{partialE} \dvec E= (\pi^* \nabla^\calA)_{\vec{H}} E, \qquad E \in \Gamma(\pi^* \calA).\end{equation}
For definition of the pullback connection $\pi^* \nabla^{\calA}$, see Appendix~\ref{sec:Pullback}. Alternatively, $\dvec$ is the unique linear operator on $\Gamma(\pi^* \calA)$, such that for $A \in \Gamma(\calA)$, $f \in C^\infty(T^*M)$,
$$\dvec \pi^* A|_p = \nabla^\calA_{\pi_* \vec{H}_p} A, \qquad \dvec (fA) = (\vec{H} f) A + f \dvec A = \{ H, f \} A + f\dvec A.$$

In what follows, $\mathcal{A}$ will always be tensor bundle, i.e., $\mathcal{A} = T^*M^{\otimes i} \otimes TM^{\otimes j}$ for some $i \geq 0$, $j \geq 0$. This bundle will always be equipped with connection $\nabla^{\calA} = \nabla$ where $\nabla$ is our chosen connections compatible with the sub-Riemannian structure. In this case, we can rewrite \eqref{partialE} as
$$\dvec E |_p = (\pi^* \nabla)_{\hat h \sharp p} E |_p, \qquad E \in \Gamma(\pi^* \calA),$$
with $\hat h$ denoting the horizontal lift with respect to $\hat \nabla$.
We note that if $\gamma(t)$ is a normal geodesic with extremal $\lambda(t)$ and $X_1(t), \dots, X_i(t)$ and $\alpha_1(t), \dots, \dots, \alpha_j(t)$ are respectively vector fields and one-forms along the curve, then
\begin{align*}
& \frac{d}{dt}  E|_{\lambda(t)}(X_1(t), \dots, X_i(t), \alpha_1(t), \dots, \alpha_j(t)) \\
& = (\dvec E)|_{\lambda(t)}(X_1(t), \dots, X_i(t), \alpha_1(t), \dots, \alpha_j(t)) \\
& \qquad + E|_{\lambda(t)}(D_t X_1(t), \dots, X_i(t), \alpha_1(t), \dots, \alpha_j(t)) \\
& \qquad + \cdots +  E|_{\lambda(t)}(X_1(t), \dots, D_t X_i(t), \alpha_1(t), \dots, \alpha_j(t)) \\
& \qquad + E|_{\lambda(t)}(X_1(t), \dots, X_i(t), D_t\alpha_1(t), \dots, \alpha_j(t)) \\
& \qquad + \cdots +  E|_{\lambda(t)}(X_1(t), \dots,  X_i(t), \alpha_1(t), \dots, D_t \alpha_j(t)). 
 \end{align*}

\begin{example}[Euler one-form and derivatives of normal geodesics] \label{ex:Eul}
Consider the canonical section $\eul \in \Gamma(\pi^* T^*M)$ defined by $\eul|_p = p$ for $p \in T^*M$. The notation reflect that $\vl \eul \in \Gamma(T(T^*M))$ is usually referred to as the Euler vector field. We will call $\eul$ the \emph{Euler one-form}.

Let $\gamma(t)$ be a normal geodesic with extremal $\lambda(t)$. Then by definition, we have that
$$D_t^k \lambda(t) = \dvec^k\eul|_{\lambda(t)}.$$
Recall that from Proposition~\ref{pro:NormalGeodesics} that $\hat D_t \lambda(t) = D_t \lambda(t) + T_{\sharp \lambda(t)}^* \lambda(t) =0$, and hence it follows that
\begin{eqnarray*}
\dvec \eul|_p & = &  - T_{\sharp p}^* p \, ; \\
\dvec^2 \eul|_p & = & - (\nabla_{\sharp p} T)_{\sharp p}^* p + T_{\sharp T_{\sharp p}^* p}^* p + T^*_{\sharp p} T^*_{\sharp p} p \, ; \\
\dvec^3 \eul|_p & = & - (\nabla^2_{\sharp p, \sharp p} T)_{\sharp p}^* p + (\nabla_{\sharp T_{\sharp p}^* p} T)_{\sharp p}^* p + 2(\nabla_{\sharp p} T)_{\sharp T_{\sharp p}^* p}^* p + 2 (\nabla_{\sharp p} T)_{\sharp p}^* T_{\sharp p}^* p \\
& &  + T_{\sharp (\nabla_{\sharp p}T)_{\sharp p}^* p}^* p - T_{\sharp T_{\sharp T^*_{\sharp p} p}^* p}^* p - T_{\sharp T_{\sharp p}^* T_{\sharp p}^* p}^* p - T_{\sharp T_{\sharp p}^* p}^* T^*_{\sharp p} p \\
& &  - T^*_{T_{\sharp p}^* p} T_{\sharp p} p + T_{\sharp p}^* (\nabla_{\sharp p} T)_{\sharp p}^* - T^*_{\sharp p} T^*_{\sharp T_{\sharp p} p} p - T^*_{\sharp p} T^*_{\sharp p} T^*_{\sharp p} p \, . 
\end{eqnarray*}
As a consequence, we have that $D^k_t \dot \gamma(t) = \sharp \dvec^k \eul |_{\lambda(t)}$. Here we have used that $\nabla$ is compatible with $(\calE, g)$, so the map $\dvec$ commutes with the map $\sharp$.
\end{example}

\begin{example}[Hamiltonian functions from vector fields] \label{ex:HX}
For any section $X \in \Gamma(\pi^* TM)$, consider the Hamiltonian function $H_X(p) = p(X|_p)$. We can write this function as $H_X = \eul(X)$, so we have
$$\{ H, H_X\} = \dvec H_X = (\dvec \eul)(X) + \eul (\dvec X) = - H_{T(\sharp \eul, X)} + H_{\dvec X}.$$
\end{example}

\begin{remark} \label{re:khom}
Let $\pr_-: \pi^* \calA \to \calA$ be the natural projection and let $U \subset T^*M$ be an open set. We say that a section $E \in \Gamma(\pi|_U^* \calA)$ is \emph{$k$-homogeneous} if for any $p \in U$ and $c \in \mathbb{R}$ such that $cp \in U$, we have
$$\pr_- E|_{cp} = \pr_- c^k E_p.$$
By definition, we have that if $E$ is a $k$-homogeneous section, then $\dvec E$ is a $k+1$-homogeneous section.
\end{remark}

\subsection{Twist polynomials}
Let $(M, \calE, g)$ be a sub-Riemannian manifold with a compatible connection $\nabla$. In order change between the adjoint connection $\hat \nabla$, which is related to the Hamiltonian, and the connection $\nabla$ which is compatible with the sub-Riemannian structure, we introduce the idea of \emph{twist polynomials} $P_k$.

Let $\End TM \cong T^*M \otimes TM$ be the endomorphism bundle equipped with the connection $\nabla$. We define $P_k \in \Gamma(\pi^* \End TM)$, by
$$P_0 = \id_{TM}, \qquad P_1 = - T(\sharp p, \, \cdot \,) = - T_{\sharp p} ,$$
and iteratively,
$$P_{k} = \dvec P_{k-1} + P_1P_{k-1} = (\dvec + P_1)^k \id_{TM}.$$
We note that it follows from this definition that $P_k = \sum_{j=0}^{k-1} \binom{k-1}{j} P_j \dvec^{k-j-1} P_1$.
Also, $P_k$ is a $k$-homogeneous section by definition. In fact, for $x \in M$ fixed, the map $T^*_x M \to \End T_x M$, $p \mapsto P_k|_p$ is a polynomial function. Furthermore, they have the following properties.
\begin{lemma}
Let $\gamma(t) = \exp(tp)$ be a normal geodesic with extremal $\lambda(t)$.
\begin{enumerate}[\rm (a)]
\item Let $X(t)$ be a vector field along $\gamma(t)$. Then for $k \geq 0$,
\begin{equation} \label{PkRel}  \hat D_t P_k|_{\lambda(t)} X(t)= P_k |_{\lambda(t)} D_t X(t) + P_{k+1}|_{\lambda(t)} X(t).\end{equation}
\item Let $\ptr_t, \hptr_t: T_{x} M \to T_{\gamma(t)} M$ be parallel transport along $\gamma$ relative to respectively $\nabla$ and $\hat \nabla$. Then we have
$$P_k|_p = \frac{d^k}{dt^k} \hptr_t^{-1} \ptr_t |_{t=0}.$$
\end{enumerate}
\end{lemma}
\begin{proof} The statement in (a) follows from the definition of $\dvec$. To prove (b), let $\lambda(t) = e^{t\vec{H}}(p)$ be an extremal along normal geodesic $\gamma(t)$ and write $P_k|_{\lambda(t)} = P_k(t)$. Recall that if $X(t)$ is a vector field along the curve, then $D_t X(t) = \ptr_t \frac{d}{dt} \ptr_t^{-1} X(t)$ and similar relations hold for the covariant derivatives of~$\hat \nabla$. We can rewrite the relation \eqref{PkRel} as
$$\hptr_t \frac{d}{dt} \hptr_t^{-1} P_k(t) \ptr_t = P_{k+1}(t) \ptr_t .$$
Using this formula iteratively and the fact that $P_0(t) = \id_{T_{\gamma(t)}M}$, we have that
\[
\hptr_t^{-1} P_{k}(t)\ptr_t = \frac{d}{dt} \hptr_t^{-1} P_{k-1} (t) \ptr_t = \frac{d^k}{dt^k} \hptr_t^{-1} \ptr_t. \qedhere\]
\end{proof}

\begin{example}
Recall the definition of the Euler one-form $\eul$ from Example~\ref{ex:Eul}. We then have $P_1 = - T_{\sharp \eul}$,
\begin{eqnarray*}
\dvec P_1  & = & - (\nabla_{\sharp \eul} T)_{\sharp \eul} - T_{\sharp \dvec \eul}, \\
\dvec^2 P_1 & = & - (\nabla_{\sharp \eul, \sharp \eul}^2 T)_{\sharp \eul} -2(\nabla_{\sharp \eul} T)_{\sharp \dvec \eul} - (\nabla_{\sharp \dvec \eul} T)_{\sharp \eul} - T_{\sharp \dvec^2 \eul}.
\end{eqnarray*}
The second and third twist polynomials are then given by
\begin{eqnarray*}
P_2 & = & \dvec P_1 + P_1^2, \\
P_3 & = & \dvec^2 P_1  + 2 P_1 \dvec P_1+ (\dvec P_1) P_1 + P_1^3. 
\end{eqnarray*}

\end{example}

\subsection{Ampleness and the Young diagram of a geodesic} Let $\nabla$ be a connection compatible with the sub-Riemannian structure and let $P_0, P_1, P_2, \dots$ be the corresponding twist polynomial. Our goal in this section will be to determine the Young diagram of a geodesic from the twist polynomials from \eqref{PkRel}. For any $p \in T_x^*M$, introduce the following flag of subspaces in $T_xM$,
$$0 = \mathfrak{E}_p^0 \subseteq \mathfrak{E}_p^1 = \calE_x \subseteq \mathfrak{E}_p^2 \subseteq \cdots ,$$
where for $i \geq 0$,
\begin{equation} \label{FrakE} \mathfrak{E}_p^i = \spn \{ P_j|_p v \, : v \in \calE_{\pi(p)}, 0 \leq j \leq i-1 \}. \end{equation}
Define $[P_i] |_p: \calE_{\pi(p)} \to TM/ \mathfrak{E}_p^i$ by
\begin{equation}
\label{Pcheck} [P_i]|_p v = P_i|_p v \mod \mathfrak{E}_p^i.
\end{equation}

Then we have the following result.
\begin{proposition}
\begin{enumerate}[\rm (a)]
\item The spaces $\{\mathfrak{E}_p^i\}_{p \in T^*M}$ are independent of choice of compatible connection for any $i \geq 0$.
\item Let $\lambda(t) = e^{t\vec{H}}(p)$ be an extremal with projection $\gamma(t) = \exp(tp)$, $t \in [0,t_1]$. Then the following holds.
\begin{enumerate}[\rm (i)]
\item $\gamma(t)$ is an abnormal curve if and only if $\mathfrak{E}_{\lambda(t)}^i$ is a proper subspace of $T_{\gamma(t)}M$ for any $t \in [0,t_1]$ and $i \geq 0$.
\item The curve $\gamma(t)$ is ample at $t_0$ if and only if $\mathfrak{E}_{\lambda(t_0)}^s = T_{\gamma(t_0)} M$ for some $s \geq 1$.
\item Write
$$d_{i+1}(p) = \rank [P_i] |_p = \rank \mathfrak{E}_p^{i+1} - \rank \mathfrak{E}_p^{i}.$$
Then $\gamma(t)$ is equiregular if and only if $d_{i}(\lambda(t))$ is constant for any $i \geq 0$.
\item If $\gamma(t)$ is an ample, equiregular geodesic with $d_i = d_i(\lambda(t))$ and such that $\mathfrak{E}^{s-1}_{\lambda(t)} \subsetneq T_{\gamma(t)} M$ and $\mathfrak{E}^s_{\lambda(t)} = T_{\gamma(t)} M$, then its Young diagram $\mathbb{Y}$ is $\mathbb{Y} = \mathbb{Y}(d_1, \dots, d_s)$.
\end{enumerate}
\end{enumerate}
\end{proposition}
\begin{proof}
Let $\Lambda(t)$ be the Jacobi curve of a normal geodesic $\gamma(t)$, define $E(\alpha)$ as in \eqref{E} as in consider $F_0(X)$ defined as in \eqref{FS} with $S=0$. Recall the definition of $\Lambda^{(i)}$ in \eqref{Lambdai}. By Lemma~\ref{lemma:EFGeneral}, for any $i \geq 1$,
$$\Lambda^{(i)}(t) = \{ E(\alpha)(t) , F_0( (D_t^0)^j X)(t) \, : \alpha \in \Gamma(\gamma^* T^*M) , X \in \Gamma(\gamma^* \calE),0 \leq j \leq i-1\}.$$
Furthermore, we note that for any $X \in \Gamma(\gamma^* \calE)$ horizontal
$$D_t^0 X(t) = \hat D_t X(t) \mod \calE_{\gamma(t)} = \hat D_t P_0 X(t) \mod \calE_{\gamma(t)} = P_1 X(t) \mod \calE_{\gamma(t)},$$
and similarly,
$$(D_t^0) P_iX(t) = P_{i+1} X(t) \mod \mathfrak{E}_{\lambda(t)}^{i+1}.$$
As a result, we have
$$\Lambda^{(i)}(t) = \{ E(\alpha)(t) , F_0(X)(t) \, : \alpha \in \Gamma(\gamma^* T^*M) , X \in \Gamma(\gamma^* TM), X(t) \in \mathfrak{E}_{\lambda(t)}^i \}.$$
The result follows from this realization.
\end{proof}
\begin{remark}
We note that if $\calE^k_x$, $x \in M$ is defined as in Section~\ref{sec:DefSR}, then we have $\mathfrak{E}^k_p \subseteq \mathcal{E}_x^k$ for any $p \in T^*_xM$. However, these will not coincide in general.
\end{remark}

\subsection{Canonical frames}
In this section, we rewrite the result of Theorem~\ref{th:ZLBasis} in terms of our connection $\nabla$ and  Lemma~\ref{lemma:EFGeneral}. 
\begin{theorem}[Canonical frame along a geodesic] \label{th:CanonicalFrame}
Let $\gamma = \exp(tp)$ be an ample, equiregular geodesic with Young diagram $\mathbb{Y} = \mathbb{Y}(d_1, \dots, d_s)$ and reduced Young diagram $\mathsf{Y} = \mathbb{Y}(\mathsf{d}_1, \dots, \mathsf{d}_s)$. Then there exists a unique symmetric map $S \in \Gamma(\gamma^* \Sym^2 T^*M)$ such that there is a choice of frame $\{ X_{a,b}\}_{(a,b) \in \mathbb{Y}}$ along~$\gamma$ with the property that for $1 \leq a \leq d_1$, $1 \leq b < n_a$,
$$D_t^S X_{a,b}  = X_{a,b+1}, \qquad D_t^S X_{a,n_a} =0,$$
$$\text{$X_{1,1}$ , $\dots$, $X_{d_1,1}$ is an orthonormal basis of $\calE$ along $\gamma$,}$$
and, furthermore, if $\mathfrak{R}^S_\gamma$ is defined as in \eqref{frakRS}, then for any $(a,b), (i,j) \in \mathbb{Y}$ with $[a,b] = (\mathsf{a}, \mathsf{b})$, $[i,j] = (\mathsf{i}, \mathsf{j})$ then
\begin{enumerate}[\rm (i)] 
\item if $\mathsf{a} = \mathsf{i}$ and  $b = \mathsf{b} < \mathsf{n_a}$, then
$$\mathfrak{R}^S_\gamma(X_{a,b}, X_{i,b+1}) = -\mathfrak{R}^S_\gamma(X_{a,b+1}, X_{i,b});$$
\item if $\mathsf{a} = \mathsf{i}$ and $\mathsf{j} \neq \{\mathsf{b} -1, \mathsf{b}, \mathsf{b}+1\}$, then $\mathfrak{R}^S_\gamma(X_{a,b}, X_{i,j}) =0;$
\item if $\mathsf{a} < \mathsf{i}$, $\mathsf{j} < \mathsf{n_i}$ and $\mathsf{j} \not \in \{ \mathsf{b}, \mathsf{b}+1\}$, then $\mathfrak{R}^S_\gamma(X_{a,b}, X_{i,j}) =0;$
\item if $\mathsf{a} < \mathsf{i}$, $\mathsf{b} < \mathsf{n_i} -1$, then $\mathfrak{R}^S_\gamma(X_{a,b}, X_{i,n_i}) =0 ;$
\item if $\mathsf{a} < \mathsf{i}$ and $\mathsf{n_{a}}- \mathsf{n_{i}} \geq \mathsf{b} + \mathsf{j} $, then  $\mathfrak{R}^S_\gamma(X_{a,b}, X_{i,j}) =0$.
\end{enumerate}
The frame $\{X_{a,b} \}_{(a,b) \in \mathbb{Y}} $ is unique up to a choice of an initial choice of orthonormal bases of $\spn\{ X_{a,1}(0) \, : \, [a,1] = (\mathsf{a},1)\}$ for any $1 \leq \mathsf{a} \leq \mathsf{d}_1$.
\end{theorem}
\begin{proof}
Let us rewrite $E_{a,b} = (-1)^{b-1} E(\alpha_{a,b})$ and $F_{a,b} = (-1)^{b-1} F_S(X_{a,b})$. From \eqref{omegaHV} and the requirement that $\{ E_{a,b}, F_{a,b}\}$ is a Darboux basis, it follows that $\{\alpha_{a,b}\}_{(a,b) \in \mathbb{Y}}$ is the coframe of $\{X_{a,b}\}_{(a,b)\in \mathbb{Y}}$. The equations for the derivatives of $E_{a,b}$ and $F_{a,b}$ are by Lemma~\ref{lemma:EFGeneral} equivalent to 
\begin{equation}  \label{alphaX} \left\{ \qquad
\begin{aligned}[c]
D_t^S X_{a,b}  &=  X_{a,b+1},  \\
\sharp \alpha_{a,1} &=X_{a,1}, \\
D_t^S \alpha_{a,b+1} &= - \alpha_{a,b}, \\
\end{aligned}
\qquad
\begin{aligned}[c]
D_t^S X_{a,n_a}  & =0, \\
\sharp \alpha_{a,b+1} & = 0, \\
D_t^S \alpha_{a,1}& = 0, 
\end{aligned} \right.
\end{equation}
for $1 \leq b < n_a$. Note that since $\alpha_{a,1} (X_{i,1}) = \langle X_{a,1}, X_{i,1} \rangle = \delta_{a,i}$, it follows that $\{ X_{a,1}(t)\, : \, 1 \leq a \leq \rank \calE_{\gamma(t)} \}$ is an orthonormal basis along $\gamma$. This fact together with the equations for $D_t^S$-derivatives of $X_{a,b}$ determine all the equations \eqref{alphaX}.
\end{proof}

We will call $\{ X_{a,b} \}_{(a,b) \in \mathbb{Y}}$ a \emph{canonical frame} along the geodesic $\gamma$. Define subbundles of $TM$ along the geodesic $\gamma$ by
$$\Box^{\mathsf{a},\mathsf{b}}(t) = \spn \{ X_{a,b}(t) \, : \, [a,b] = (\mathsf{a}, \mathsf{b}) \}.$$
We also give these spaces an inner product, such that $X_{a,b}$ becomes an orthonormal basis. We see that these spaces and their inner product are independent of choice of initial frame.

\begin{proposition} \label{prop:Decomp}
Let $[P_1], \dots, [P_s]$ be defined as in \eqref{Pcheck}.
\begin{enumerate}[\rm (a)]
\item For any $1 \leq b \leq s$, if we define $\mathfrak{E}_p^b$ as in \eqref{FrakE}, then
$$\mathfrak{E}_{\lambda(t)}^b = \bigoplus_{\begin{subarray}{c} (\mathsf{i}, \mathsf{j}) \in \mathsf{Y} \\ \mathsf{j} \leq \mathsf{b} \end{subarray}} \Box^{\mathsf{i}, \mathsf{j}}(t).$$
\item We have $\Box^{\mathsf{d}_1,1}(t) = \ker [P_1]|_{\lambda(t)}$ and iteratively,
$$\Box^{\mathsf{a}, 1}(t) = \ker [P_{\mathsf{n_{a}}}] |_{\lambda(t)} \cap \left( \oplus_{\mathsf{i} > \mathsf{a} } \Box^{\mathsf{i}, 1}(t) \right)^\bot,$$
where the orthogonal complement is defined relative to $g$ in $\calE_{\gamma(t)}$.

\end{enumerate}

\end{proposition}

\begin{proof}
\begin{enumerate}[\rm (a)]
\item It is clear that the statement is true for $b=1$ and hence we can complete the proof by induction. If our hypothesis holds true for $b$, then we have that
\begin{align*}
\bigoplus_{\begin{subarray}{c} (\mathsf{i}, \mathsf{j}) \in \mathsf{Y} \\ \mathsf{j} \leq \mathsf{b}+1 \end{subarray}} \Box^{\mathsf{i}, \mathsf{j}}(t) 
& =  \bigoplus_{\begin{subarray}{c} (\mathsf{i}, \mathsf{j}) \in \mathsf{Y} \\ \mathsf{j} \leq \mathsf{b} \end{subarray}} \Box^{\mathsf{i}, \mathsf{j}}(t) 
+ \left\{ D^S_t X(t) \, : \, X(t) \in \bigoplus_{\begin{subarray}{c} (\mathsf{i}, \mathsf{j}) \in \mathsf{Y}_p \\ \mathsf{j} \leq \mathsf{b} \end{subarray}} \Box^{\mathsf{i}, \mathsf{j}}(t) \right\} \\
& =  \mathfrak{E}_{\lambda(t)}^b 
+ \left\{ D^S_t X(t) \, : \, X(t) \in \mathfrak{E}_{\lambda(t)}^b \right\} .\end{align*}
By the definition, we can write any vector field $X(t)$ with valued in $\mathfrak{E}_{\lambda(t)}^b$ as $X(t) = \sum_{j=0}^{b-1} P_j |_{\lambda(t)}Y_{j+1}(t)$ where $Y_1, \dots, Y_b$ takes values in $\calE_{\gamma(t)}$. We then note that
$$D_t^S X(t) = \hat D_t X(t) \mod \mathfrak{E}_{\lambda(t)}^b = P_{b}|_{\lambda(t)} Y_{b}  \mod \mathfrak{E}_{\lambda(t)}^b,$$
and so by definition, we have $\oplus_{\begin{subarray}{c} (\mathsf{i}, \mathsf{j}) \in \mathsf{Y} \\ \mathsf{j} \leq \mathsf{b}+1 \end{subarray}} \Box^{\mathsf{i}, \mathsf{j}}(t) = \mathfrak{E}_{\lambda(t)}^{b+1}$.
\item If $X(t)$ is a vector field with values in $\Box^{\mathsf{a},1}(t)$, then
\begin{align*}
& (D_t^S)^k X(t) = (\hat D_t)^k X(t) \mod \calE_{\gamma(t)}\\
& = (\hat D_t)^{k-1} (D_t X(t) +P_1|_{\lambda(t)} X(t)) \mod \calE_{\gamma(t)} \\
&  = \sum_{j=0}^{k} \binom{k}{j} P_{j}|_{\lambda(t)} D_t^{k-j} X(t) \mod \calE_{\gamma(t)} .
\end{align*}
Hence, if $(D_t^S)^{n_a} X_{a,1} =0$, then $P_{n_a} X_{a,1}$ is a linear combination of elements in $\mathfrak{E}_{\lambda(t)}^{n_a}$. It follows that $[P_{n_a}] X_{a,1} = 0$. To complete the proof, we note that for the block $\Box^{\mathsf{d}_1,1}$ satisfies $\mathsf{n}_{\mathsf{d}_1} =1$ since there are always some elements in the kernel of $[P_1]|_{\lambda(t)}$. In particular, we have that $P_1|_{\lambda(t) }  \dot \gamma(t) =0$.  \qedhere
\end{enumerate}
\end{proof}

Only the values of $S(X,Y)$ where at least one of the vector fields are horizontal has an impact on the connection $D_t^S$. Hence, the values of $S$ when both $X$ and $Y$ are in the span of $\{ X_{a,b} \, : \, b \geq 2\}$ are determined by the curvature conditions (i)-(v). This next result reveals how curvature normalization conditions determine $S$.

\begin{proposition} \label{prop:CurvXab}
Let $\{ X_{a,b} \}_{(a,b) \in \mathbb{Y}}$ be a canonical frame along $\gamma$ corresponding to extremal $\lambda$. With the convention that $X_{a,n_a+1} =0$, then
\begin{align*}
\mathfrak{R}^S_\gamma(X_{a,b},X_{i,j}) & =\frac{1}{2} \lambda R(\dot \gamma, X_{a,b})X_{i,j}  + \frac{1}{2} \lambda R(\dot \gamma, X_{i,j})X_{a,b}  \\
& \qquad+ \frac{1}{2} \lambda (\nabla_{X_{a,b}} T)(\dot \gamma, X_{i,j}) + \frac{1}{2} \lambda (\nabla_{X_{i,j}} T)(\dot \gamma, X_{a,b}) \\ \nonumber
&\qquad + \langle A^\sharp(X_{a,b}), A^\sharp (X_{i,j}) \rangle_g - \langle S^\sharp(X_{a,b}), S^\sharp(X_{i,j}) \rangle_g  \\
& \qquad - \frac{d}{dt} S(X_{a,b},X_{i,j}) + S(X_{a,b+1} , X_{i,j}) + S(X_{a,b}, X_{i,j+1}).
\end{align*}
\end{proposition}

\begin{proof}
The result follows from the computation
\begin{align*}
& - (\hat D_t S)(X_{a,b}, X_{ij}) = - \frac{d}{dt} S(X_{a,b}, X_{i,j}) + S(\hat D_t X_{a,b}, X_{i,j}) + S(X_{a,b}, \hat D_tX_{i,j}) \\
& = - \frac{d}{dt} S(X_{a,b}, X_{i,j}) + S(X_{a,b+1} - A^\sharp_\lambda (X_{a,b}) - S^\sharp (X_{a,b}) , X_{i,j}) \\
 & \qquad + S(X_{a,b}, X_{i,j+1} - A^\sharp_\lambda(X_{i,j}) - S^\sharp (X_{i,j})) \\
 & = - \frac{d}{dt} S(X_{a,b}, X_{i,j}) + S(X_{a,b+1}, X_{i,j}) + S(X_{a,b}, X_{i,j+1}) - 2 \langle S^\sharp (X_{a,b}), S^\sharp(X_{i,j}) \rangle_g  \\
  & \qquad -\langle  S^\sharp (X_{a,b}) , A^\sharp(X_{i,j}) \rangle_g - \langle A^\sharp(X_{a,b}), S^\sharp(X_{i,j}) \rangle_g. \qedhere
\end{align*}
\end{proof}

\subsection{The canonical horizontal frame through twist polynomials}
Let $\gamma(t)$ be an ample, equiregular geodesic of step $s$ with Young diagram $\mathbb{Y}$. Let $\nabla$ be a connection compatible with our sub-Riemannian structure with corresponding twist polynomials $P_1, \dots, P_{s}$. In the previous section, we showed that we can determine a decomposition $\calE_{\gamma(t)} = \oplus_{\mathsf{a}=1}^{\mathsf{d}_1} \Box^{\mathsf{a},1}(t)$ from knowing the twist polynomials. Let $\pr_{\mathsf{a}}(t) = \pr^{\mathsf{a},1}(t): \calE_{\gamma(t)} \to \Box^{\mathsf{a},1}(t)$ be the corresponding projections and $\pr_{< \mathsf{a}} = \sum_{\mathsf{i}=1}^{\mathsf{a}-1}\pr_{\mathsf{i}}$. We will show how we can determine the horizontal elements of the canonical frame as well.

Recall from Proposition~\ref{prop:Decomp} that $P_{\mathsf{n_{a}}} \Box^{\mathsf{a,1}} \subseteq \mathfrak{E}^{\mathsf{n_a}}$ and $P_{\mathsf{n_{a}}+1} \Box^{\mathsf{a,1}} \subseteq \mathfrak{E}^{\mathsf{n_a}+1}$. Hence, $P_{\mathsf{n_a}}$ and $P_{\mathsf{n_a}+1}$ can be decomposed into twist polynomials of lower order on $\Box^{\mathsf{a},1}$. Define linear maps $B(t),C(t): \calE_{\gamma(t)} \to \calE_{\gamma(t)}$ according to the following rules.
\begin{enumerate}[$\bullet$]
\item For any $1 \leq \mathsf{a} \leq \mathsf{d}_1$ and $u \in \Box^{\mathsf{a},1}$,
\begin{align} \label{Bmap} P_{\mathsf{n_a}}|_{\lambda(t)} u &= - P_{\mathsf{n_a} -1}|_{\lambda(t)} B(t) u  \mod \mathfrak{E}^{\mathsf{n_a}-1}_{\lambda(t)} \\ \nonumber
&= - P_{\mathsf{n_a} -1}|_{\lambda(t)} (B_0(t) + B_+(t)) u \mod \mathfrak{E}^{\mathsf{n_a}-1}_{\lambda(t)}.\end{align}
with $B_0(\Box^{\mathsf{a},1}) = \Box^{\mathsf{a},1}$ and $B_+(\Box^{\mathsf{a},1}) = \oplus_{\mathsf{i} < \mathsf{a}} \Box^{\mathsf{i},1}$.
\item For any $1 \leq \mathsf{a} \leq \mathsf{d}_1$ and $u \in \Box^{\mathsf{a},1}$,
\begin{equation} \label{Cmap} P_{\mathsf{n_{a}}+1}|_{\lambda(t)} u = - P_{\mathsf{n_a} }|_{\lambda(t)} C(t) u \mod \mathfrak{E}^{\mathsf{n_a}}_{\lambda(t)}.\end{equation}
with $C(\Box^{\mathsf{a},1}) = \oplus_{\mathsf{i} < \mathsf{a}} \Box^{\mathsf{i},1}$.
\end{enumerate}
Observe that for $\mathsf{a} = 1$, we have $\mathfrak{E}^{\mathsf{n}_1} = TM$, so $C$ always vanishes on $\Box^{1,1}$. Hence, it is sufficient to compute $P_1, \dots, P_{\mathsf{n}_1}$ for to find the maps $B$ and $C$.

In what follows, we will let $b^\dagger$ denote the dual of an endomorphism $b: \calE_x \to \calE_x$ with respect to the inner product $g_x$.
\begin{theorem}[Universal formula for the canonical horizontal frame] \label{th:HorizontalFrame}
Define an anti-symmetric linear map $Q(t) : \calE_{\gamma(t)} \to \calE_{\gamma(t)}$ by
\begin{equation} \label{Qmap} Q(t) = \frac{1}{2} \sum_{\mathsf{a} =1}^{\mathsf{d}_1} \frac{1}{\mathsf{n_a}} \pr_{\mathsf{a}} (B_0 - B_0^\dagger - 2A^\sharp) \pr_{\mathsf{a}} + (C- B_+) - (C - B_+)^\dagger.\end{equation}
Let $\{ X_{a,b} \}_{(a,b) \in \mathbb{Y}}$ be the canonical frame uniquely determined by $X_{a,1}(0) = u^{a} \in \Box^{[a,1]}(0)$. Then $X_{1,1}, \dots, X_{d_1,1}$ are solutions of
$$D_t X_{a,1} = Q X_{a,1}, \qquad X_{a,1}(0) = u^a.$$
Furthermore, 
\begin{align} \label{ScalE}
S^\sharp |_{\calE} & =  \frac{1}{2} (B_0 + B_0^\dagger)  - \sum_{\mathsf{a} =1}^{\mathsf{d}_1} ( \mathsf{n_{a}} (C- B_+) + \pr_{< \mathsf{a}} A^\sharp) \pr_{\mathsf{a}} \\ \nonumber
& \qquad  - \left(  \sum_{\mathsf{a} =1}^{\mathsf{d}_1} ( \mathsf{n_{a}} (C- B_+) + \pr_{< \mathsf{a}} A^\sharp) \pr_{\mathsf{a}} \right)^\dagger. \end{align}
Finally, we note that if $\wp_1(t): \calE_{\gamma(t)} \to T_{\gamma(t)} M$ is defined by
$$\wp_1(t)=(P_1 + A^\sharp + S^\sharp)|_{\lambda(t)}|_{\calE} + Q(t),$$
with $S^\sharp|_{\calE}$ and $Q$ given as above, then $\wp_1 u =0$ for $u \in \Box^{\mathsf{d_1},1}$, while for any $X_{a,1}$ with $n_a >1$, we have
$$\wp_1 X_{a,1} = X_{a,2}.$$
\end{theorem}
\begin{proof}
Let $\{ X_{a,b}\}_{(a,b) \in \mathbb{Y}}$ be the canonical frame with the given initial conditions.
Write $D_tX_{a,1} = Q(t) X_{a,1} = (Q_0(t) + Q_+(t) - Q_+(t)^\dagger) X_{a,1}$ where $Q_0(\Box^{\mathsf{a},1}) = \Box^{\mathsf{a},1}$ and $Q_+(\Box^{\mathsf{a},1}) = \oplus_{\mathsf{i}< \mathsf{a}} \Box^{\mathsf{i},1}$. In this definition, we have used that $Q(t): \calE_{\gamma(t)} \to \calE_{\gamma(t)}$ is anti-symmetric since $X_{1,1}, \dots, X_{d_1,1}$ is an orthonormal frame. We note the following observations.
\begin{enumerate}[$\bullet$]
\item For any section $Y$ of $\mathfrak{E}_{\lambda(t)}^k$, $\hat D_t Y$ is a section of $\mathfrak{E}_{\lambda(t)}^{k+1}$.
\item As a corollary of the above statement, we have that for any vector field $Y$ along $\gamma(t)$ and $l \geq k$,
$$(D_t^S)^k Y(t) = \hat D_t^k Y(t) \mod \mathfrak{E}_{\lambda(t)}^l.$$
\item If $X$ is a section of $\Box^{\mathsf{a},1}$ and $k \geq n_{\mathsf{a}}$, then $\hat D^k X$ is a section of $t \mapsto \mathfrak{E}^k_{\lambda(t)}$.
\end{enumerate}
From these observations and the definition of the canonical frame,
\begin{align*}
& 0  = (D_t^S)^{n_a} X_{a,1} \\
& = (D_t^S)^{n_a-1} (P_1  + Q_0 + Q_+ - Q_+^\dagger  + A^\sharp + S^\sharp ) X_{a,1} \mod \mathfrak{E}^{n_a-1} \\
& = \hat D_t^{n_a-1} (P_1  + Q_0 + Q_+   + A^\sharp + S^\sharp ) X_{a,1} \mod \mathfrak{E}^{n_a-1} \\
& = P_{n_a } X_{a,1}  + P_{n_a-1}( n_a (Q_0 + Q_+ )  + (A^\sharp + S^\sharp) ) X_{a,1} \mod \mathfrak{E}^{n_a-1} \\
& =   P_{n_a-1}(- B_0 - B_+ + n_a (Q_0 + Q_+ )  + (A^\sharp + S^\sharp) ) X_{a,1} \mod \mathfrak{E}^{n_a-1}.
\end{align*}
From here, we have that for $u,v \in \Box^{\mathsf{a},1}$,
$$\langle \mathsf{n_a} Q_0 u, v \rangle_g = \left\langle \left(\frac{1}{2} B_0 - \frac{1}{2} B_0^\dagger - A^\sharp \right) u, v \right\rangle_g, \quad
S(u, v )  = \left\langle \left(\frac{1}{2} B_0 + \frac{1}{2} B_0^\dagger \right) u, v \right\rangle_g .$$
%
Furthermore, for $\mathsf{i} < \mathsf{a}$ and $u \in \Box^{\mathsf{a},1}$, $v \in \Box^{\mathsf{i},1}$, we have
$$S(u ,v ) = \langle  (B_+ - \mathsf{n_a} Q_+ - A^\sharp) u , v \rangle_g. $$
Next, we observe that
\begin{align*}
& P_{n_a} Q_+ X_{a,1} =  P_{n_a} Q X_{a,1} \mod \mathfrak{E}^{n_a}  = P_{n_a} D_t X_{a,1} \mod \mathfrak{E}^{n_a} \\
& =  - P_{n_a+1} X_{a,1} + \hat D_t P_{n_a} X_{a,1} \mod \mathfrak{E}^{n_a}  =   P_{n_a} C X_{a,1} - \hat D_t P_{n_a -1} B X_{a,1} \mod \mathfrak{E}^{n_a} \\
& =  P_{n_a} (C- B_+) X_{a,1} \mod \mathfrak{E}^{n_a}.
\end{align*}
This gives us formulas for $Q$ and $S^\sharp|_{\calE}$. Finally, we note that if $n_a \geq 2$,
\[X_{a,2} = D_t^S X_{a,1} = (P_1  + Q + A^\sharp + S^\sharp) X_{a,1}. \qedhere \]
\end{proof}

\begin{example}[Geodesics of Riemannian geometry] For the case $\calE_{\gamma(t)} = T_{\gamma(t)}M$ where $\mathbb{Y} = \mathbb{Y}(\dim M)$ and $\mathsf{Y} = \mathsf{Y}(1)$, we will only need
$$P_1|_{\lambda(t)} = - T_{\sharp \lambda(t)} = - T_{\dot \gamma(t)} = - B(t).$$
By Theorem~\ref{th:HorizontalFrame}, we have that
$$Q = \frac{1}{2} \left(T_{\dot \gamma} - T_{\dot \gamma}^\dagger - 2A^\sharp_{\lambda} \right) ,\qquad 
S^\sharp (X_{a,1}) = \frac{1}{2} \left( T_{\dot \gamma} + T_{\dot \gamma}^\dagger  \right) X_{a,1},$$
so
\begin{align*}
D^S_t & = D_t - T_{\dot \gamma} + A^\sharp_{\lambda} + \frac{1}{2} \left( T_{\dot \gamma} + T_{\dot \gamma}^\dagger  \right) = D_t - \frac{1}{2} T_{\dot \gamma} + A^\sharp_{\lambda} + \frac{1}{2}  T_{\dot \gamma}^\dagger = D_t - Q.
\end{align*}
which is the covariant derivative of the Levi-Civita connection of $g$, as expected.
\end{example}

\subsection{The Bonnet-Myers theorem for the final box}
Let $\gamma(t)$ be an ample, equiregular geodesic with Young diagram $\mathbb{Y} = \mathbb{Y}(d_1, \dots, d_s)$ and reduced Young diagram $\mathsf{Y} = \mathbb{Y}(\mathsf{d}_1, \dots, \mathsf{d}_s)$. We will consider the Ricci curvature of the final box in the Young diagram. Let $\nabla$ be a compatible connection with corresponding twist polynomials $P_1, P_2 , \dots $. We define
$$\underline{\Box}(t) = \Box^{\mathsf{d}_1,1}(t) = \ker [P_1] = \spn \{ X_{a,1} \, : \, [a,1] = ( \mathsf{d}_1, 1)\},$$
with orthogonal projection $\underline{\pr}: \calE_{\gamma(t)} \to \underline{\Box}(t)$.
From Proposition~\ref{prop:CurvXab}, for any $(a,1)$ with $[a,1] = (\mathsf{d}_1,1)$,
\begin{align*}
\mathfrak{R}^S_\gamma(X_{a,1},X_{a,1}) & = \lambda R(\dot \gamma, X_{a,1})X_{a,1} + \lambda (\nabla_{X_{a,1}} T)(\dot \gamma, X_{a,1}) + |A^\sharp(X_{a,1})|_g^2 \\ 
&\qquad  - | S^\sharp(X_{a,1})|_g^2   - \frac{d}{dt} S(X_{a,1},X_{a,1}) .
\end{align*}
and by Theorem~\ref{th:HorizontalFrame},
\begin{align*} 
S^\sharp X_{a,1} & =  \frac{1}{2} (B_0 + B_0^\dagger) X_{a,1}  + (2B_+ - C)X_{a,1} - (\id - \underline{\pr}) A^\sharp X_{a,1}. \end{align*}
We then have that
\begin{align} \label{RicUnderline}
& \underline{\Ric}(t) = \Ric(\mathsf{d}_1,1)(t) = \tr_{\underline{\Box}(t)} \mathfrak{R}^S_{\gamma}(\times, \times) = \underline{\tr} \, \mathfrak{R}^S_{\gamma}(\times, \times) \\ \nonumber
& = \underline{\tr} \, \lambda R(\dot \gamma, \times)\times + \underline{\tr} \, \lambda (\nabla_{\times} T)(\dot \gamma, \times) + \frac{1}{4} |pT(\underline{\pr} \, \cdot \,, \underline{\pr} \, \cdot \,)|_{g^* \otimes g^*}^2 \\ \nonumber
&\qquad - \frac{1}{4} \underline{\tr} |(B_0+ B_0^\dagger) \times |_g^2   - |(2B_+ -C) \underline{\pr} \, \cdot \, |_{g^* \otimes g}^2 \\ \nonumber
& \qquad +2 \, \underline{\tr} \, \langle (2B_+ - C) \times, A^\sharp \times \rangle_g  -  \frac{d}{dt} \underline{\tr} \, \langle B_0 \times, \times \rangle.
\end{align}
Using Theorem~\ref{th:BMTheorem} we have the following result.
\begin{theorem} \label{th:BMFinalBox}
Assume a dense set of $M$ connected by ample, equiregular geodesic. Assume that for any such geodesic with Young diagram $\mathbb{Y} = \mathbb{Y}(d_1, \dots, d_s)$ with $d_1 - d_2 >1$ and
$$\frac{1}{d_1-d_2-1} \underline{\Ric}(t) \geq k_1 |\dot \gamma|^2 ,$$
where $\underline{Ric}$ is as in \eqref{RicUnderline}.
Then $M$ is compact, with finite fundamental grop and diameter bound
$$\diam M \leq \frac{\pi}{\sqrt{k_1}}.$$
\end{theorem}

We are now able to prove the theorem mentioned in the introduction.
\begin{proof}[Proof of Theorem~\ref{th:Universal}] \label{sec:ProofThUniversal}
Let $\bar{g}$ be any taming Riemannian metric of $g$. Let $\calA$ be the orthogonal complement of $\calE$. We define $\nabla$ by \eqref{NiceNabla}.  Let $\gamma$ be an ample equiregular geodesic. Write $\underline{\Box}(t) = \ker T(\dot \gamma, \, \cdot \,) = \ker P_1$. Since $T(\calE, \calE) \subseteq \calA$, we have that $B |\Box^{\mathsf{d}_1, 1} = B| \underline{\Box}= 0$. We furthermore have for any $v \in \underline{\Box}(t)$,
$$P_2v = - (\nabla_{\dot \gamma(t)} T)( \dot \gamma , v) - T(D_t \dot \gamma, v) =  - P_1 C(t) v=T(\dot \gamma, C(t)v),$$
and
\begin{align*}
\underline{\Ric}(\lambda(t)) & = \underline{\tr} \, \lambda R(\dot \gamma, \times)\times + \underline{\tr} \, \lambda (\nabla_{\times} T)(\dot \gamma, \times) + \frac{1}{4} |\lambda(t)T(\underline{\pr} \, \cdot \,, \underline{\pr} \, \cdot \,)|_{g^* \otimes g^*}^2 \\ 
&\qquad  - \underline{\tr} |C \times |_g^2 - \underline{\tr} \, \lambda(t) T(\times, C\times) .
\end{align*}
Finally, we will discuss independence of connection.
\begin{enumerate}[$\bullet$]
\item For any $p \in T^*_xM \setminus \Ann(\calE)_x$, we define $K(p): \calE \to TM/\calE$,
$$K(p): X_x \mapsto - [Y, X]|_x \mod \calE, \qquad Y_x = \sharp p.$$
Then $\ker T(\sharp p, \, \cdot \, ) = \ker K(p)$ which does not depend on $\nabla$.
\item For any $v \in \Box_p$, write $\gamma(t) = \exp(tp)$ and let $X$ and $Y$ be any horizontal vector fields with $X_x = v$ and $Y_x = \sharp p$. Then
\begin{align*}
& - (\nabla_{\dot \gamma(t)} T)( \dot \gamma , v) - T(D_t \dot \gamma, v) = - D_t T(\dot \gamma, X(t)) \\
&  = [Y, [Y,X]]_x \mod \spn \{ \calE, K(p) \calE\} = 0.
\end{align*}
Hence $C_p$ is defined for every $p \not \in \Ann(\calE)$.
\item Finally, by definition, $\Ric$ does not depend on $\nabla$ along extremals of ample equiregular, geodesics. The set of covectors $p$ such that $\exp(tp)$ is ample for short time is open and dense by Lemma~\ref{lemma:AmpleDense}. Furthermore, since the rank of each $\mathfrak{E}_{p}^k$ can only increase locally, there is an open dense set of covectors such that $\exp(pt)$ is ample and equiregular for short time. \qedhere
\end{enumerate}
\end{proof}

%% file: 6Cont.tex
\section{Global reformulation} \label{sec:Global}
\subsection{Maximal Young diagram}
In this section, we take the previous description of curvature along each geodesic and rewrite them in terms of tensors. More precisely, we want look at properties of all geodesics with maximal Young diagram through sections of pullbacks of tensor bundles. Recall that for $d_1 \geq d_2 \geq \cdots \geq d_s$, we write $\mathbb{Y}(d_1, \dots, d_s)$ for the Young diagram with $s$ columns, where column $i$ has $d_i$ boxes. For convenience, we write $d_{j} = 0$ for $j > s$. We give the set of all Young diagrams lexicographic ordering, i.e., we say that $\mathbb{Y}(d_1, \dots, d_s) > \mathbb{Y}(e_1, \dots, e_r)$ if there is some $i$ where $d_j = e_j$ for $1 \leq j <i$ and $d_i > e_i$.

Let $(M, \calE, g)$ be a sub-Riemannian manifold and let $\nabla$ be a compatible connection. Let $P_1, P_2, \dots, $ be the corresponding twist polynomials. Let $\pi:T^*M \to M$ be the canonical projection. For $p \in TM$, define $d_1(p)  = d_1 = \rank \calE_{\pi(p)} = \rank \mathfrak{E}_p^1 = \rank [P_0]|_p$. Iteratively, we define
$$d_i(p) = \left\{ \begin{array}{ll} \rank [P_{i-1}]|_p & \text{if $d_{i-1}(p) \geq \rank [P_{i-1}]|_p$, } \\ \\ 0 & \text{if $d_{i-1}(p) < \rank [P_{i-1}]|_p$.} \end{array} \right.$$
In particular, this makes $d_1(p), d_2(p), \dots$ a non-increasing sequence. Write
$$s(p) = \max \{ i \, : \, d_i(p) \neq 0\}, \qquad \mathbb{Y}_p = \mathbb{Y}(d_1(p), \dots, d_{s(p)}(p)).$$
We note the following relation from \cite[Proposition~5.23]{ABR18}.
\begin{proposition}
For any $x \in M$, we have that $\{ p \in T_x^* M \,:\, \mathbb{Y}_p \text{ is maximal in $T_x M$}\}$ is Zarinski open in $T_x^*M$.
\end{proposition}

For an open set $U \subseteq M$, we let $\mathbb{Y}_U$ denote the maximal Young diagram in $\{ \mathbb{Y}_p \, : \, p \in T^* U\}$ with respect to our mentioned ordering. Introduce the set
$$\Sigma U = \{ p \in T^* U \, : \, \mathbb{Y}_p = \mathbb{Y}_U\}.$$
This will always be an open set, as the rank of $[P_i]$ can only increase locally. Since $\pi$ is an open map, $\pi(\Sigma U)$ is open as well.

\begin{definition}
We call an open set $U$ a constancy domain if $\pi(\Sigma U) $ is connected and dense in $U$. A constancy domain is called complete for some open, dense subset $\tilde \Sigma \subseteq \Sigma U$, $\vec{H}|_{\tilde \Sigma}$ is a complete vector field.
\end{definition}
For us, it will be important that if $\gamma(t)$ is a normal geodesic in $U$ with its extremal $\lambda(t)$ contained in $\Sigma U$, then $\gamma$ is ample and equiregular.

\begin{example}[Martinet distribution] \label{ex:Martinet}
Consider the sub-Riemannian manifold $(M, \calE, g)$, where $M = \mathbb{R}^3$ with coordinates $(x,y,z)$ and where $(\calE, g)$ is determined by having orthonormal basis
$$X = \frac{\partial}{\partial x}, \qquad Y = \frac{\partial}{\partial y} + x^2 \frac{\partial}{\partial z}.$$
If $Z = \frac{\partial}{\partial z}$, we define a compatible connection $\nabla$ by $\nabla X = \nabla Y = \nabla Z = 0$. Its torsion is given by
$$T = -2 x dx \wedge dy \otimes Z.$$
We have that
\begin{eqnarray*}
P_1|_p & =& 2x ( p(X) dy - p(Y) dx) \otimes Z, \\
P_2|_p & = & - (\nabla_{\sharp p} T)_{\sharp p} + T_{\sharp T_{\sharp p}^* p} \\
& = & 2p(X) ( p(X) dy - p(Y) dx) \otimes Z + 4x^2 p(Z) (p(X) dx + p(Y) dy) \otimes Z.
\end{eqnarray*}
This gives us $\mathfrak{E}_p^1 = \calE_{\pi(p)}$, 
$$\mathfrak{E}_p^2 = \left\{ \begin{array}{ll} \calE_{\pi(p)} & \text{if $x = 0$ or $p(X) = p(Y) =0$,} \\ T_{\pi(p)} M & \text{otherwise,}\end{array} \right.$$
$$\mathfrak{E}_p^3 = \left\{ \begin{array}{ll} \calE_{\pi(p)} & \text{if $x = p(X) = 0$ or $p(X) = p(Y) =0$,} \\ T_{\pi(p)} M & \text{otherwise.}\end{array} \right.$$
Hence, we have that
$$\mathbb{Y}_p =  \left\{ \begin{array}{ll} \mathbb{Y}(2) & \text{if $x = 0$ or $p(X) = p(Y) =0$,} \\ \mathbb{Y}(2,1) & \text{otherwise.}\end{array} \right. $$
Let $U$ be any open set. Then
$$\pi(\Sigma U) = U \setminus \{ (x,y,z) \in U \, : \, x =0\}.$$
In particular, $U$ is a constancy domain if and only if it does not intersect the plane $ x = 0$.
\end{example}

\begin{example} \label{ex:Changing}
Consider the function $\phi(t) = e^{-1/t^2}$ for $t >0$ and $\phi(t) =0$ for $t \leq 0$. On $M = \mathbb{R}^5$, with coordinates $(x_1, x_2, x_3, y_1, y_2)$, define a sub-Riemannian structure $(\calE,g)$ by defining an orthonormal basis $\spn \{ X_1, X_2, X_3\}$ by
$$X_1 = \partial_{x_1}, \qquad X_2 = \partial_{x_2} + x_1 \partial_{y_1} + \frac{1}{2} x^2_1 \partial_{y_2}, \qquad X_3= \partial_{x_3} + \phi(x_1) \partial_{y_2}.$$
Let us write $W_1 = x_1 \partial_{y_1} + x_1 \partial_{y_2}$ and $W_2 = \partial_{y_2}$. Define a connection $\nabla$ by assuming $\nabla X_j = 0$ and $\nabla W_i=0$ with $j=1,2,3$ and $i=1,2$. Write $X_1^*, X_2^*, X_3^*, W_1^*, W^*_2$ for the dual basis. We then have
$$T = - X_1^* \wedge X_2^* \otimes W_1 - \phi X_1^* \wedge X_3^* \otimes W_2 - X_1^* \wedge W_1^* \otimes W_2.$$
Further computation yields
\begin{align*}
\dvec \eul & = H_{W_1} (H_{X_1} X_2^* - H_{X_2} X_1^*) +  H_{W_2} (\phi H_{X_1} X_3^* - \phi H_{X_3} X_1^* - H_{X_1} W_1^* ), \\
P_1 & = (H_{X_1} X_2^* - H_{X_2} X_1^*) \otimes W_1 + ( \phi H_{X_1} X_3^* - \phi H_{X_3} X_1^* + H_{X_1} W_1^*) \otimes W_2, \\
P_2|_{\calE} & = (- H_{W_1} H_{X_2} - \phi H_{W_2} H_{X_3}) X_2^* \otimes W_1 - H_{W_1} H_{X_1} X_1^* \otimes W_1 \\
& \qquad+  H_{\sharp d\phi} (H_{X_1} X_3^* - H_{X_3} X_1^*) \otimes W_2 + \phi (- H_{W_1} H_{X_2} - \phi H_{W_2} H_{X_3}) X_3^* \otimes W_2 \\
& \qquad -  \phi^2 H_{X_1} X_1^* \otimes W_2 + H_{X_1} (H_{X_1} X_2^* - H_{X_2} X_1^*) \otimes W_2.
\end{align*}
Define $M_+ =\{ x_1 > 0\}$. We see from the above expressions that
$$\mathbb{Y}_U = \left\{\begin{array}{ll} 
\mathbb{Y}(3,1,1) & \text{if $U \cap M_+ = \emptyset$,} \\
\mathbb{Y}(3,2) & \text{if $U \cap M_+ \neq \emptyset$.} \end{array} \right.$$
and
$$\Sigma U = \left\{\begin{array}{ll} 
\{ p \in T^*U \, : \, H_{X_1}(p) \neq 0\} & \text{if $U \cap M_+ = \emptyset$,} \\
\{ p \in T^*(U \cap M^+) \, : \, H_{X_1}(p) \neq 0 \} & \text{if $U \cap M_+ \neq \emptyset$.} \end{array} \right.$$
In particular,
$$\Sigma U = \left\{\begin{array}{ll} 
\pi(\Sigma U) =U & \text{if $U \cap M_+ = \emptyset$,} \\
\pi(\Sigma U)  = U \cap M_+ & \text{if $U \cap M_+ \neq \emptyset$.} \end{array} \right.$$
\end{example}

From Examples~\ref{ex:Martinet} and~\ref{ex:Changing}, we see that if $x \in M$ is a given point, then $\pi(\Sigma U)$ does not need to contain $x$ for any neighborhood $U$ of $x$ and furthermore, for some neighborhood $U$, $\pi(\Sigma U)$ does not need to even have $x$ as a limit point. However, we note that the set
\begin{equation} \label{MSigma} M_{\Sigma} = \{ x \in M \, : \, \text{there exists a neighborhood $U$ of $x$ with $\pi(\Sigma U) = U$} \},\end{equation}
is open (by definition) and dense. To see the latter claim, observe that if $y$ is any point and $U$ any neighborhood of $y$, then $\pi(\Sigma U) \subseteq M_{\Sigma}$ and so $U \cap M_{\Sigma}$ is nonempty of any neighborhood of $y$. Hence, there is an open and dense set $M_{\Sigma}$ in $M$ where we always find a neighborhood with $\pi(\Sigma U) = U$ and where the formalism of this section is well defined.

\begin{example}
\begin{enumerate}[\rm (a)]
\item In Example~\ref{ex:Martinet}, we have $M_\Sigma = M \setminus \{ x =0\}$.
\item In Example~\ref{ex:Changing}, we have $M_\Sigma = M \setminus \{ x_1 =0\}$.
\end{enumerate}
\end{example}

\begin{remark}[On property $(*)$]  \label{re:Star} We make the following remarks on the property $(*)$ from Section~\ref{sec:Introduction} for when the generic minimizing geodesic is normal, ample and equiregular. From the sub-Riemannian Hopf-Rinow theorem, see e.g. \cite{Bel96}, we know if $M$ is complete then for any $x \in M$, there is a minimizer connecting $x$ with any other point. As mentioned in Section~\ref{sec:minimizers}, minimizers can also be abnormal, however, a dense subset on $M$ will have a normal geodesic as its minimizer by \cite{RiTr05}. In particular, if $(M,\calE,g)$ is a complete constancy domain then it satisfies property $(*)$.

Notice from Example~\ref{ex:Martinet} that for every $(x_0,y_0,z_0)$ with $x_0 \neq 0$, $\gamma(t) = \exp(tp)$, $p \in (\Sigma M)_{(x_0,y_0,z_0)}$, will always be ample and equiregular for short time, but might lose their equiregularity property if the geodesic crosses the line $x = 0$. This example shows that even though the generic short geodesic will be ample and equiregular, this need not be a generic property of geodesic defined on their maximal time interval. If we know something about the sub-Riemannian manifold $(M, \calE, g)$ to ensure some constant local structure, then completeness implies the property $(*)$. This is the case for contact manifold and manifolds with fat horizontal bundles. For definition of fat subbundles, see Section~\ref{sec:Step2}.
\end{remark}

\subsection{Continuous formulation of the canonical connection}
We will now give a continuous formulation of Theorem~\ref{th:ZLBasis} and Theorem~\ref{th:CanonicalFrame}.
\begin{enumerate}[$\bullet$]
\item Let $\nabla$ be any connection compatible with the sub-Riemannian structure $(\calE, g)$ and let $P_1, P_2, \dots, $ be the corresponding twist polynomials. Define $\mathfrak{E}^i$ as in \eqref{FrakE} and correspondingly the maps $[P_i] : \calE \to TM/\mathfrak{E}^i$. Use these maps to determine the set $\bpi = \pi|_{\Sigma M}: \Sigma M \to M$ of covectors with maximal Young diagram. Let $\mathbb{Y} = \mathbb{Y}(d_1, \dots, d_s)$ and $\mathsf{Y}  = \mathbb{Y}(\mathsf{d}_1, \dots, \mathsf{d}_s)$ be respectively the (maximal) Young diagram and the reduced Young diagram  of elements $\Sigma M$.
\item Define a decomposition $\bpi^* \calE = \oplus_{\mathsf{a} =1}^{\mathsf{d}_1} \Box^{\mathsf{a},\mathsf{b}}$ into subbundles by $\Box^{\mathsf{d}_1,1} = \ker [P_1]$ and
$$\Box^{\mathsf{a},1} = \ker [P_{\mathsf{n_a}}] \cap (\oplus_{\mathsf{i} =\mathsf{a} +1}^{\mathsf{d}_1} \Box^{\mathsf{i},1} )^\perp.$$
\item Let $A$ be defined as in \eqref{Amap}. Define maps $B,C: \bpi^*\calE \to \bpi^* \calE$ as in \eqref{Bmap} and \eqref{Cmap}, and introduce the map $Q$ as in \eqref{Qmap}. Note that $Q, B, C$ are $1$-homogeneous as defined in Remark~\ref{re:khom}.
On sections of $\bpi^* \calE$, introduce a differential operator
$$\dvec_{-Q} X = \dvec X - Q X.$$
Extend this to all tensor bundles of $\bpi^* \calE$ by defining $\dvec_{-Q} f = \dvec f$ on functions and requiring it to satisfy the Leibniz rule. In particular, for an endomorphism $b: \bpi^* \calE \to \bpi^* \calE$, we have that
$$\dvec_{-Q} b = \dvec b - Q b + b Q.$$ 
We note in particular that the map $\dvec_{-Q}$ preserves sections of $\Box^{\mathsf{a},1}$, $1 \leq \mathsf{a} \leq \mathsf{d}_1$.
\end{enumerate}
Finally, for any section $S \in \Gamma(\bpi^* \Sym^2 T^*M)$, we introduce the \emph{corresponding twist functions} $\wp_k=\wp^S_k: \bpi^* \calE \to \bpi^* TM$, by $\wp_0 =\id_{\calE}$, $\wp_1 = (P_1 + Q + A^\sharp + S^\sharp)|_{\calE}$ and iteratively
$$\wp_{k+1} = (\dvec_{-Q} + P_1 + Q + A^\sharp + S^\sharp ) \wp_k.$$
Even though $Q$ is only defined on $\bpi^* \calE$, the above expression is well defined since $\wp_{k+1} = (\dvec + P_1 + A^\sharp+ S^\sharp ) \wp_k + \wp_k Q$. Define the corresponding curvature operator $\mathfrak{R}^S \in \Gamma(\bpi^* \Sym^2 T^*M)$ as
\begin{align*}
\mathfrak{R}^S(X, Y) & = \eul R( \sharp \eul, X)X  + \eul (\nabla_X T)(\sharp \eul, X)  \\
& \qquad + | (S^\sharp + A^\sharp)(X) |_{g} - (\dvec S)(X,X) + 2S(X,P_1 X) .
\end{align*}
We have the following continuous formulation.
\begin{theorem}[Canonical twist functions] \label{th:ReformulationCont}
There is a unique choice of $S\in \Gamma(\bpi^* \Sym^2 T^*M)$ satisfying the following properties
\begin{enumerate}[\rm (a)]
\item Its twist functions satisfy $\ker \wp_k = \oplus^{\mathsf{d}_1}_{\mathsf{a} = \mathsf{d}_1-k+1} \Box^{\mathsf{a},1}$.
\item Its curvature operator $\mathfrak{R}^S$ satisfies 
\begin{enumerate}[\rm (i)] 
\item If $ u \in \Box^{\mathsf{a},1}$, $1\leq \mathsf{a} \leq \mathsf{d}_1 $, then for any $b \geq 0$,
$$\mathfrak{R}^S( \wp_b u, \wp_{b+1} u) = 0.$$
\item If $u,v \in \Box^{\mathsf{a},1} $, then for any $b \geq 0$ and $j \neq \{b-1, b, b+1\}$, $\mathfrak{R}^S(\wp_b u , \wp_j v) =0$.
\item If $\mathsf{a} < \mathsf{i}$, $j < \mathsf{n_i}-1$ and $j \not \in \{ b, b+1\}$, then $\mathfrak{R}^S(\wp_b u, \wp_j v) =0$ for any $u \in \Box^{\mathsf{a},1}$ and $v \in \Box^{\mathsf{i},1}$.
\item If $\mathsf{a} < \mathsf{i}$, $b < \mathsf{n_i} -2$, then $\mathfrak{R}^S(\wp_{b} u , \wp_{\mathsf{n_i}-1} v) =0$ for any $u \in \Box^{\mathsf{a},1}$, $v\in \Box^{\mathsf{i},1}$.
\item If $\mathsf{a} < \mathsf{i}$ and $\mathsf{n_{a}}- \mathsf{n_{i}} -2 \geq b+ j$, then  $\mathfrak{R}^S(\wp_{b} u, \wp_{j} v) =0$ for any $u \in \Box^{\mathsf{a},1}$, $v\in \Box^{\mathsf{i},1}$.
\end{enumerate}
\end{enumerate}
\end{theorem}

\begin{proof}
Let $\lambda(t)$ be a normal extremal in $\Sigma M$ with projection $\gamma(t)$. Let $\{ X_{a,b} \}_{(a,b) \in \mathbb{Y}}$ be a canonical basis. Then we know that $X_{a,2} = \wp_1 X_{a,1}$. Using induction, we have that if $X_{a,b} = \wp_{b-1} X_{a,1}$ for $b < n_a$, then
\begin{align*}
X_{a,b+1} & = D^S_t X_{a,b} = (D_t + (P_1 + A^\sharp + S^\sharp )|_{\lambda(t)}) \wp_{b-1}|_{\lambda(t)} X_{a,1} \\
& = (\dvec \wp_{b-1} + \wp_{b-1} Q) X_{a,1} + (P_1 + A^\sharp + S^\sharp) \wp_{b-1}|_{\lambda(t)} X_{a,1} = \wp_b|_{\lambda(t)} X_{a,1}.
\end{align*}
By the same calculations, we have $\wp_{n_a} X_{a,1} =0$. 
\end{proof}
We note that $\wp_k$ is $k$-homogeneous, but not necessarily polynomial. The canonical twist functions gives us a canonical decomposition
$$\bpi^* TM = \oplus_{(\mathsf{a},\mathsf{b}) \in \mathsf{Y}} \Box^{\mathsf{a},\mathsf{b}} = \oplus_{(\mathsf{a},\mathsf{b}) \in \mathsf{Y}} \wp_{\mathsf{b}-1}\Box^{\mathsf{a},1}.$$
Define the corresponding Ricci curvature $\Ric: \Sigma M \to \mathbb{R}^{\mathsf{Y}}$, $\Ric = (\Ric^{\mathsf{a}, \mathsf{b}})_{(\mathsf{a},\mathsf{b}) \in \mathsf{Y}}$ by
\begin{equation} \label{RicDef} \Ric^{\mathsf{a}, \mathsf{b}}(p): = \tr_{\Box^{\mathsf{a}, \mathsf{b}}_p} \mathfrak{R}^S(\times, \times) = \tr_{\Box^{\mathsf{a}, \mathsf{1}}_p} \mathfrak{R}^S(\wp_{b-1}\times, \wp_{b-1}\times).\end{equation}
Along any extremal $\lambda(t)$ contained in $\Sigma M$ of a normal geodesic, we have that this coincides with the Ricci curvature as given in Section~\ref{sec:LQRic}. By definition, $\Ric^{\mathsf{a}, \mathsf{b}}$ is a $2\mathsf{b}$-homogeneous function.

\begin{remark} \label{re:CurvWp}
We can reformulate Proposition~\ref{prop:CurvXab} as
\begin{align*}
\mathfrak{R}^S(\wp_{b} u,\wp_{j} v) & =\frac{1}{2} \eul R(\sharp \eul, \wp_b u) \wp_j v  + \frac{1}{2} \eul R(\sharp \eul, \wp_j u)\wp_b v  \\
& \qquad+ \frac{1}{2} \eul (\nabla_{\wp_b u} T)(\sharp \eul, \wp_j v) + \frac{1}{2} \eul (\nabla_{\wp_j v} T)(\sharp \eul, \wp_b u) \\ \nonumber
&\qquad + \langle A^\sharp(\wp_b u), A^\sharp (\wp_j v) \rangle_g - \langle S^\sharp(\wp_b u), S^\sharp(\wp_j v) \rangle_g - \dvec S(\wp_b u,\wp_j v) \\
& \qquad + S(\wp_{b+1} u , \wp_{j} v) + S(\wp_b u, \wp_{j+1} v).
\end{align*}
In particular,
\begin{align*}
\Ric^{\mathsf{a},\mathsf{b}} & = \tr_{\mathsf{a}} \eul R(\sharp \eul, \wp_{\mathsf{b}-1} \times) \wp_{\mathsf{b}-1} \times+ \tr_{\mathsf{a}} \eul (\nabla_{\wp_{\mathsf{b}-1} \times} T)(\sharp \eul, \wp_{\mathsf{b}-1} \times) \\ \nonumber
&\qquad + \tr_{\mathsf{a}} \langle A^\sharp( \wp_{\mathsf{b}-1} \times), A^\sharp ( \wp_{\mathsf{b}-1} \times) \rangle_g - \tr_{\mathsf{a}} \langle S^\sharp( \wp_{\mathsf{b}-1} \times), S^\sharp( \wp_{\mathsf{b}-1} \times) \rangle_g \\
& \qquad - \dvec \tr_{\mathsf{a}} S( \wp_{\mathsf{b}-1} \times, \wp_{\mathsf{b}-1} \times)  + 2 \tr_{\mathsf{a}} S(\wp_{\mathsf{b}} \times , \wp_{\mathsf{b}-1} \times) .
\end{align*}

\end{remark}

\subsection{Computational algorithm} \label{sec:Algorithm}
We will summarize the previous section with a practical algorithm for computing the connection and curvature of a sub-Riemannian connection defined in \cite{ZeLi09,ABR18,BaRi16} using our methods. Let $(M, \calE, g)$ be a sub-Riemannian manifold and with cotangent bundle as $\pi: TM \to M$. Let $\eul \in \Gamma(\pi^* T^*M)$ be Euler section $\eul|_p = p$.
\begin{enumerate}[\rm (I)]
\item Choose an affine connection $\nabla$ compatible with the sub-Riemannian structure that has torsion $T$ and curvature $R$. Define $A \in \Gamma(\pi^* TM)$ by $A = \frac{1}{2} \eul T(\, \cdot \, , \, \cdot \,)$.
\item Compute sufficiently many twist polynomials to determine the set $\bpi: \Sigma M\to M$, the decomposition $\bpi^* \calE = \oplus_{i=1}^{\mathsf{d}_1} \Box^{i,1}$ and find $B$ and $C$ as in \eqref{Bmap} and \eqref{Cmap}. Actually, it is sufficient to complete the following computations.
\begin{enumerate}[\rm (a)]
\item Compute $P_1 = - T(\sharp \eul, \, \cdot \,)$. This is the only twist polynomial one needs to find completely. Define $\mathfrak{E}^2_p = \calE + P_1|_p \calE$ and let $\Sigma^1$ denote the set of all $p$ where the rank of $\mathfrak{E}^2$ is maximal. Define $\Box^{\mathsf{d}_1, 1} = \ker [P_1]$ on $\Sigma^1$ with $\mathsf{d}_1$ to be determined.
\item For $k \geq 1$, assume that $\Sigma^k$, $\mathfrak{E}^{k+1}$ is well defined. We also assume that for some $a \geq 1$, we have $\Box^{\mathsf{d}_1,1}$, $\dots$, $\Box^{\mathsf{d}_1-i-1,1}$ defined such that $\ker [P_j]  \cap ( \oplus_{i=1}^a \Box^{\mathsf{d}_1 -i +1,1})^\perp = 0$ for any $j \leq k$.

\

Write $\dvec = \pi^*\nabla_{\vec{H}}$. If $\mathfrak{E}^{k+1}$ is a proper subset of $\pi^*|_{\Sigma^k} TM$ and $P_{k+1} = (\dvec + P_1) P_k$, we only need to compute $P_{k+1} |_{\calE} \mod \mathfrak{E}^{k}$. This is sufficient to find $\mathfrak{E}^{k+2} = \mathfrak{E}^{k+1} + P_{k+1} \mathfrak{E}^{k+1}$ and define $\Sigma^{k+1}$ as the set of elements in $\Sigma^k$ such that $\mathfrak{E}^{k+2}$ has maximal rank. Finally if $\ker [P_j]  \cap ( \oplus_{i=1}^a \Box^{\mathsf{d}_1 -i +1,1})^\perp \neq 0$, define the intersection as $\Box^{\mathsf{d}_1 - a,1}$.

\

If $\mathfrak{E}^{k+1}$ equals $\pi^*|_{\Sigma^k} TM$, then $\Sigma^k = \Sigma M$, $a = \mathsf{d}_1$ and $k +1 = \mathsf{n}_1 $. The final computation needed is $P_{\mathsf{n}_1} | \oplus_{\mathsf{n_a} \geq \mathsf{n}_1 -1} \Box^{\mathsf{a},1} \mod \mathfrak{E}^{\mathsf{n}_1 -1}$. 
\item Having completed the above steps, we can define $B, C: \bpi^* \calE \to \bpi^* \calE$ such that for any $u \in \Box^{\mathsf{a},1}$,
$$P_{\mathsf{n_a}} u = - P_{\mathsf{n_{a}} -1} Bu \mod \mathfrak{E}^{\mathsf{n_a}-1}, \qquad P_{\mathsf{n_a}+1} u = -P_{\mathsf{n_a}} Cu \mod \mathfrak{E}^{\mathsf{n_a}}.$$
We have $C(\Box^{\mathsf{a},1}) = \oplus_{\mathsf{i} < \mathsf{a}} \Box^{\mathsf{i},1}$, and a decomposition $B = B_0 + B_+$ with $B_0(\Box^{\mathsf{a},1}) = \Box^{\mathsf{a},1}$ and $B_+(\Box^{\mathsf{a},1}) = \oplus_{\mathsf{i} < \mathsf{a}} \Box^{\mathsf{i},1}$.
\end{enumerate}
\item Define $Q$ as in \eqref{Qmap}.
\item We finally need to determine $S \in \Gamma(\bpi^* \Sym^2 T^*M)$.
\begin{enumerate}[\rm (a)]
\item For $v, u \in \bpi^* \calE$, $S(u,v)$ is determined from $B$ and $C$ by \eqref{ScalE}.
\item Define $\wp_1 = P_1 + Q + A^\sharp + S^\sharp$ and iteratively $\wp_{k+1} = (\dvec_{-Q} +\wp_1)\wp_k$. We determine $S$ from the condition
$$\ker \wp_k = \oplus_{\mathsf{a} =1}^k \Box^{\mathsf{d}_1-\mathsf{a}+1,1}.$$
and the curvature normalization conditions (i)-(v).
\end{enumerate}
\end{enumerate}

\subsection{Canonical non-linear connetion}
Let $\Sigma M$ be the set of covectors with maximal Young diagram. By restricting ourselves to $\pi(\Sigma M)$, we can consider $\bpi: \Sigma M \to M$ as a fibration. We define $h_p, \hat h_p: T_{\pi(p)} M \to T_p\Sigma M$ as horizontal lift relative to respectively $\nabla$ and $\hat \nabla$. Relative to $\nabla$, let $S$ be the canonical symmetric map. If we write $\calV = \ker \pi_*$, we define a connection $T(\Sigma M) = \calH^S \oplus \calV$ with
\begin{align*}
\calH^S_p & = \{ h_p^Sv = \hat h_p v + \vl_p ( A|_p - S|_p)(v)  \, : \, v \in T_{\bpi(p)} M \} \\
& = \{ h^S_p v=  h_p v - \vl_p ( A|_p + S|_p)(v) \, : \, v \in T_{\bpi(p)} M \}.\end{align*}
In other words, if $\gamma(t)$ is a geodesic in $M$ with extremal $\lambda(t)$ contained in $\Sigma M$, then $\Gamma_S(t) = e^{-t\vec{H}}_* \calH^S_{\lambda(t)}$ is the canonical complement to the Jacobi curve as defined in Section~\ref{sec:Complements}.
We note that in particular $h^S_p \sharp p = \vec{H}|_p = \hat h_p \sharp p$. We will write a decomposition
$$\calH^S = \bigoplus_{(\mathsf{a}, \mathsf{b}) \in \mathsf{Y}} \calH^S(\mathsf{a},\mathsf{b}), \qquad \pi_* \calH^S(\mathsf{a},\mathsf{b})_p = \Box^{\mathsf{a}, \mathsf{b}}_p.$$
Introduce an endomorphism $\wp_+: \bpi^* TM \to \bpi^* TM$ by defining
$$\wp_+: \wp_{\mathsf{b}-1} u \mapsto \wp_{\mathsf{b}} u, \qquad u \in \calE, \qquad \mathsf{b} =1, \dots, \mathsf{n_a}.$$
Then for any section $X \in \Gamma(TM)$,
$$[\vec{H} , h^S X] = h^S \wp_+ X - \vl R^S(\sharp p, X) = h^S \wp_+ X + \vl \mathfrak{R}^S(X).$$
where $R^S$ is the curvature operator of $\calH^S$, see Appendix~\ref{sec:NonlinearC}. Similarly, we have that
$$[\vec{H}, \vl \alpha ] = - \vl \wp_+^* \alpha - h^S \sharp \alpha.$$

%% file: Dim3.tex
\section{Sub-Riemannian manifolds with growth vector $(2,3)$} \label{sec:23}
We will do the computations for the simplest non-trivial general case, 3-dimensional contact manifolds, with the methods introduced above. To compare with previous computations, see \cite[Section~7.5]{ABR18}, \cite{ABR17} and \cite[Section~17]{ABB19}.

\subsection{Connection and geodesics}
Consider a sub-Riemannian manifold $(M, \calE, g)$ with $M$ of dimension~$3$ and with $\calE$ of rank~$2$. We will work locally around a regular point of~$\calE$. Hence, we can assume that the growth vector of~$\calE$ is always $(2,3)$.
\begin{example}[Bundles over Riemannian surfaces] \label{ex:BundleRS}
Let $\check{M}$ be a Riemannian surface with Riemannian metric $\check{g}$. Let $y_0 \in \check{M}$ be any point. By working locally around~$y_0$, we may assume that $\check{M}$ has trivial cohomology. Choose any orientation on $\check{M}$ and let $\mu$ be the corresponding Riemannian volume form. Define $M : = \check{M} \times \mathbb{R}$ and consider the fibration $\zeta: M \to \check{M}$ where $\zeta$ is the projection on the first factor. Let $z\in \mathbb{R}$ denote the coordinates of the fibers of~$\zeta$. If $\alpha$ is any one-form on $\check{M}$ satisfying $d\alpha = \mu$, we define a corresponding sub-Riemannian structure $(\calE, g)$ on $M$ by
$$\calE = \ker \theta, \qquad \theta := dz - \zeta^* \alpha, \qquad \langle u,v \rangle_g = \langle \zeta_* u, \zeta_* v \rangle_{\check{g}}, \qquad u,v \in \calE.$$ 
If any other one-form $\tilde \alpha$ with $d\tilde \alpha = \mu$ is used to define a sub-Riemannian structure $(\tilde \calE, \tilde g)$, then the result will only differ by an isometry. Explicitly, we must have $\tilde \alpha = \alpha + df$ for a unique $f \in C^\infty(\check{M})$ with $f(y_0) = 0$ by our assumptions on cohomology, and the resulting isometry $\phi: (M, \calE, g) \to (M, \tilde \calE, \tilde g)$ is then given as $\varphi(y, z) = (y, z + f(y))$, $z \in \mathbb{R}$, $y \in \check{M}$. Furthermore, if we use $-\alpha$ in the place of $\alpha$, we again get something isometric through the map $(y, z) \mapsto (y,-z)$, and hence, changing the orientation of $\check{M}$ also give us something isometric. It is also simple to verify that if we have two sub-Riemannian manifolds $M^1$ and $M^2$ constructed in this way from respectively $\check{M}^1$ and $\check{M}^2$, then any local sub-Riemannian isometry $\varphi$ from a neighborhood of $M^1$ into $M^2$ induces a corresponding local isometry between $\check{M}^1$ and $\check{M}^2$.

In conclusion, the local geometry of a sub-Riemannian manifold $(M, \calE, g)$ constructed in the above way is uniquely determined by the local geometry of $(\check{M},\check{g})$ and consequently by its Gaussian curvature.
\end{example}

We will now show that Example~\ref{ex:BundleRS} describes the local geometry of any $(2,3)$-sub-Riemannian manifold with locally bounded curvature. Recall the definition of $\Ric^{\mathsf{a},\mathsf{b}}$ as in \eqref{RicDef}.

\begin{theorem} \label{th:23BoundedCurvature}
Let $(M,\calE,g)$ be a sub-Riemannian manifold with constant growth vector of $\calE$ equal to $(2,3)$. Then the maximal Young diagram is $\mathbb{Y}(2,1)$ which is the Young diagram of every $p \in \Sigma M = TM \setminus \Ann(\calE)$. Define
$$k_1(x) = \inf_{p \in \Sigma M_x, |p|_{g^*}=1} \Ric^{1,1}(p) , \qquad k_2(x) = \inf_{p \in \Sigma M_x, |p|_{g^*} =1} \Ric^{1,2}(p) .$$
Then $k_2(x) \in \{ - \infty , 0\}$ for every $x \in M$.  If $k_2 \equiv 0$, then $(M, \calE, g)$ is locally isometric to a bundle over a Riemannian manifold as in Example~\ref{ex:BundleRS} with $k_1$ being the pull-back of the Gaussian curvature.
\end{theorem}
We will give the proof of this theorem in Section~\ref{sec:Proof23}. We note that if $(M, \calE, g)$ is complete, $k_2 \equiv 0$ and $k_1 \geq k>0$, then  $M$ is compact with diameter bound $\frac{2\pi}{\sqrt{k}}$ and with finite fundamental group by Section~\ref{sec:LQRic}. This diameter bound is sharp for any scaling of the Hopf fibration $S^1 \to S^3 \to S^2$, see e.g. \cite{LeLi18}.

\subsection{Notation and local assumptions}
As we are considering local geometry, we may assume that $\calE$ and $M$ are orientable. Choose an arbitrary orientation of~$\calE$. This gives us a corresponding endomorphism $J: \calE \to \calE$ such that for any unit vector $v\in \calE_x$, $\{ v, Jv\}$ is a positively oriented basis of $\calE_x$. We let $\theta$ be the unique one-form satisfying $\ker \theta = \calE$ and
$$d\theta (v, Jv) = - |v |_g^2, \qquad v \in \calE.$$
Define $Z$ as \emph{the Reeb vector field} of $\theta$, i.e. the unique vector field satisfying
$$\theta(Z) = 1, \qquad d\theta(Z, \, \cdot \,) = 0.$$
Introduce a taming Riemannian metric $\bar{g}$ of $g$ by defining $\calA = \spn \{ Z\}$ to be orthogonal to $\calE$ with $Z$ being a unit vector field. Let $\pr_{\calA}$ and $\pr_{\calE}$ be the corresponding orthogonal projections. Also, for every vector field $X$, we write $X^* = \langle X, \, \cdot \, \rangle_{\bar{g}}$ for the corresponding one-from.

We extend $J$ to an endomorphism of $TM$ by defining $JZ = 0$, giving us the identities
$$d\theta(u,v) = \langle u, Jv \rangle_{\bar{g}}, \qquad J^2 = -\pr_\calE.$$
Define a symmetric endomorphism $\tau : \calE \to \calE$ by $\frac{1}{2} (\calL_Z g)(u, v) =  \langle \tau u, v \rangle_g$. We note that $\calL_Zg$ is well defined since $[X,Z]$ takes values in $\calE$ for any horizontal $X$. We also extend $\tau$ to all of $TM$ by the relation $\tau Z = 0$. Then
\begin{align*}
0 & = d(\calL_Z \theta)(u,v) = (\calL_Z d\theta)(u,v) \\
&= ( \calL_Z g)(u,Jv) + \langle u, (\calL_Z J) v \rangle_g = \langle u, (2\tau J + (\calL_Z J)) v \rangle_g, \end{align*}
and hence obtain the identity $(\calL_Z J) J = -J (\calL_Z J)= 2\tau$. In particular, if $w \in \calE_x$, $x \in M$ is an eigenvector of $\tau$, then so is the orthogonal $Jw$ with eigenvalue only differing by a sign. Hence, we know that locally there exists a unit vector field $\xi \in \calE$  and a function $\chi \in C^\infty(M)$ such that
$$\tau = \chi \xi^* \otimes \xi^* - \chi  (J\xi)^* \otimes J \xi.$$

Define the Tanno connection, see e.g. \cite{Tan89},
$$\nabla_X Y = \pr_{\calE} \nabla^{\bar{g}}_{\pr_{\calE} X}  \pr_{\calE} Y+ \theta(X) ([Z, Y] + \tau Y) + (X \theta(Y)) Z .$$
with $\pr_{\calE}$ being the orthogonal projection to $\calE$. Write $\Sigma M = T^*M \setminus \Ann(\calE)$ with canonical projection $\bpi: \Sigma M \to M$. Finally, define functions $r$ and $\varphi$ on $\Sigma M$ by
$$p(\xi) = r(p) \cos \varphi(p), \qquad p(J\xi) = r(p) \sin \varphi(p),$$
and a one-form $\beta(v) := \langle \nabla_v \xi, J \xi\rangle_g$.
Since $\{\xi, J\xi\}$ form a local orthonormal basis for $\calE$, the covariant derivatives of $\xi$ are uniquely determined by $\beta$.
\begin{proposition} \label{prop:23Connection}
Every $p \in \Sigma$ has Young diagram and reduced Young diagram $\mathbb{Y} = \mathsf{Y} = \mathbb{Y}(2,1)$. We have decomposition $\bpi^* TM = \Box^{1,1} \oplus \Box^{1,2} \oplus \Box^{2,1}$ with orthonormal basis
$$\begin{array}{ll} \Box^{1,1} & \Box^{1,2} \\ \Box^{2,1} & \end{array} = \begin{array}{ll} \spn\{ Y_1 \} & \spn\{X_{1,2}\} \\ \spn\{Y_0\} & \end{array} = \arraycolsep=2pt\def\arraystretch{1.5} \begin{array}{ll} \spn \{ \frac{1}{r} J\sharp \eul \} & \spn\{ r Z -  \frac{1}{r} H_Z \sharp \eul\} \\  \spn \{ \frac{1}{r} \sharp \eul \} & \end{array} $$
If $\{ \alpha_{a,b}\}_{(a,b) \in \mathbb{Y}}$ is the dual basis of $\{X_{a,b} \}_{(a,b) \in \mathbb{Y}}$, then
$$\alpha_{2,1} = \frac{1}{r} \eul ,\qquad \alpha_{1,1} = \frac{1}{r} (J\sharp \eul)^* \qquad \alpha_{1,2} = \frac{1}{r} \theta.$$
We have corresponding canonical connection $T\Sigma M = \calH^S \oplus \ker \bpi_*$ spanned by
\begin{align*}
h^S Y_0 & = h Y_0 +H_Z \vl \alpha_{1,1} + r^2\chi \cos(2\varphi) \vl \alpha_{1,2}, \\
h^S Y_1 & = h Y_1 - 2 r^2 \chi \sin(2\varphi) \vl \alpha_{1,2} , \\
h^S X_{1,2} & = h X_{1,2} - (H_Z^2 + r^2 \chi \sin (2\varphi)) \vl \alpha_{1,1} \\
& \qquad  - r^2 \left( 4 r \chi \beta(Y_0) -  r d\chi(Y_1) - 4\chi H_Z  \right)\cos(2\varphi) \vl \alpha_{1,2} \\ \nonumber
& \qquad   + 2 r^3 \left( \chi \beta(Y_1)  + d\chi(Y_0) \right) \sin(2\varphi) \vl \alpha_{1,2}.
\end{align*}
Furthermore, the only non-zero parts of the curvature $\mathfrak{R}^S$ is given by
\begin{align*}
\mathfrak{R}^S(Y_1,Y_1) & = r^2 \kappa   + H_Z^2 + 3 r^2 \chi \sin (2\varphi) , \\
\mathfrak{R}^S(X_{1,2},X_{1,2}) & = r^4 d\chi(Z + 8 \beta(Y_0) Y_0 ) \cos(2\varphi)  + 4 r^4 \chi \beta(Y_0) \beta(Y_1) \cos(2\varphi)  \\
& \qquad  -  r^4 (\nabla_{Y_0} d\chi)(Y_1) \cos(2\varphi) + 4 r^4 \chi (\nabla_{Y_0} \beta)(Y_0) \cos(2\varphi) \\ 
& \qquad + 2r^4 \chi \beta(Z + 4 \beta(X_{1,2}) X_{1,2}) \sin(2 \varphi)  -  2 r^4 \chi^2 (1+ \cos^2(2\varphi)) \\
& \qquad - 2  r^4 \beta(Y_0) d\chi(Y_1 + \beta(Y_1)  Y_0) \sin(2\varphi)  \\
& \qquad - 2 r^4 (\nabla_{Y_0} d\chi)(Y_0) \sin(2\varphi) + 6 r^2 \chi H_Z^2 \sin(2\varphi)  \\
& \qquad  - 8 r^3 H_Z\left(   d\chi(Y_0) \cos(2\varphi) + 2 \chi \beta(Y_0)  \sin(2\varphi)\right)  .
\end{align*}
\end{proposition}

\subsection{Proof of Proposition~\ref{prop:23Connection}}
By the definition of the Tanno connection, $\nabla J = 0$, $\nabla g^* = 0$ and $\nabla \bar{g} = 0$. 
 Its torsion is given by
$$T = d\theta \otimes Z + \chi \theta \wedge \xi^* \otimes \xi - \chi \theta \wedge (J\xi)^* \otimes J \xi. $$
Write $Y_0 = \frac{1}{r} \sharp \eul$ and $JY_0 = Y_1$. We note that $\dvec r =0$ and that
\begin{align*}
A & = - \frac{1}{2} H_Z Y_0^* \wedge Y_1^* + \frac{1}{2} r \chi \theta \wedge (\cos \varphi \xi^* - \sin \varphi (J \xi)^*),\\
\dvec \eul & =  rH_Z Y_1 + r^2 \chi \cos(2 \varphi) \theta ,
\end{align*}
and hence $\dvec Y_0  = H_Z Y_1$, $\dvec Y_1  = - H_Z Y_1$,
$$\dvec H_Z = r^2 \chi \cos(2\varphi), \qquad \dvec Z  = 0, \qquad \dvec \varphi =  H_Z -r  \beta(Y_0).$$

\subsubsection{Canonical decomposition}
From the expression of the torsion
\begin{align*}
P_1 & = r X^*_{1,1} \otimes Z +  r\chi \theta \otimes( \cos \varphi \xi - \sin \varphi J\xi ), \\
P_2 |_{\calE} & = -  r H_Z Y_0^* \otimes Z. 
\end{align*}
We hence have that $\Box^{2,1} = \ker [P_1] = \spn \{ Y_0 \} $, $\Box^{1,1} = \spn \{ Y_1\} $ and
$$B Y_0 = 0, \qquad BY_1 = 0, \qquad C Y_0 = H_Z Y_1, \qquad C Y_1 =0.$$
It follows that
$$Q = H_Z ( Y_0^* \otimes Y_1 - Y_0^* \otimes Y_1), \qquad S^\sharp |_{\calE} = -\frac{1}{2} H_Z (Y_0^* \otimes Y_1 + Y_1^* \otimes Y_0 )  .$$
Using these formulas together, we get $\Box^{1,2} =\spn \{X_{1,2} \}$,
$$X_{1,2} = \wp_1 Y_1 = (P_1 + Q + A^\sharp + S^\sharp) = r Z - H_Z Y_0 .$$

\subsubsection{Connection}
We will determine $S$ and hence the canonical connection. We will first use that
\begin{align*}
& \wp_2 Y_1  = 0 = (\dvec + P_1 + A^\sharp + S^\sharp ) X_{1,2} \\
& = - r^2 \chi \cos(2\varphi) Y_0 - H_Z^2 Y_1 + r ^2\chi (\cos \varphi \xi - \sin \varphi J \xi)  \\
& \qquad + \frac{1}{2} r^2 \chi (\cos \varphi \xi - \sin \varphi J\xi) + \frac{1}{2} H_Z^2 Y_1 + S^\sharp X_{1,2} ,
\end{align*}
so
\begin{align*}
S^\sharp X_{1,2} & = r^2 \chi \cos(2\varphi) Y_0 + \frac{1}{2} H_Z^2 Y_1 - \frac{3}{2} r ^2\chi (\cos \varphi \xi - \sin \varphi J \xi) \\
& =  - \frac{1}{2}r^2 \chi \cos (2\varphi) Y_0 + \frac{1}{2} \left( H_Z^2 + 3 r^2 \chi \sin (2\varphi) \right)Y_1.
\end{align*}
Finally, we have the curvature normalization condition $\mathfrak{R}^S(Y_1, X_{1,2}) = 0$. Using Proposition~\ref{prop:CurvXab}, we obtain
\begin{align*}
\mathfrak{R}^S(Y_1,X_{1,2}) & =\frac{1}{2} \eul R(\sharp \eul, Y_1)X_{1,2}  + \frac{1}{2} \eul R(\sharp \eul, X_{1,2}) Y_1 + \frac{1}{2} \eul (\nabla_{Y_1 } T)(\sharp \eul, X_{1,2}) \\
& \qquad + \frac{1}{2} \eul (\nabla_{X_{1,2}} T)(\sharp \eul , Y_1) + \langle A^\sharp(Y_1), A^\sharp (X_{1,2}) \rangle_g  \\ \nonumber
&\qquad - \langle S^\sharp(Y_1), S^\sharp(X_{1,2}) \rangle_g - \dvec (S(Y_1,X_{1,2}))  + S(X_{1,2} , X_{1,2})  \\
& = \frac{1}{2} r^3 \langle Y_0 , R(Y_0, Z) Y_1 \rangle_g - \frac{1}{2} r^3\langle Y_0, (\nabla_{Y_1 } \tau)Y_0 \rangle_g \\
& \qquad + \frac{1}{4} H_Z \langle Y_0, r^2 \chi (\cos \varphi \xi - \sin \varphi J\xi) + H_Z^2 Y_1 \rangle_g \\
& \qquad  - \frac{1}{4} H_Z \langle Y_0 , r^2 \chi \cos (2\varphi) Y_0 - \left( H_Z^2 + 3 r^2 \chi \sin (2\varphi) \right)Y_1 \rangle_g \\
&\qquad  - \frac{1}{2} \dvec \left( H_Z^2 + 3 r^2 \chi \sin (2\varphi) \right) + S(X_{1,2} , X_{1,2}).
\end{align*}
By Corollary~\ref{cor:Curv}, Appendix, we have
$$\langle Y_0, R(Y_0, Z) Y_1 \rangle_g = \langle (\nabla_{Y_0} \tau) Y_0, Y_1 \rangle_g - \langle (\nabla_{Y_1} \tau) Y_0, Y_0 \rangle .$$
Furthermore
\begin{align*}
\dvec ( H_Z^2 + 3 r^2 \chi \sin(2\varphi) ) & = 2 r^2 \chi H_Z \cos (2\varphi) + 3 r^3 d\chi(Y_0) \sin(2\varphi) \\
& \qquad + 6r^2 \chi  (H_Z - r\beta(Y_0))  \cos(2 \varphi).
\end{align*}
Hence
\begin{align} \nonumber
- S(X_{1,2}, X_{1,2})   & = \frac{1}{2} r^3 \langle (\nabla_{Y_0} \tau) Y_0, Y_1 \rangle_g -  r^3\langle Y_0, (\nabla_{Y_1 } \tau)Y_0 \rangle_g + \frac{1}{4} r^2 \chi H_Z  \cos (2\varphi ) \\ \nonumber
& \qquad - \frac{1}{4} r^2H_Z \chi \cos(2\varphi)  - r^2 \chi H_Z \cos (2\varphi) - \frac{3}{2} r^3 d\chi(Y_0) \sin(2\varphi) \\ \nonumber
& \qquad - 3r^2  (H_Z - r\beta(Y_0))  \chi \cos(2 \varphi) \\ \label{SX12X12}
& = r^2 \left( 4 r \chi \beta(Y_0) -  r d\chi(Y_1) - 4\chi H_Z  \right)\cos(2\varphi) \\ \nonumber
& \qquad   - 2 r^3 \left( \chi \beta(Y_1)  + d\chi(Y_0) \right) \sin(2\varphi) .
\end{align}
In summary, the map $S$ is given by the matrix $[S]$
$$[S] = \kbordermatrix{
    & Y_0 & Y_1 & X_{1,2} \\
    Y_0 & 0 & - \frac{1}{2} H_Z & - \frac{1}{2}r^2 \chi \cos (2\varphi)  \\
    Y_1 & - \frac{1}{2} H_Z & 0 & \frac{1}{2} H_Z^2 + \frac{3}{2} r^2 \chi \sin (2\varphi)  \\
    X_{1,2} & - \frac{1}{2}r^2 \chi \cos (2\varphi) & \frac{1}{2} H_Z^2 + \frac{3}{2} r^2 \chi \sin (2\varphi) & S(X_{1,2} , X_{1,2} )
  }.$$
with $- S(X_{1,2}, X_{1,2})$ is given in \eqref{SX12X12}. Using the formula for $A$,
$$[A] = \kbordermatrix{
    & Y_0 & Y_1 & X_{1,2} \\
    Y_0 & 0 & - \frac{1}{2} H_Z & - \frac{1}{2}r^2 \chi \cos (2\varphi)  \\
    Y_1 & \frac{1}{2} H_Z & 0 & - \frac{1}{2} H_Z^2 + \frac{1}{2} r^2 \chi \sin (2\varphi)  \\
    X_{1,2} &  \frac{1}{2}r^2 \chi \cos (2\varphi) & \frac{1}{2} H_Z^2 - \frac{1}{2} r^2 \chi \sin (2\varphi) & 0
  }.$$
we have the connection $\calH^S$.

\subsubsection{Curvature}
We finally see that
\begin{align*}
\mathfrak{R}^S(Y_1,Y_1) & = \eul R(\sharp \eul, Y_1)Y_1  + \eul (\nabla_{Y_1} T)(\sharp \eul, Y_1) + | A^\sharp(Y_1) |_g^2 \\ 
&\qquad  - | S^\sharp(Y_1) |_g^2 - \dvec S(Y_1,Y_1)  + 2 S(X_{1,2} , Y_1)\\
& = r^2 \kappa   + H_Z^2 + 3 r^2 \chi \sin (2\varphi) . \\
\mathfrak{R}^S(X_{1,2},X_{1,2})  & = r \eul (\nabla_{X_{1,2}} \tau) (\sharp \eul)  + | A^\sharp(X_{1,2}) |^2_g - | S^\sharp(X_{1,2})|^2_g - \dvec S(X_{1,2},X_{1,2}) .
\end{align*}
We compute that
\begin{align*}
r \eul(\nabla_{X_{1,2}}\tau) (\sharp \eul) & = r^3 \left( r \langle Y_0, (\nabla_Z \tau) Y_0 \rangle - H_Z \langle Y_0, (\nabla_{Y_0} \tau) Y_0 \rangle \right) \\
& = r^4 d\chi(Z) \cos(2\varphi) + 2r^4 \chi \beta(Z) \sin(2 \varphi) \\
& \qquad  - r^3 d\chi(Y_0) H_Z \cos(2\varphi) - 2r^3 \chi \beta(Y_0) H_Z\sin(2 \varphi), \\
|A^\sharp(X_{1,2})|^2 & =  \frac{1}{4} r^4 \chi^2 - \frac{1}{2} r^2 \chi H_Z^2 \sin(2\varphi)  +  \frac{1}{4} H_Z^4  , \\
|S^\sharp(X_{1,2})|^2 & = \frac{1}{4} r^4 \chi^2  + 2 r^4 \chi^2 \sin^2 (2\varphi) + \frac{3}{2} r^2 \chi H_Z^2 \sin(2\varphi)  + \frac{1}{4} H_Z^4,
\end{align*}
\begin{align*}
& - \dvec S(X_{1,2}, X_{1,2}) \\
& = - 2 r^2 (H_Z - r \beta(Y_0))  \left( 4 r \chi \beta(Y_0) -  r d\chi(Y_1) - 4\chi H_Z  \right)\sin(2\varphi) \\
& \qquad   - 4 r^3 (H_Z - r \beta(Y_0)) \left( \chi \beta(Y_1)  + d\chi(Y_0) \right) \cos(2\varphi)  \\
& \qquad + r^2 \left( 4 r^2 d\chi(Y_0) \beta(Y_0) + 4 r^2 \chi (\nabla_{Y_0} \beta)(Y_0)+ 4 r  H_Z\chi \beta(Y_1) \right)\cos(2\varphi) \\ \nonumber
& \qquad + r^2 \left(  -  r^2 (\nabla_{Y_0} d\chi)(Y_1) +  r H_Z d\chi(Y_0)  \right)\cos(2\varphi) \\ \nonumber
& \qquad - r^2 \left(  4 r d\chi(Y_0) H_Z + 4r^2 \chi^2 \cos(2\varphi)  \right)\cos(2\varphi) \\ \nonumber
& \qquad   - 2 r^3 \left( r d\chi(Y_0) \beta(Y_1) - \chi H_Z \beta(Y_0)   \right) \sin(2\varphi) \\ \nonumber
& \qquad   - 2 r^3 \left(  r (\nabla_{Y_0} d\chi)(Y_0) + H_Z d\chi(Y_1) \right) \sin(2\varphi) .
\end{align*}
Summing over all of these terms, we obtain the formula for $\mathfrak{R}^S(X_{1,2}, X_{1,2})$
This completes the proof

\subsection{Proof if Theorem~\ref{th:23BoundedCurvature}} \label{sec:Proof23}
Let $x$ be any point. If $\chi(x) \neq 0$, then for any $p \in \Sigma M_x$ with $r(p) =1$ and $\sin (2\varphi(p)) = -\mathrm{sgn}(\chi(x))$,
$$\Ric(p) = -6 H_Z^2(p) |\chi(x)|  + O(H_Z(p)),$$
as $H_Z \to \pm \infty$. Hence, this implies that $k_2(x) = - \infty$. Similarly, one can show that $k_2(x) = -\infty$ if $\chi(x) = 0$ but with $d\chi|_x \neq 0$. If $\chi(x) = 0$ and $d\chi|_x =0$, it follows that $\Ric^{1,2}(p) = 0$ for any $p \in T_xM$.

Using the above fact, we deduce that $k_2$ is locally bounded if and only if $\chi$ vanishes identically which is again equivalent to $\calL_Z g =0$. Hence, $Z$ is a sub-Riemannian Killing vector field and $e^{tZ}$ is a local isometry whenever it is defined. Define $\Phi$ as the foliation along the vector field $Z$. Since $Z$ is a transverse Killing vector field, by working locally, we can assume that $\check{M} = M / \Phi$ is a smooth manifold with Riemannian metric $\check{g}$ such that the metric on $g$ is a pullback of this metric on the quotient. By our definition of the connection $\nabla$, it follows that $\kappa$ is the pull-back of the Gaussian curvature of $\check{M}$, see \cite[Section~3]{GrTh16a} for details. Locally around a point $y_0 \in \check{M}$, trivialize the fibration $\zeta: M \to \check{M}$ to $M= \check{M} \times I$, where $I$ is some open intervall around $0$ in $\mathbb{R}$. Finally, since $d\theta(v,w) = -1$ for any positively oriented orthonormal basis of $\calE$, it follows that the contact form is on the form described in Example~\ref{ex:BundleRS}.

%% file: 7Step2.tex
\section{Sub-Riemannian manifolds with fat horizontal bundles} \label{sec:Step2}
\subsection{Fat subbundles and geodesics of step 2}
Let $(M, \calE, g)$ be a horizontal distribution with $n = \dim M$ and $d_1 = \rank \calE$. The subbundle $\calE$ is called \emph{fat} if for any $x \in M$ and any vector field $X$ with values in $\calE$ with $X_x \neq 0$, we have $\calE_x + [X, \calE]|_x = T_xM$. Independently of the connection $\nabla$ chosen, we note that the map $[P_1]|_p$ is surjective on $TM/\calE$ whenever $\sharp p \neq 0$. It follows that for any $p \neq \Ann(\calE)$, we have Young diagram
$$\mathbb{Y}_p  = \mathbb{Y}(d_1, n- d_1), \qquad \mathsf{Y}_p = \mathbb{Y}(2,1),$$
and $\Sigma M = T^*M \setminus \Ann \calE$. $M$ is thus a complete constancy domain if and only if it is complete.

\subsection{Curvature with a particular choice of connection}
Let $(M, \calE, g)$ be a sub-Riemannian manifold with $\calE$ fat. Let $\bar{g}$ be any Riemannian metric taming $g$. Define $\calA = \calE^\perp$ as the orthogonal complement to $\calE$ relative to $\bar{g}$ and let $\pr_{\calE}$ and $\pr_{\calA}$ denote the respective orthogonal projections. For any section $Z \in \Gamma(\calA)$, we define a corresponding map $\tau_Z: \calE \to \calE$,
$$\langle \tau_Z X, Y \rangle_g = \frac{1}{2} (\calL_Z \pr_{\calE}^*g)(X,Y), \qquad X, Y \in \Gamma(\calE).$$
Note that $Z \mapsto \tau_Z$ is tensorial. We extend the definition of the map by defining $\tau_X Y = \tau_{\pr_{\calA} X} \pr_{\calE} Y$ for $X,Y \in \Gamma(TM)$. Let $\nabla$ be defined as in \eqref{NiceNabla}. This connection is compatible with $(\calE, g)$ but not necessarily $\bar{g}$. It preserves the splitting $TM = \calE \oplus \calA$ under parallel transport. Its torsion is given for $X,Y \in \Gamma(TM)$,
\begin{eqnarray*}
T(X,Y) &=& K(X,Y) + \tau_X Y - \tau_Y X, \\
K(X,Y) & := & - \pr_{\calA} [\pr_{\calE} X, \pr_{\calE} Y] - \pr_{\calE} [\pr_{\calA} X, \pr_{\calA} Y].
\end{eqnarray*}


\subsection{Computation of canonical decomposition}
We want to give the canonical decomposition $\bpi^* TM = \Box^{1,1} \oplus \Box^{1,2} \oplus \Box^{2,1}$ for any sub-Riemannian manifold with a fat subbundle $\calE$. For any $z \in T_xM$,  we define a map $J_z : T_xM \to T_xM$, by
$$\langle J_z u , v \rangle_g = - \langle \pr_{\calA} z, K(\pr_{\calE} u, \pr_{\calE} v) \rangle_{\bar{g}}.$$
Since we assumed that our distribution is fat, it follows that $J_{z}$ maps $\calE_x$ onto itself bijectively for any non-zero $z \in \calA_x$. We write $J_z^{-1}:T_xM \to T_xM$ for the map satisfying $J_z J_{z}^{-1} = J_{z}^{-1} J_z = \pr_{\calE_x}$.

Define $r(p) = |p|_{g^*} = |\sharp p|_g$ and let $Y_0 \in \Gamma(\pi^* \calE)$ and $Z_0 \in \Gamma(\pi^* \calA)$ be $\bar{g}$-unit vector fields such that
$$p(X) = r \langle Y_0, X\rangle_{\bar{g}} + H_{Z_0} \langle Z_0, X \rangle_{\bar{g}} , \qquad X \in \Gamma(TM) $$
We note that that $Z_0$ is not well defined for $p \in\Ann(\calA)$, however, by defining $H_{Z_0}(p) =0$ on $\Ann(\calA)$, the above formula is still valid. For convenience, we will define $Z_0|_p = 0$ whenever $p \in \Ann(\calA)$. With this notation in place, we have the following result.
\begin{proposition}
We have canonical decomposition $\bpi^* TM = \Box^{1,1} \oplus \Box^{1,2} \oplus \Box^{2,1}$, such that $\Box^{2,1}$ is the orthogonal complement of $\Box^{1,1}$ in $\calE$,
$$\Box^{1,1}_p =  \spn \{ J_z^{-1} \sharp p \, : \, z\in  \calA_{\pi(p)} \setminus 0 \}.$$
and if $\pr_j: \calE \to \Box^{j,1}$ are orthogonal projections for $j=1,2$, then
$$\Box^{2,1}_p = \left\{ \begin{subarray}{l} rz + \frac{1}{4} H_{Z_0} |z|^2_{\bar{g}} (4 \pr_2 + \pr_1) J_{Z_0} J_z^{-1}Y_0 \\
 \qquad + |z|^2_{\bar{g}} \left(  - \frac{1}{4} (3(\Phi \pr_1) + (\Phi \pr_1)^\dagger) + 2 (\Phi \pr_2)^\dagger \right)J_z^{-1}Y_0 \end{subarray} \, : \, z \in \calV_{\pi(p)} \right\}.$$
with
$$\Phi v = J_{(\nabla_{Y_0} K)(Y_0, v)} Y_0 + \frac{H_{Z_0}}{r} J_{K(J_{Z_0} Y_0 , v)} Y_0.$$
\end{proposition}

\begin{proof}
We compute
\begin{align*}
P_1 v & = - r K(\sharp Y_0, v) + r \tau_v Y_0, \\
\dvec \eul & = r H_{Z_0} (J_{Z_0} Y_0)^* + r \langle \tau Y_0, Y_0 \rangle, \\
A^\sharp v & = - \frac{1}{2} H_{Z_0} J_{Z_0} v + \frac{1}{2} r  \tau_v Y_0.
\end{align*}
It follows that
$$\Box^{2,1}_p = \ker [P_1]_p = \ker K_{\sharp p} |_{\calE}, \qquad p \in \Sigma M = T^*M \setminus \Ann(\calE),$$
with $\Box^{1,1}$ as its orthogonal complement in $\calE$,
$$\Box^{1,1}_p = \spn\{ J_z^{-1} Y_0(p) \, : \, z \in \calA_{\pi(p)} \}.$$
Define $K^{-1}_{\sharp p}: T_{\pi(p)} M \to \Box^{1,1}_p$ as the map vanishing on $\calE$ and satisfying $K_{\sharp p}^{-1} K_{\sharp p} | \calE = \pr_1 = \pr^{1,1}$. In other words
$$K^{-1}_{\sharp p} z = \frac{|z|^2_{\bar{g}}}{r(p)} J_z^{-1} Y_0|_p, \qquad z \in \calA_x, p \in T^*_xM, x \in M.$$

We observe that $u \in \calE$,
\begin{align*}
P_2 u & =  - r^2 (\nabla_{Y_0} K)(Y_0, u)  - r H_{Z_0} K(J_{Z_0} Y_0, u) - r^2 \tau_{K(Y_0, u)} Y_0, 
\end{align*}
so as a consequence, if $\pr_1 = \pr^{1,1}$ and $\pr_2 = \pr^{2,1}$, then
\begin{align*}
B_0 & =  K_{Y_0}^{-1} (r(\nabla_{Y_0} K)(Y_0, u)  + H_{Z_0} K(J_{Z_0} Y_0, u))\pr_{1} = K^{-1}_{\sharp \eul} \dvec (K_{\sharp \eul} )\pr_{1} , \\
B_+  & =   0, \\
C & =  K^{-1}_{\sharp \eul} \dvec (K_{\sharp \eul}) \pr_{2}.
\end{align*}
%
%
%
From the equation \eqref{Qmap} and \eqref{ScalE}, we have
\begin{align*}
Q &= \frac{1}{4}  (B_0 - B_0^\dagger) + C - C^\dagger+ \frac{1}{4} H_{Z_0} \pr_1 J_{Z_0} \pr_1 + \frac{1}{2} H_{Z_0} \pr_{2} J_{Z_0} \pr_{2}  , \\
S^\sharp |_{\calE} & = \frac{1}{2} ( B_0 + B_0^\dagger) - C - C^\dagger -\frac{1}{2} H_{Z_0} \pr_{2} J_{Z_0} \pr_{1} + \frac{1}{2} H_{Z_0} \pr_{1}  J_{Z_0} \pr_{2}.
\end{align*}
We get the first canonical twist function given such that for any $u \in \Box^{1,1}$,
\begin{align*}
& \wp_1u = (P_1 + Q + A^\sharp + S^\sharp)u \\
& = - r K(Y_0, u) - \frac{1}{4} H_{Z_0}(4 \pr_2 + \pr_1) J_{Z_0} u + \frac{1}{4}(3B_0 + B_0^\dagger) u - 2C^\dagger u,
\end{align*}
and inserting $u = - |z|^2_{\bar{g}} J_z^{-1} Y_0$,
\[\Box^{2,1}_p = \left\{ \begin{subarray}{c} rz + |z|^2_{\bar{g}} \left( \frac{1}{4} H_{Z_0} (4 \pr_2 + \pr_1) J_{Z_0}  - \frac{1}{4} (3B_0 + B_0^\dagger)  + 2 C^\dagger \right)J_z^{-1}Y_0 \end{subarray}  \, : \, z \in \calV_{\pi(p)} \right\}. \qedhere \]
\end{proof}

Finally, the map $S^\sharp$ is determined by the equation
\begin{align*}
0 = \wp_2 u & = (\dvec +P_1 + A^\sharp + S^\sharp) \wp_1 u+ \wp_1 Qu .
\end{align*}
The remaining part of $S$ then follow from identity $\mathfrak{R}^S(\wp_1 u , v) + \mathfrak{R}^S(u, \wp_1 v) = 0$. We will complete this computation in a special case.

\subsection{H-type manifolds}
Let $(M, \calE, g)$ be a sub-Riemannian manifold with $\calE$ a fat subbundle. Let $\bar{g}$ be taming Riemannian metric. Let $\nabla$ and $J$ be defined as above.
\begin{definition}
We say that $(M, \calE, \bar{g})$ is H-type if $J^2_{z} = - | z|_{\bar{g}}^2 \pr_{\calE}$ for any $z \in \calA$. Equivalently, it is $H$-type if for any non-zero $z \in \calA$,
$$J_z^{-1} = - \frac{1}{|z|_{\bar{g}}^2} J_z.$$
For $H$ type manifolds, we further make the following definitions.
\begin{enumerate}[\rm (a)]
\item It satisfies the $J^2$-condition if for any $z, z' \in \calA_x$, $v \in \calE_x$, $x \in M$, we have  $J_z J_{z'} v= J_{z''}v$ for some $z'' \in \calA_x$. We remark that $z''$ may depend on $v$ as well as $z$ and $z'$.
\item It is horizontally parallel if $\nabla_v J = 0$ for any $v \in \calE$.
\item It is said to have a horizontally parallel Clifford structure it is horizontally parallel and for some $\kappa_{\calA} \geq 0$,
$$(\nabla_{z_1} J)_{z_2} =  \kappa_{\calA}( J_{z_1} J_{z_2} + \langle z_1, z_2 \rangle_{\bar{g}} \pr_{\calE}), \qquad z_1, z_2 \in \calA.$$
\end{enumerate}
\end{definition}
\begin{definition}
Assume that $\calA = \calE^\perp$ is integrable with a corresponding foliation $\Phi$. Then $\Phi$ is called a totally geodesic foliation if
$$(\calL_X \bar{g})(Z, Z) = 0, \qquad (\calL_Z \bar{g})(X, X) = 0, \qquad X \in \Gamma(\calE), Y \in \Gamma(\calA).$$
or equivalently if $\tau = 0$ and that $\nabla$ is compatible with $\bar{g}$.
\end{definition}

We consider the case when $\calA$ correspond to a totally geodesic foliation, so in particular $T = K$. Furthermore, we want $(M, \calE, \bar{g})$ to be H-type with horizontally parallel Clifford structure satisfying the $J^2$-condition. For such manifolds and for any $v \in \calE_x$, $x \in M$, introduce an algebra $\mathbb{A}_v = \mathbb{R} \mathbf{1} \oplus \calA_{x}$ with unit $\mathbf{1}$ and with multiplication defined such that if we introduce the convention $J_{\mathbf{1}} := \pr_{\calE_x}$, then
$$J_{z_1} J_{z_2} v= J_{z_1 \cdot z_2}v, \qquad z_1, z_1 \in \mathbb{A}_x.$$
Note that since $\mathbb{A}_x$ is a division field, it is isomorphic to the complex numbers $\mathbb{C}$, the quaternions $\mathbb{H}$ or the octonions $\mathbb{O}$, and hence $\calA$ has rank $1$, $3$ or $7$. One can then verify that the rank of $\calE$ has to be multiple of respectively $2$, $4$ or $8$. For a full classification of such manifolds, see \cite{BGMR18,BGMR19}.

Since the case of $\rank \calA =1$ is a special case \cite{ABR17}, we will only consider the case $\rank \calA$ equal to $3$ or $8$. For such manifolds, we have the following result.
\begin{proposition}
Define $X_0^{1,2}= r Z_0 - H_{Z_0}Y_0$ and $Y_1 = J_{Z_0} Y_0$. For any $z \in \calA_x$ with $\langle Z_0, z \rangle_{\bar{g}} =0$, we define
$$Y_z = J_z Y_0, \qquad X^{1,2}_z|_p = r z- \frac{3}{4} H_{Z_0}(p) J_{Z_0|_p \cdot z} Y_0|_p.$$
Then $\bpi^* TM = \Box^{1,1} \oplus \Box^{1,2} \oplus \Box^{2,1}$, with
$$\Box^{1,1}_p = \{ J_z \sharp p \, : \, z \in \calA_{\pi(p)} \}, \quad \Box^{1,2}_p = \spn \{ X_0^{1,2}|_p, X^{1,2}_z|_p \, : \, z\in \calA_{\pi(z)}, \langle Z_0|_p, z\rangle_{\bar{g}} =0 \},$$
and with $\Box^{2,1}$ being the orthogonal complement of $\Box^{1,1}$ in $\calE$. If $\pr^{1,2}|_p: T_{\pi(p)}M \to \Box_p^{1,2}$ is the corresponding projection, we define covectors
$$\alpha_{v}|_p = \langle v, (\id - \pr^{1,2}) \, \cdot \, \rangle_{\bar{g}}, \qquad \beta_{z}|_p = \frac{1}{r} \langle z, \pr^{1,2} \, \cdot \, \rangle_{\bar{g}}, \qquad v \in \calE_{\pi(p)}, z \in \calA_{\pi(p)} .$$
The decomposition $T\Sigma M = \calH^S \oplus \ker \bpi_*$ is given by horizontal lifts with $v \in \Box^{2,1}$, $z \in \calA$, $\langle Y_0, v \rangle_g = 0$, $\langle Z_0, z \rangle_{\bar{g}} = 0$,
\begin{align*}
h^S Y_0 &= h Y_0 - H_{Z_0} \vl \alpha_{Y_1}, \\
h^S Y_1 & =  h Y_1 - \frac{1}{2} H_{Z_0}^2 \vl \beta_{Z_0} , \\
h^S Y_z & =  h Y_z + \frac{1}{2} H_{Z_0} \vl \alpha_{Y_{Z_0 \cdot z}} - \frac{9}{16} H_{Z_0}^2 \vl \beta_z, \\
h^S v & = h v + \frac{1}{2} H_{Z_0} \vl \alpha_{ J_{Z_0} v}, \\
h^S X_0^{1,2} & = h X^{1,2}_0 - H_{Z_0}^2 \vl \alpha_{Y_1}- \frac{3}{8} r^2 H_{Z_0} \vl \beta_{K(Y_0, Z_0 \cdot K(Y_0, R(Y_0, Y_1) Y_0))} ,\\
h^S X^{1,2}_z & = h X^{1,2}_z  - \frac{3}{16} H_{Z_0}^2  \vl \alpha_{Y_z} + \frac{3}{8} r^2 H_{Z_0}  \vl \beta_{K(Y_0,  R(Y_0, Y_{Z_0 \cdot z}) Y_0) - Z_0 \cdot K( Y_0,  R(Y_0, Y_z) Y_0) } .
\end{align*}
\end{proposition}

\begin{proof} We will consider all of our computations for $p \in T^*M \setminus \{ \Ann(\calE) \cup \Ann(\calA)\}$ and we leave the special case $p \in \Ann(\calA)$ to the reader. From our assumptions on~$M$, we have
\begin{align*}
\dvec \eul  &=  r H_{Z_0} (J_{Z_0} Y_0)^*,  &A^\sharp v  &= -\frac{1}{2}  H_{Z_0} J_{Z_0} v, \\
P_1 v  & =  -  r K(Y_0, v) , & \ P_2 v &= - r H_{Z_0} K(J_{Z_0} Y_0, v). \end{align*} 
We note that 
$$\dvec Y_0 = H_{Z_0} J_{Z_0} Y_0, \qquad \dvec Z_0 =  0, \qquad \dvec J = 0, \qquad \dvec H_{Z_0} = 0.$$
We observe that from the H-type assumption $\langle J_{z_1} v, J_{z_2} v \rangle = \langle z_1, z_2 \rangle |v|_g^2$, $v \in \mathcal{E}_x$, $z_1,z_2 \in \mathbb{A}_v$.
Hence, if $v \in \calE_x$ is a unit vector, we have for $z, z' \in \calA_x$,
$$K(J_z v, J_{z'} v) = - z \cdot z' - \langle z, z'\rangle \mathbf{1} \qquad K(J_zv , v) = z.$$
It follows that $\Box^{1,1}_p = \{ J_z Y_0 \, : \, z \in \calA_{\pi(p)} \}$. Write decomposition
$$\Box^{1,1}_p  = \mathbb{R} J_{Z_0} Y_0|_p \oplus_\perp \boxtimes^{1,1}_p, \qquad \Box^{2,1}_p = \mathbb{R} Y_0|_p \oplus_{\perp} \boxtimes^{2,1}_p ,$$
and observe that
\begin{align*}
\boxtimes^{1,1}_p & = \spn \{ J_z Y_0 \, : \, z \in \calA_{\pi(p)}, \langle Z_0, z \rangle = 0 \}, \\
\boxtimes^{2,1}_p & = \{ v \in \calE_{\pi(p)} \, : \, K(J_z Y_0|_p, v) = 0 \text{ for any $z \in \mathbb{A}_{Y_0}$} \}. 
\end{align*}

From the equation of $P_2$, observe $B_+ =0$, $B_0(\Box^{2,1}) =0$, $C(\Box^{1,1}) =0$, and furthermore, for any $v \in \boxtimes^{2,1}$,
$$B_0 Y_1 = 0, \qquad B_0 Y_z = - H_{Z_0} Y_{Z_0 \cdot z}, \qquad  C Y_0 = H_{Z_0} Y_1, \qquad Cv =0.$$
As a consequence,
$$Q Y_1 = - H_{Z_0} Y_0, \quad Q Y_z = - \frac{1}{4} H_{Z_0} Y_{Z_0 \cdot z}, \quad Q Y_0 = H_{Z_0} Y_1, \quad Qv =\frac{1}{2} H_{Z_0} J_{Z_0} v,$$
and
$$S^\sharp Y_1 = - \frac{1}{2} H_{Z_0} Y_0, \qquad S^\sharp Y_z = 0, \qquad S^\sharp Y_0 = - \frac{1}{2} H_{Z_0} Y_1, \qquad S^\sharp v =0.$$
The first canonical twist function gives us,
$$\wp_1 Y_1 = r Z_0 - H_{Z_0} Y_0 = X^{1,2}_0, \qquad \wp_0 Y_z  = rz - \frac{3}{4} H_{Z_0} Y_{Z_0 \cdot z} = X^{1,2}_z, $$

Write $\pr_\boxtimes$ for the orthogonal projection from $\calE$ to $\boxtimes^{1,1} \oplus \boxtimes^{2,1}$. We then observe that $\dvec \pr_\boxtimes = 0$. Hence, since $\wp_1 = - rK(Y_0, \, \cdot \, ) - H_{Z_0} Y_1^* \otimes Y_0 - \frac{3}{4} H_{Z_0} J_{Z_0} \pr_{\boxtimes} \pr_1$, we have
\begin{align*}
\dvec  \wp_1 & = - r H_{Z_0} K(J_{Z_0} Y_0, \, \cdot \,) + H_{Z_0}^2 (Y_0^* \otimes Y_0 - Y_1^* \otimes Y_1 ) .
\end{align*}
As a consequence,
\begin{align*}
0 = \wp_2 Y_1 &= (\dvec \wp_1) Y_1 + (P_1 + A^\sharp + S^\sharp)X^{1,2}_0 + \wp_1 Q Y_1 \\
& = - H_{Z_0}^2 Y_1 + \frac{1}{2} H_{Z_0}^2 Y_1 + S^\sharp X^{1,2}_0,
\end{align*}
and
\begin{align*}
0  = \wp_2 Y_z  &= rH_{Z_0} (Z_0 \cdot z)- \frac{3}{4} r H_{Z_0} (Z_0 \cdot z) - \frac{3}{8}H_{Z_0}^2 Y_z \\
& \qquad+ S^\sharp X^{1,2}_z - \frac{1}{4} H_{Z_0} (r Z_0 \cdot z + \frac{3}{4} H_{Z_0} Y_z) .\end{align*}
or in other words
$$S^\sharp X_1^{1,2} = \frac{1}{2} H_{Z_0}^2 Y_1, \qquad S^\sharp X_z^{1,2} =  \frac{9}{16} H_{Z_0}^2 Y_z. $$

We finally use curvature restrictions to determine $S$. We first note that from Corollary~\ref{cor:Curv}, we have that for any $X,Y \in \Gamma(\calE)$, $Z,W \in \Gamma(\calA)$
$$R(X, Z)Y =0, \qquad R(X,Z) W =0, $$
$$R(X,Y) Z= (\nabla_Z K)(X,Y)  = - \kappa_{\calA} (K(J_Z X,Y) - Z \langle X,Y \rangle_g).$$
and since the torsion only has values in $\calA$, for $X_i \in \Gamma(\calE)$, $i=1,2,3,4$,
$$\langle R(X_1, X_2) X_3,X_4 \rangle_g = \langle R(X_3, X_4) X_1,X_2\rangle_g.$$
These give us identities
\begin{align*}
& 0 = \mathfrak{R}^S( Y_1 , X_0^{1,2}) \\
& =\frac{1}{2} r \eul R(Y_0, Y_1)X_0^{1,2}  + \frac{1}{2} r \eul R(Y_0, X_0^{1,2}) Y_1 + \frac{1}{2} r \eul (\nabla_{X^{1,2}_0} T)(Y_0, Y_1) \\
& \qquad  + \langle A^\sharp(Y_1), A^\sharp (X^{1,2}_0) \rangle_g - \langle S^\sharp(Y_1), S^\sharp(X^{1,2}_1) \rangle_g - \frac{d}{dt} S(Y_1 ,X^{1,2}_0) + S(X^{1,2}_0 , X^{1,2}_0)  \\
& = S(X^{1,2}_0 , X^{1,2}_0)  ,
\end{align*}
and
\begin{align*}
& 0 =\mathfrak{R}^S(Y_z,X_z^{1,2}) \\
& = \frac{1}{2} r\eul R(Y_0, Y_z) X^{1,2}_z  + \frac{1}{2} r \eul R(Y_0, X^{1,2}_z) Y_z + \frac{1}{2} r \eul (\nabla_{X_z^{1,2}} T)(Y_0, Y_z) \\
& \qquad + \langle A^\sharp(Y_z), A^\sharp (X_z^{1,2}) \rangle_g - \langle S^\sharp(Y_z), S^\sharp(X_z^{1,2}) \rangle_g - \dvec S(Y_z, X^{1,2}_z) + S(X^{1,2}_z , X^{1,2}_z) \\
& =  -  \frac{3}{4} r^2 H_{Z_0} \langle Y_0, R(Y_0, Y_z) Y_{Z_0 \cdot z} \rangle_g  - \frac{1}{2} r^2 H_{Z_0} \langle  (\nabla_{z} J)_{Z_0} Y_0, Y_z \rangle \\
& \qquad   + \frac{9}{16} \dvec H_{Z_0}^2 + S(X^{1,2}_z , X^{1,2}_z) \\
& =  - \frac{3}{4} r^2 H_{Z_0} \langle Y_0, R(Y_0, Y_{Z_0 \cdot z}) Y_z \rangle_g  + S(X^{1,2}_z , X^{1,2}_z) \\
& =  - \frac{3}{4} r^2 H_{Z_0} \langle Y_0, R(Y_0, Y_{Z_0 \cdot z}) Y_z \rangle_g  + S(X^{1,2}_z , X^{1,2}_z).
\end{align*}
We next observe the relation,
\begin{align*}
& 0 =\mathfrak{R}^S(Y_1,X^{1,2}_z) + \mathfrak{R}^S(X_0^{1,2} , Y_z) \\
& = \frac{1}{2} r^2 H_{Z_0} \langle Z_0, R(Y_0, Y_1) z \rangle_g - \frac{3}{4} H_{Z_0} r^2 \langle Y_0, R(Y_0, Y_1 ) Y_{Z_0 \cdot z} \rangle_g  \\
& \qquad + \frac{1}{2} r^2 H_{Z_0} \langle (\nabla_{z} J)_{Z_0} Y_0, Y_1 \rangle_g + 2S(X_0^{1,2}, X_z^{1,2}) \\ 
& =  - \frac{3}{4} r^2 H_{Z_0} \langle Y_0, R(Y_0, Y_1) Y_{Z_0 \cdot z} \rangle_g  + 2S(X_0^{1,2} , X_z^{1,2}) .
\end{align*}
Finally, if $z, Z \in \calA$ with $\langle Z_0, Z \rangle_{\bar{g}} = \langle Z_0, z \rangle_{\bar{g}} = \langle Z, z \rangle_{\bar{g}} =0$,
\begin{align*}
& 0 =\mathfrak{R}^S(Y_z ,X^{1,2}_Z) + \mathfrak{R}^S(X_z^{1,2}, Y_Z)  \\
& = \frac{1}{2} r^2 H_{Z_0} \langle Z_0, R(Y_0, Y_z) Z \rangle_{\bar{g}} - \frac{3}{4} r^2 H_{Z_0} \langle Y_0 , R(Y_0, Y_z) Y_{Z_0 \cdot Z} \rangle_{\bar{g}}  \\
&\qquad + \frac{1}{2} r^2 H_{Z_0} \langle Z_0, R(Y_0, Y_Z) z \rangle_{\bar{g}} - \frac{3}{4} r^2 H_{Z_0} \langle Y_0 , R(Y_0, Y_Z) Y_{Z_0 \cdot z} \rangle_{\bar{g}}  \\
& \qquad - \frac{1}{2} r^2 H_{Z_0} \langle (\nabla_{Z }J)_{Z_0} Y_0, Y_z \rangle - \frac{1}{2} r^2 H_{Z_0} \langle (\nabla_{z} J)_{Z_0} Y_0, Y_Z \rangle \\
& \qquad  - \frac{3}{16} H_{Z_0}^3 ( \langle J_{z} Y_0, J_{Z_0 \cdot z} \rangle_g +  \langle J_{z} Y_0, J_{Z_0 \cdot z} \rangle_g) + 2S(X_z^{1,2}, X_Z^{1,2})  \\
& =- \frac{3}{4} r^2 H_{Z_0} \left( \langle R(Y_0, Y_z) Y_{Z_0 \cdot Z} , Y_0 \rangle_g + \langle R(Y_0, Y_Z) Y_{Z_0 \cdot z} , Y_0 \rangle_g \right) + 2S(X_z^{1,2}, X_Z^{1,2}).
\end{align*}
Using that
\begin{align*}
&  \langle R(Y_0, Y_z) Y_{Z_0 \cdot Z} , Y_0 \rangle_g + \langle R(Y_0, Y_Z) Y_{Z_0 \cdot z} , Y_0 \rangle_g  \\
& =   \langle Z, K(Y_0,  R(Y_0, Y_{Z_0 \cdot z}) Y_0) - Z_0 \cdot K( Y_0,  R(Y_0, Y_z) Y_0)  \rangle .
\end{align*} 
and that $A(v,w) = -H_{Z_0} \langle J_{Z_0} v, w \rangle_{\bar{g}}$, the proof is completed. \end{proof}

We will also present the Ricci curvatures in this case. Write $d_1 = \rank \calE$ and $d_2 = \rank \calA$. For any $p \in T^*M \setminus (\Ann(\calE) \cup \Ann(\calA))$, we define
$$\kappa_\calE(p) = \sum_{j=1} \langle Y_0, R(Y_0, J_{Z_0} Y_0) J_{Z_0} Y_0 \rangle_{g}(p).$$
\begin{proposition} \label{prop:RicHType}
The Ricci curvature $\Ric = (\Ric^{\mathsf{a},\mathsf{b}})_{(\mathsf{a},\mathsf{b}) \in \mathsf{Y}}$ is given by
$$\Ric^{2,1} = \frac{1}{2} (d_1 - d_2 - 1) (r^2 \kappa_{\calA} + \frac{1}{2} H_{Z_0}^2), \qquad \Ric^{1,1} = \frac{5 d_2 -3}{2} r^2 \kappa_{\calA}  + \frac{11d_2+7}{8}  H_{Z_0}^2,$$
$$\Ric^{1,2} = \frac{3}{32} H_{Z_0}^2 \left( 7 r^2 \kappa_{\calA} + \frac{3}{2} r^2 (\kappa_{\calA} - \kappa_{\calE} ) -  \frac{15}{8} H_{Z_0}^2 \right) .$$
\end{proposition}

We will first need the following lemma, which is obtained by a modification of the proof of \cite[Theorem~3.16]{BGMR18}.
\begin{lemma} \label{lemma:AverageHtype}
Write $d_2 = \rank \calA$. For any $x \in M$, let $\calZ \subseteq \calE_x$ be a subspace such that $J_z \calZ \subseteq \calZ$ for any $z \in \calA_x$. If $v \in \calZ$, then
$$\tr_{\calZ} \langle R(\times ,v)v , \times \rangle_g = \kappa_{\calA} \left( \frac{1}{2}  \rank \calZ + 2 (d_2-1) \right) |v|_g^2. $$
If $v \perp \calZ$, then $\tr_{\calZ} \langle R(\times ,v)v , \times \rangle_g = \frac{1}{2} \kappa_{\calA} (\rank \calZ) |v|_g^2$.
\end{lemma}
We remark that by the symmetry of $(v,w) \mapsto \tr_{\calZ} \langle R(\times ,v)w , \times \rangle_g$, this map is completely determined by the above result.
\begin{proof}
Without loss of generality, we may assume that $v$ is a unit vector. Let $z_1, z_2 \in \calA_x$ two orthogonal unit elements. We will first use the following identity
\begin{align*}
& \tr_{\calZ} \langle R(\times, v) v, \times \rangle_g = - \tr_{\calZ} \langle R(\times, v) J_{z_1}^2 v, \times \rangle_g \\
& = - \tr_{\calZ} \langle [ R(\times, v) , J_{z_1} ] J_{z_1} v, \times \rangle_g + \tr_{\calZ} \langle R(\times, v)  J_{z_1} v, J_{z_1} \times \rangle_g ,\end{align*}
and
\begin{align*} 
& \tr_{\calZ} \langle R(\times, v)  J_{z_1} v, J_{z_1} \times \rangle_g \\
& =  \frac{1}{2} \tr_{\calZ} \langle R(\times, v)  J_{z_1} v, J_{z_1} \times \rangle_g -  \frac{1}{2} \tr_{\calZ} \langle  R(J_{z_1} \times, v)  J_{z_1} v,  \times \rangle_g \\
& \stackrel{\text{Bianchi}}{=}- \frac{1}{2} \tr_{\calZ} \langle J_{z_1} v,   R(\times, J_{z_1} \times) v \rangle_g = -  \frac{1}{2} \tr_{\calZ} \langle R(v, J_{z_1} v) \times, J_{z_1} \times \rangle_g .
\end{align*}
Furthermore,
\begin{align*}
& 2 \tr_{\calZ} \langle R(v, J_{z_1} v) \times, J_{z_1} \times \rangle_g \\
& =  - \tr_{\calZ} \langle R(v, J_{z_1} v) \times, J_{z_2}^2 J_{z_1} \times \rangle_g + \tr_{\calZ} \langle R(v, J_{z_1} v) J_{z_2} \times, J_{z_1} J_{z_2} \times \rangle_g \\
& =  \tr_{\calZ} \langle [R(v, J_{z_1} v) ,J_{z_2}] \times, J_{z_1} J_{z_2} \times \rangle_g .
\end{align*}
Hence, we can complete the proof by finding a formula for $[R(v,w),J_z]$.

Observe that since we have a horizontal parallel Clifford structure, we have that for any $X , Y\in \Gamma(\calE)$, $Z \in \Gamma(\calA)$,
$$R(X,Y)Z = (\nabla_Z T)(X,Y) = -\kappa_{\calA} (K(J_Z X, Y)  -Z \langle X,Y\rangle_g)$$
and
\begin{align*}
& (R(X,Y) J)_Z = (\nabla_{T(X,Y)} J)_Z = \kappa_{\calA} J_{T(X,Y) \cdot Z + \langle T(X,Y), Z \rangle_g \mathbf{1}} \\
& = [R(X,Y) , J_Z] - \kappa_{\calA} J_{T(J_Z X, Y)  -Z \langle X,Y\rangle_g}  .
\end{align*}
It follows that
\begin{align*}
[R(X,Y) , J_Z] & = 2 \kappa_{\calA} J_{T(J_Z X, Y)  -Z \langle X,Y\rangle_g}  .
\end{align*}
In particular, if $w \in (\{ J_z v \, : \, z \in \mathbb{A}_x \})^\perp$, then
$$[R(v, J_{z_i} v) ,J_{z_j}] = 2\kappa_{\calA} J_{z_j \cdot z_i + \langle z_j, z_i \rangle_{\bar{g}} \mathbf{1}}, \qquad [R(v, w) ,J_{z_2}] = 0.$$

In conclusion, if $v$ is in $\calZ$, then
\begin{align*}
& \tr_{\calZ} \langle R(\times, v) v, \times \rangle_g  = \tr_{\calZ} \langle [ R(v, \times) , J_{z_1} ] J_{z_1} v, \times \rangle_g - \frac{1}{4}  \tr_{\calZ} \langle [R(v, J_{z_1} v) ,J_{z_2}] \times, J_{z_1} J_{z_2} \times \rangle_g  \\
& = \tr_{\calA} \langle [ R(v, J_{\times} v) , J_{z_1} ] J_{z_1} v, J_{\times} v \rangle_g +  \frac{1}{2} \kappa_{\calA} \tr_{\calZ} \langle J_{z_1} J_{z_2} \times, J_{z_1} J_{z_2} \times \rangle_g  \\
& = \kappa_{\calA} \left( 2 (d_2 -1) +  \frac{1}{2} \rank \calZ \right)
\end{align*}
In a similar way, we can show the result for $v$ orthogonal to $\calZ$.
\end{proof}

\begin{proof}[Proof of Proposition~\ref{prop:RicHType}]
 Recall that $\Box^{1,1}$, $\Box^{1,2}$ and $\Box^{2,1}$ have ranks respectively $d_2$, $d_2$ and $d_1 -d_2$.
We have
\begin{align*}
\Ric^{2,1}  & = r^2 \tr_{\Box^{2,1}} \langle R(Y_0, \times)\times , Y_0 \rangle_g  + \frac{1}{4} (d_1- d_2 -1) H_{Z_0}^2   \\
& = \frac{1}{2} (d_1 - d_2 - 1 ) (r^2 \kappa_{\calA} + \frac{1}{2} H_{Z_0}^2),
\end{align*}
\begin{align*}
\Ric^{1,1} &  = r^2 \tr_{\Box^{1,1}} \langle R(Y_0, \times)\times , Y_0 \rangle_g + \frac{d_2}{4} H_{Z_0}^2 - \frac{1}{4} H_{Z_0}^2 + H_{Z_0} + \frac{9}{8} H_{Z_0}^2 (d_2-1)    \\
&  = r^2 \tr_{\Box^{1,1} \oplus \mathbb{R} Y_0} \langle R(Y_0, \times)\times , Y_0 \rangle_g +  \frac{11d_2+7}{8}  H_{Z_0}^2 \\
&  = r^2 \kappa_{\calA} ( \frac{1}{2} (d_2+1) + 2 (d_2-1) ) + \frac{11d_2+7}{8}  H_{Z_0}^2 ,
\end{align*} 
\begin{align*} 
\Ric^{1,2} & =  \tr_{\calA \cap Z_0^\perp} r \eul R(Y_0, X^{1,2}_\times) X_{\times}^{1,2} + \frac{1}{4} H_{Z_0}^4 + \frac{1}{4} \frac{9}{16} (d_2-1) H_{Z_0}^4  \\
& \qquad - \frac{1}{4} H_{Z_0}^4 -  \frac{1}{4} \frac{81}{64} (d_2-1) H_{Z_0}^4 - \frac{3}{4} r^2 H_{Z_0} \dvec \left( \tr_{\calA \cap Z_0^\perp} \langle Y_0, R(Y_0, Y_{Z_0 \cdot \times}) Y_\times \rangle_g  \right) \\
& =  -\frac{3}{4} r^2 H_{Z_0}^2 \tr_{\calA \cap Z_0^\perp}  \langle Z_0, R(Y_0, Y_{Z_0 \cdot \times}) \times \rangle_{\bar{g}} + \frac{9}{16}r^2 H_{Z_0}^2  \tr_{\calA \cap Z_0^\perp} \langle Y_0, R(Y_0, Y_\times) Y_\times \rangle_g \\
& \qquad -  \frac{1}{4} \frac{45}{64} (d_2-1) H_{Z_0}^4 - \frac{3}{4} r^2 H_{Z_0} \dvec \left( \tr_{\calA \cap Z_0^\perp} \langle Y_0, R(Y_0, Y_{Z_0 \cdot \times} )Y_\times \rangle_g  \right). \end{align*} 
We observe that
\begin{align*}
 \tr_{\calA \cap Z_0^\perp}  \langle Z_0, R(Y_0, Y_{Z_0 \cdot \times}) \times \rangle_{\bar{g}} & =  \kappa_{\calA} \tr_{\calA \cap Z_0^\perp}  \langle Y_{Z_0 \cdot \times}, Y_{Z_0 \cdot \times} \rangle_{\bar{g}} = \kappa_{\calA} (d_2-1), \\
 \tr_{\calA \cap Z_0^\perp} \langle Y_0, R(Y_0, Y_\times) Y_\times \rangle_g & = 2(d_2-1)\kappa_{\calA} + \frac{1}{2} \kappa_{\calA}(d_2+1) - \langle Y_0, R(Y_0, Y_1) Y_1 \rangle_g,
\end{align*}
 and finally
\begin{align*}
& \tr_{\calA \cap Z_0^\perp} \langle Y_0, R(Y_0, Y_{Z_0 \cdot \times} )Y_\times \rangle_g \\
& = - \tr_{\calA \cap Z_0^\perp} \langle J_{Z_0}Y_0, R(Y_0, Y_{ \times} )Y_\times \rangle_g + \tr_{\calA \cap Z_0^\perp} \langle Y_0, [R(Y_0, Y_{\times} ), J_{Z_0}] Y_{ \times} \rangle_g \\
& = - \tr_{\Box^{1,1} \oplus \mathbb{R} Y_0} \langle J_{Z_0}Y_0, R(Y_0, \times )\times \rangle_g + 2 \kappa_{\calA}\tr_{\calA \cap Z_0^\perp} \langle Y_0, J_{T(J_{Z_0} Y_0, Y_\times) } Y_{ \times} \rangle_g   =0.
\end{align*}
The result follows.
\end{proof}


%% file: ModelStep2.tex
\section{Step 2 model spaces in the sense of isometries} \label{sec:Model}
We consider the following spaces. Let $(M, \calE, g)$ be a sub-Riemannian manifold satisfying the following properties.
\begin{enumerate}[\rm (i)]
\item $(M, \calE, g)$ is complete and simply connected.
\item The horizontal bundle $\calE$ is step~$2$, so $\calE + [\calE,\calE] = TM$.
\item For every linear isometry $q: \calE_x \to \calE_y$, there is an isometry $f: (M,\calE, g) \to (M, \calE, g)$ such that $f_*|_{\calE_x} =q$.
\end{enumerate}
From \cite{Gro16}, if these conditions are satisfied, then $M$ has constant growth vector $\mathfrak{G} = (d_1, \frac{1}{2} d_1(d_1+1) )$ and has the structure of a simply connected Lie group with an invariant sub-Riemannian structure. We will hence write $M =G$ for the remainder of this section. Each such sub-Riemannian manifold $(G, \calE,g)$ is uniquely determined by a parameter $\kappa \in \mathbb{R}$, in the sense that its Lie algebra $\mathfrak{g} = \mathfrak{g}_\kappa$ isomorphic to the vector space $\mathbb{R}^{d_1} \times \mathfrak{so}(d_1)$ with brackets
$$[(x,0) , (y,0) ] = (0, y x^t - x y^t) =: (0, x \wedge y), \qquad [(0, X), (x,0)] = (\kappa X x,0),$$
$$[(0,X), (0,Y) ] = (0, [X,Y]),\qquad x,y \in \mathbb{R}^{d_1}, X,Y \in \so(d_1),$$
where $x^t$ is the transpose of $x$. With this identification, $(\calE, g)$ is given by the left translation of $\mathfrak{e} = \{ (x,0) \in \mathfrak{g} \, : \, x \in \mathbb{R}^{d_1} \}$ with the standard inner product of~$\mathbb{R}^{d_1}$. We note that up to scaling, $G$ is then isomorphic as Lie group to the free nilpotent group of step 2 for $\kappa =0$, the universal cover group of $\SO(d_1+1)$ for $\kappa =1$ or the universal cover group of $\SO(d_1,1)$ for $\kappa= -1$. See \cite{Gro16} for more details.

\subsection{Computation of connection}
 Whenever there is no confusion, we will write $(x, X) = x+X$, using lower case letters for elements in $\mathbb{R}^{d_1}$ and capital letters for $\mathfrak{so}(d_1)$. We will use the same symbol for elements in~$\mathfrak{g}$ and their corresponding left invariant vector fields. Define an inner product on~$\mathfrak{g}$ by
$$\langle x + X, y + Y \rangle  = \langle x, y \rangle + \langle X, Y \rangle := \langle x, y \rangle - \frac{1}{2} \tr XY.$$
In other words, if $e_1, \dots, e_{d_1}$ is the standard basis of $\mathbb{R}^{d_1}$, then $\{ e_k, e_i \wedge e_j \, : \, 1 \leq k \leq d_1, 1 \leq i < j \leq d_1\}$ is an orthonormal basis. We note the properties of this inner product
$$\langle X, x \wedge y \rangle = \langle Xx , y \rangle, \qquad \langle [X, Y_1], Y_2 \rangle = - \langle Y_1, [X, Y_2 ] \rangle.$$
Extend this metric to a taming Riemannian metric $\bar{g}$ by left translation.

Define functions $(\psi, \Psi) : T^*M \to \mathfrak{m}$ by
$$p(x +X)|_{\pi(p)} = \langle \psi(p), x \rangle +  \langle \Psi(p), X \rangle, \qquad p \in T^*M.$$
Let $\nabla$ be the connection such that all left invariant vector fields are parallel. The torsion is then given by $T(x +X,y +Y) = -[x+X,y +Y]$. For this connection, we will have
$$R = 0, \qquad \text{and} \qquad \nabla T = 0.$$
From the formula of the torsion, we observe the following identities
\begin{align*}
P_1 (x+X) & = \psi \wedge x -\kappa (X \psi) ,\\
\dvec \eul (x+X) & = \langle \Psi ,  \psi \wedge x \rangle =  \langle \Psi \psi, x \rangle, \\
A(x +X, y + Y) & =-  \frac{1}{2} \langle \psi + \Psi, \kappa ( X y - Yx) + x\wedge y + [X,Y] \rangle \\
& =  \frac{1}{2} \langle  - \Psi x + \kappa X\psi ,  y \rangle +\frac{1}{2} \langle x \wedge \psi + [X, \Psi], Y \rangle
\end{align*}
In particular, we have
$$\dvec \psi = \Psi \psi, \qquad \dvec \Psi =0, \qquad A^\sharp(x+X) = \frac{1}{2} (- \Psi x +\kappa X \psi).$$
and consequently, for $k \geq 0$, then $\dvec | \Psi^k \psi | =0$.

We observe that
$$\Box^{\mathsf{d}_1,1}  = \ker [P_1] = \spn\{ \psi\}.$$
Furthermore,
$$P_2 x =  \Psi \psi \wedge x \mod \calE.$$
and iteratively $P_k x = \Psi^{k-1} \psi \wedge x \mod \calE$. It follows that the set with maximal Young diagram is
$$\Sigma G = \{ p \in T^* G \, : \, \psi(p), \Psi(p) \psi(p), \dots, \Psi^{d_1-2}(p) \psi(p)  \text{ are linearly independent}\},$$
with Young diagram and reduced Young diagram $\mathbb{Y} = \mathsf{Y}= \mathbb{Y}(d_1, d_1 -1, \dots, 2,1)$.
\begin{proposition}
The sub-Riemannian model space $(M,\calE, g)$ is a complete constancy domains.
\end{proposition}
\begin{proof}
We will show that the restriction of the Hamiltonian to $\Sigma G$ is complete. For any $p \in T^*M$, define an element in $\eta|_p\in\wedge^{d_1-1} \mathfrak{e}$ by
$$\eta_p = \psi(p) \wedge \Psi(p) \psi(p) \wedge \cdots \wedge \Psi^{d_1-2}(p) \psi(p).$$
Define $y|_p$ as the orthogonal complement to $\spn\{ \psi|(p) , \Psi(p) \psi(p), \cdots , \Psi^{d_2-2}(p) \psi(p)\}$. Then since $\langle \Psi^{d_1 -1} \phi, \Psi^{d_1 -1} y \rangle_g =0$, we have
$$\dvec \eta = \psi \wedge \Psi \psi \wedge \cdots \wedge \Psi^{d_1-1} \psi= \langle \Psi^{d_1-1} \psi, y \rangle \psi \wedge \Psi \psi \wedge \cdots \wedge y.$$
In particular, $\dvec |\eta_p| =0$, so the set $\Sigma G = \{ |\eta_p| \neq 0\}$ is preserved under the Hamiltonian flow.
\end{proof}

Define $y_j: \Sigma G \to \mathfrak{e}$, $0 \leq j \leq d_1 - 2$, such that $y_0, \dots, y_k$ is an orthonormal basis of $\spn\{ \psi, \Psi \psi , \dots, \Psi^k \psi \}$ and define $y_{d_1-1}$ as the orthogonal complement of $\spn\{\psi, \Psi \psi , \dots, \Psi^{d_1-2} \psi \}$ in $\mathfrak{e}$. Then $\Box^{\mathsf{a},1} = \spn\{ y_{\mathsf{d}_1 - \mathsf{a}} \}$ and for $k \geq 2$,
$$\mathfrak{E}^k = \spn \left\{ y_l ,  y_i \wedge y_j \,  : \,\begin{array}{c} i=0, \dots, k-2 \\ l,j=0, \dots, d_1-1 \end{array} \right\}.$$
From the anti-symmetry of $\Psi$, if $0 \leq k, l \leq d_1 -1$ is such that $k$ is even and $l$ is odd, then
$$\langle y_k ,y_l \rangle =0.$$

Observe that $P_1 y_0 = 0, $ while for $1\leq k \leq d_1 -1$,
\begin{align*}
P_{k+1} y_k &= \Psi^{k} \psi \wedge y_k  \mod \calE = \langle y_{k-1}, \Psi^{k} \psi \rangle  y_{k-1} \wedge y_{k} \mod \mathfrak{E}^{k} =0.
\end{align*}
Furthermore, for $0 \leq k \leq d_1 -2$,
\begin{align*}
P_{k+2} y_k & = \Psi^{k+1} \psi \wedge y_k \mod \calE = -|\Psi^{k+1} \psi|  y_k \wedge y_{k+1} \mod \mathfrak{E}^{k+1} \\
& =  - \frac{|\Psi^{k+1} \psi|}{|\Psi^k \psi|} P_{k+1} y_{k+1} \mod \mathfrak{E}^{k+1}.
\end{align*}
In summary, we have $B = 0$, while
$$C = \sum_{k=0}^{d_1-2} \frac{|\Psi^{k+1} \psi|}{|\Psi^{k} \psi|} y_k^* \otimes y_{k+1} = \sum_{k=0}^{d_1-2} C_k, \qquad C_k :=  \frac{|\Psi^{k+1} \psi|}{|\Psi^{k} \psi|} y_k^* \otimes y_{k+1}.$$
We can hence conclude from \eqref{Qmap} and \eqref{ScalE} that $Q =C - C^\dagger$,
\begin{align*} 
S^\sharp |_{\calE} =: S^\sharp_{\calE} & =  - \sum_{k =0}^{d_1-2} \left( (k+1) \frac{|\Psi^{k+1} \psi|}{|\Psi^{k} \psi|} - \frac{1}{2} \langle y_{k+1}, \Psi y_k \rangle \right) (y_k^* \otimes y_{k+1}  + y_{k+1}^* \otimes y_{k}) \\
&  =   - \sum_{k =0}^{d_1-2} (k+1)(C_k + C_k^\dagger) +\frac{1}{2} \sum_{k =0}^{d_1-2} \langle y_{k+1}, \Psi y_k \rangle  (y_k^* \otimes y_{k+1}  + y_{k+1}^* \otimes y_{k})
 \end{align*}

We note that since $D_t y_k = Q y_{k}$, then
$$\dvec( y_k^* \otimes y_{k+1}) = Q ( y_k^* \otimes y_{k+1}) -  ( y_k^* \otimes y_{k+1})Q,$$
so in particular $\dvec Q =0$ and $\dvec S^\sharp_{\calE} = Q S^\sharp_{\calE} - S_{\calE}^\sharp Q$

Let $\wp_b$ be the canonical twist functions. Then 
$$\wp_1 y_k = \psi  \wedge y_k-  \frac{|\Psi^{k+1} \psi|}{|\Psi^{k} \psi|} (k y_{k+1} +(k+1) y_{k-1}) - \langle y_{k-1}, \Psi y_k \rangle y_{k-1} - \frac{1}{2} \sum_{s=0}^{k-3} \langle y_s, \psi y_k \rangle y_s .$$
We will complete the computation for the case $d_1 =3$.

\subsection{Case $d_1 =3$} We complete the computation for the special case of $d_1 = 3$. We will identify $\mathfrak{o}(3)$ with $\mathbb{R}^3$ by the map $x \wedge y \mapsto x \times y$, where $\times$ is the standard inner product. The Lie algebra $\mathfrak{g}$ is then given by
$$[(x,\mathbf{x}), (y, \mathbf{y})] = ( \kappa (x \times \mathbf{y} + \mathbf{x} \times y) , x \times y + \mathbf{x} \times \mathbf{y}).$$
Define $(\psi, \bpsi) : T^*M \to \mathfrak{g}$ by $p(x ,\mathbf{x})|_{\pi(p)} = \langle \psi(p), x \rangle +  \langle \bpsi(p), \mathbf{x} \rangle, p \in T^*M.$
We obtain
\begin{align*}
P_1 (x, \mathbf{x}) & =(\kappa \psi \times \mathbf{x}, \psi \times x),\\
\dvec \eul (x, \mathbf{x}) &  =  \langle \bpsi \times \psi, x \rangle, \\
A((x , \mathbf{x}), (y , \mathbf{y})) & = - \frac{1}{2} \langle \bpsi \times x + \kappa  \psi \times \mathbf{x}  , y \rangle - \frac{1}{2} \langle \psi \times x + \bpsi \times \mathbf{x} , \mathbf{y} \rangle .
\end{align*}
In particular,
$$\dvec \psi = \bpsi \times \psi, \qquad \dvec \bpsi =0, \qquad A^\sharp(x,\mathbf{x}) = \frac{1}{2} (- \bpsi \times x -\kappa \psi \times \mathbf{x} ,0 ).$$
Consider the set
$$\Sigma M = \{ p \in T^* M \, : \, \bpsi(p) \times \psi(p) \neq 0 \}.$$
On this set, we define an orthonormal basis $y_1|_p$, $y_2|_p$, $y_3|_p$ such that $y_0 = \frac{1}{|\psi|} \psi$, $y_1 = \frac{1}{|\bpsi \times \psi|} \bpsi \times \psi$ and $y_2 = y_0 \times y_1$. Write $r = |\psi |$, $\rho = |\bpsi |$ and let $\varphi$ be the angle between them, so that
$$\bpsi = \rho (\cos \varphi y_0 + \sin \varphi  y_2)$$
Observe that $\dvec r = 0$, $\dvec \rho = 0$, $\dvec \varphi = 0$ and $\dvec y_j = Q y_j = \bpsi \times y_j$. Hence we can write $Q$ in the basis $y_0, y_1, y_2$ as
$$[Q] = \rho \begin{pmatrix} 0 & - \sin \varphi & 0 \\  \sin \varphi & 0 & - \cos \varphi \\ 0 & \cos \varphi & 0 \end{pmatrix}$$
Write $y_j = (y_j,0)$ and $\mathbf{y}_j = (0, y_j)$. We see that,
\begin{align*}
P_1 y_0 & = 0,& \quad P_1 y_1 & =  r \mathbf{y}_2, &  \quad  P_1 y_2 & = -r \mathbf{y}_1 , \\
P_2y_0 & = -r\rho \sin \varphi \mathbf{y}_2 \mod \calE, & \quad P_2 y_1 &= 0 \mod \calE, & \quad P_2y_2 & =  r\rho \sin \varphi \mathbf{y}_0 \mod \calE.
\end{align*}
$$P_3y_1 = -  r\rho^2 \sin(\varphi) \cos(\varphi) \mathbf{y}_0 \mod \mathfrak{E}^2, \qquad P_3y_2 =0 \mod \mathfrak{E}^2.$$
In summary, we have $B = 0$ and $C$ given by
$$Cy_0 = \rho \sin \varphi y_1, \qquad Cy_1 = \rho \cos \varphi y_2, \qquad Cy_2 = 0. $$
Hence,
\begin{align*}
S^\sharp &|_{\calE} = -\frac{1}{2} \rho \sin \varphi ( y_0^* \otimes y_1 + y_1^* \otimes y_0) - \frac{3}{2}\rho \cos \varphi (y_1^* \otimes y_2 + y_2^* \otimes y_1)
\end{align*}
We can then do the following computations,
\begin{align*}
\wp_1y_1 & = (Q + P_1 + A^\sharp + S^\sharp )y_1 = r \mathbf{y}_2 - \rho ( \sin \varphi y_0 + \cos \varphi y_2); \\
\wp_1 y_2& = (Q + P_1 + A^\sharp + S^\sharp )y_2 = -r \mathbf{y}_1 - 2 \rho \cos \varphi y_1 ;
\end{align*}
\begin{align*}
\wp_2 y_1 &= 0 = r (\dvec + P_1 + A^\sharp ) \mathbf{y}_2 - \rho \sin \varphi S^\sharp y_0- \rho \cos \varphi (\wp_1 - S^\sharp)y_2  r + S^\sharp \wp_1y_1 \\
& = \frac{1}{2} (  \kappa r^2  +  \rho^2  )y_1  +S^\sharp \wp_1 y_1 ;  
\end{align*}
so $S^\sharp \wp_1 y_1 =  - \frac{1}{2} ( \kappa r^2  +  \rho^2 )y_1$. Furthermore,
\begin{align*}
\wp_2 y_2 &= -r (\dvec + P_1 + A^\sharp ) \mathbf{y}_1 - 2 \rho \cos \varphi (\wp_1 - S^\sharp) y_1 + S^\sharp \wp_1 y_2 \\
& =  -r( -\rho \sin \varphi \mathbf{y}_0 + \rho \cos \varphi\mathbf{y}_2 - \frac{1}{2} \kappa r y_2 )- 2 \rho \cos \varphi \wp_1 y_1 \\
& \qquad -  \rho^2 \cos \varphi ( \sin \varphi y_0 + 3 \cos \varphi y_2 ) + S^\sharp \wp_1 y_2\\
&  =  r\rho \sin \varphi \mathbf{y}_0  - 3 \rho \cos \varphi \wp_1 y_1 \\
& \qquad -  2\rho^2 \cos \varphi \sin \varphi y_0 - \left( \frac{1}{2} \kappa r^2  +4 \rho^2 \cos^2 \varphi\right) y_2 + S^\sharp \wp_1 y_2
\end{align*}
and
\begin{align*}
\wp_3 y_2 &=  0 = (\dvec + P_1 + A^\sharp + S^\sharp) \wp_2 y_2 = (\dvec + P_1 ) \wp_2 y_2 \mod \calE \\
& = - \rho^2 \sin^2 \varphi \wp_1 y_2   - \left( \frac{1}{2} \kappa r^2  +4 \rho^2 \cos^2 \varphi\right) \wp_1 y_2 + \wp_1S^\sharp \wp_1 y_2 \mod \calE ,
\end{align*}
so
$$S^\sharp \wp_1 y_2 = \left( \frac{1}{2} \kappa r^2  + \rho^2 +3 \rho^2 \cos^2 \varphi\right) y_2 + S(\wp_1 y_2, y_0) y_0.$$
Finally
\begin{align*}
\wp_3 y_2 &=  0 = (\dvec + P_1 + A^\sharp + S^\sharp) \wp_2 y_2  \\
&  = S^\sharp \wp_2 y_2  + 3 \rho \cos \varphi S^\sharp \wp_1 y_1 +  2\rho^2 \cos \varphi \sin \varphi S^\sharp y_0+  \left( \frac{1}{2} \kappa r^2  +4 \rho^2 \cos^2 \varphi\right) S^\sharp y_2 \\
& \qquad + (\dvec S( \wp_1 y_2, y_0) ) y_0 - S(\wp_1 y_2, y_0) S^\sharp y_0 - \left( \frac{1}{2} \kappa r^2  + \rho^2 +3 \rho^2 \cos^2 \varphi\right) S^\sharp y_2 \\
&  = S^\sharp \wp_2 y_2  - \frac{1}{2} \rho \cos \varphi  (3 \kappa r^2  +2  \rho^2 + \rho \cos^2 \varphi )y_1  \\
& \qquad + (\dvec S( \wp_1 y_2, y_0) ) y_0 + \frac{1}{2} \rho \sin \varphi S(\wp_1 y_2, y_0) y_1 
\end{align*}

To finally find the complete connection, we turn to the curvature normalization condition
$$\mathfrak{R}^S(y_1, \wp_1 y_1) =0, \qquad \mathfrak{R}^S(y_2, \wp_1 y_2) =0, \qquad \mathfrak{R}^S(\wp_1 y_2, \wp_2 y_2) =0,$$
$$\mathfrak{R}^S(\wp_2 y_2, y_2) =0, \qquad \mathfrak{R}^S(\wp_2 y_2, y_1) =0, \qquad \mathfrak{R}^S(\wp_1 y_2, y_1) =0,$$
$$\mathfrak{R}^S( y_2, y_0) = 0.$$
Using Remark~\ref{re:CurvWp}.
\begin{align} \label{Model3Curv}
\mathfrak{R}^S(\wp_{b} y_a,\wp_{j} y_i) & = \langle A^\sharp(\wp_b y_a), A^\sharp (\wp_j y_i) \rangle_g - \langle S^\sharp(\wp_b y_a), S^\sharp(\wp_j y_i) \rangle_g  \\ \nonumber
& \qquad - \dvec S(\wp_b y_a,\wp_j y_i) + S(\wp_{b+1} y_a , \wp_{j} y_i) + S(\wp_b y_a, \wp_{j+1} y_i).
\end{align}

It follows that
\begin{align*}
0 =\mathfrak{R}^S(y_2, y_0) & 
=  - \rho^2 \sin \varphi \cos \varphi+ S^\sharp(\wp_1 y_2, y_0), 
\end{align*}
and hence
\begin{align*}
S^\sharp \wp_1 y_2 & = \left( \frac{1}{2} \kappa r^2  + \rho^2 +3 \rho^2 \cos^2 \varphi\right) y_2 +  \rho^2 \sin \varphi \cos \varphi y_0;  \\
S^\sharp \wp_2 y_2 & =  \frac{3}{2} \rho \cos \varphi  ( \kappa r^2  +  \rho^2  )y_1;  \\
\wp_2 y_2 &=   r\rho \sin \varphi \mathbf{y}_0  - 3 \rho \cos \varphi \wp_1 y_1  -  \rho^2 \cos \varphi \sin \varphi y_0 +   \rho^2 \sin^2 \varphi y_2 .
\end{align*}
Finally, observe that
By further applying \eqref{Model3Curv},
\begin{align*}
0= \mathfrak{R}^S( y_1,\wp_{1} y_i) & =   S(\wp_{1} y_1 , \wp_{1} y_1) , \\
0 =\mathfrak{R}^S(\wp_{2} y_2, y_2) &   =  \frac{1}{2} \rho^2 \cos^2 \varphi  \left( 3 \kappa r^2   + 4 \rho^2 + 2 \rho^2 \cos^2 \varphi \right) + S(\wp_2 y_2, \wp_{1} y_2) \\
0 = \mathfrak{R}^S(\wp_{2} y_2, y_1) & =  S(\wp_2 y_2, \wp_{1} y_1) ,\\
0 = \mathfrak{R}^S(\wp_1 y_2,\wp_2 y_2) & =   S(\wp_{2} y_2 , \wp_2 y_2) ,\\
0 = \mathfrak{R}^S(\wp_{1} y_2, y_2) & =  S(\wp_1 y_2, \wp_{1} y_2) ,\\
0 = \mathfrak{R}^S(\wp_{1} y_2, y_1) & =   \rho \cos \varphi \left( \kappa r^2  + 3 \rho^2 + 4 \rho^2 \cos^2 \varphi\right)   + S(\wp_1 y_2, \wp_{1} y_1) .
\end{align*}
In conclusion, if we use the notation $\alpha \beta = \frac{1}{2} (\alpha \otimes \beta + \beta \otimes \alpha)$, and define $\alpha_{1,1}, \alpha_{2,1}, \alpha_{3,1}, \alpha_{1,2}, \alpha_{2,2}, \alpha_{1,3}$ as the coframe of $y_2, y_1, y_0 , \wp_1 y_2, \wp_1 y_1, \wp_2 y_2$, then
\begin{align*}
S&  = - \rho \sin \varphi \alpha_{2,1} \alpha_{3,1} - 3 \rho \cos \varphi \alpha_{1,1} \alpha_{2,1} + 2\rho^2 \sin \varphi \cos \varphi \alpha_{3,1} \alpha_{1,2} \\
& \qquad - (\kappa r^2 + \rho^2) \alpha_{2,1} \alpha_{2,2}  + \left(  \kappa r^2  + 2\rho^2 +6 \rho^2 \cos^2 \varphi \right) \alpha_{1,1} \alpha_{1,2} \\
& \qquad + 3 \rho \cos \varphi (\kappa r^2 + \rho^2) \alpha_{2,1} \alpha_{1,3}   - 2 \rho \cos \varphi \left( \kappa r^2  + 3 \rho^2 + 4 \rho^2 \cos^2 \varphi\right) \alpha_{1,2} \alpha_{2,2} \\
& \qquad  -  \rho^2 \cos^2 \varphi  \left( 3 \kappa r^2   + 4 \rho^2 + 2 \rho^2 \cos^2 \varphi \right) \alpha_{1,2} \alpha_{1,3}
\end{align*}
Finally, again using \eqref{Model3Curv}, we know $(\Ric) = (\Ric^{\mathsf{a},\mathsf{b}} )_{(\mathsf{a}, \mathsf{b}) \in \mathsf{Y}}$ from $\Ric^{3,1} = 0$ and,
\begin{align*}
\Ric^{2,1} & =\mathfrak{R}^S(y_1, y_1) = - \kappa r^2 - (1+ 2 \cos^2 \varphi) \rho^2 ; \\
\Ric^{2,2} & =  \mathfrak{R}^S(\wp_1 y_1, \wp_1  y_1) = - \kappa r^2 \rho^2 \cos^2\varphi - \rho^4 \sin^2 2 \varphi; \\
\Ric^{1,1} & =  \kappa r^2  +2 \rho^2 +4 \rho^2 \cos^2 \varphi ;\\
\Ric^{1,2} & = \mathfrak{R}^S(\wp_1 y_2, \wp_1 y_2) \\
&  = - \rho^2 \kappa r^2 (1+  \cos^2 \varphi ) - \rho^4 ( 1 + 10 \cos^2 \varphi +10 \cos^4 \varphi) ; \\
\Ric^{1,3} & = \mathfrak{R}^S(\wp_2 y_2, \wp_2 y_2) \\
& = - \rho^4 \cos^2 \varphi (1  + 2 \cos^2 \varphi)  (  3 \kappa r^2   +4 \rho^2 \sin^2 \varphi ).\end{align*}

%% file: BConnect.tex
\section{Some identities on connections} \label{sec:Connections}
\subsection{Pullback bundles and connections} \label{sec:Pullback}
The following formalism is included for the convenience of the reader unfamiliar with the pullback bundles and connections. For more details, we refer to e.g. \cite[Chapter~6.8, Chapter~9.1]{KMS93}. If $\pi: \calA \to N$ is a vector bundle over $N$ and $f: M \to N$ is a smooth map of manifolds, we define \emph{the pullback bundle} $f^* \pi: f^* \mathcal{A} \to M$ of $\calA$ over $M$ as
$$f^* \calA = \{ (x, a) \in M \times \calA \, : \, f(x) = \pi(a) \}, \qquad f^* \pi: (x,a) \mapsto x.$$
For every section $X \in \Gamma(\calA)$, we can define a corresponding section $f^* X \in \Gamma(f^* \calA)$ by
$$f^* X|_x = X_{f(x)}.$$
Not all sections of $f^* \calA$ are on this form in general, however, such elements always form a basis over $C^\infty(M)$. Hence, if $\nabla$ is an affine connection on $\calA$, we can define a connection $f^* \nabla$ on $f^* \calA$ by the Leibniz rule and by the relation
$$(f^* \nabla)_v f^* X = \nabla_{f_* v} X, \qquad v \in TM, X \in \Gamma(\calA) .$$

For the typical example, let $\pi: TM\to M$ be the tangent bundle over a manifold $M$, and let $\gamma: (-\ve, \ve) \to M$ be a smooth curve into $M$. Then sections of~$\gamma^* TM$ are vector field along the curve $\gamma$ and if $\nabla$ is an affine connection on $TM$, then $D_t = (\gamma^* \nabla)_{\frac{\partial}{\partial t}}$ is the corresponding covariant derivative along the curve.

\subsection{Useful curvature identities}
We will give some curvature identities that is used throughout the paper. All of these computations use the first Bianchi identity for a connection $\nabla$ with torsion $T$,
\begin{equation} \label{Bianchi} \circlearrowright R(X, Y, Z) = \circlearrowright (\nabla_X T)(Y,Z) + \circlearrowright T(T(X,Y),Z), \qquad X, Y, Z \in \Gamma(TM), \end{equation}
where $\circlearrowright$ denotes the cyclic sum.

Let $(M,\calE, g)$ be a sub-Riemannian manifold and let $\calA$ be a subbundle such that $TM = \calE \oplus \calA$.  Let $\pr_{\calE}$ and $\pr_{\calA}$ be the corresponding projections.
\begin{lemma} \label{lemma:Curv0}
Let $\nabla$ be a connection compatible with $(\calE, g)$ and preserving $\calA$ under parallel transport. For every $X, Y \in \Gamma(\calE)$ and $Z \in \Gamma(\calA)$, we write
$$B_Z(Y) X := \left(\circlearrowright (\nabla_X T)(Y,Z) + \circlearrowright T(T(X,Y),Z) \right) = - B_Z(X)Y.$$
and let $(\pr_{\calE} B_Z(X))^\dagger$ denote the adjoint with respect to $g$. Then
\begin{equation} \label{HHVCurv0} R(X, Y) Z =  \pr_{\calA} B_Z(Y) X,\end{equation}
and
\begin{align} \label{HVHCurv0}
R(X, Z) Y = \frac{1}{2} \pr_{\calE} B_Z(X) Y - \frac{1}{2} ( \pr_{\calE} B_Z(Y))^\dagger X -  \frac{1}{2} (\pr_{\calE} B_Z(X))^\dagger Y.
\end{align}
\end{lemma}

\begin{proof}
Using that the endomorphism $R(X,Y)$ preserves $\calE$ and $\calA$, we can take the projection to $\calA$ to both sides of \eqref{Bianchi} to obtain \eqref{HHVCurv0}. To prove \eqref{HVHCurv0}, we will determine the symmetric and the anti-symmetric part of $R(\, \cdot \, , Z) \, \cdot \,$. For the anti-symmetric part, we obtain that
\begin{align*}
& R(X, Z) Y - R(Y, Z) X = -  \pr_{\calE} \circlearrowright R(X, Y) Z = \pr_{\calE} B_Z(X) Y.
\end{align*}
For the symmetric part, we have that from compatibility of the metric,
\begin{align*}
& \langle R(Y, Z) Y, X \rangle_g = - \langle Y, R(Y,Z) X \rangle_g = - \langle Y, \circlearrowright R(X,Y) Z   \rangle_g = - \langle B_Z(Y)^\dagger Y,X \rangle_g.
\end{align*}
The result follows.
\end{proof}

We will look at a particular choice of connection preserving the decomposition $TM = \calE \oplus \calA$. For any section $Z \in \Gamma(\calA)$, we define $\tau_Z: TM \to TM$ by
$$(\calL_Z \pr_{\calE}^*g)(\pr_{\calE} X, \pr_{\calE} Y ) = 2 \langle \tau_Z X, Y \rangle_g, \qquad X, Y \in \Gamma(TM).$$
We note that $\tau_Z (TM) \subseteq \calE$ and that $Z \mapsto \tau_Z$ is tensorial. Choose a taming Riemannian metric $\bar{g}$ such that $\calE$ and $\calA$ are orthogonal, and define $\nabla$ as in \eqref{NiceNabla}.
We then have the following corollary of Lemma~\ref{lemma:Curv0}.
\begin{corollary} \label{cor:Curv}
Introduce the tensor
$$K(X, Y) = - \pr_{\calE} [\pr_{\calA} X, \pr_{\calA} Y ] - \pr_{\calA} [\pr_{\calE} X, \pr_{\calE} Y ],$$
and write $K_X = K(X, \, \cdot \,)$.
For any $X, Y \in \Gamma(\calE)$ and $Z \in \Gamma(\calA)$, we have
\begin{align*}
R(X, Y) Z & = (\nabla_Z K)(X, Y) + K(\tau_Z X, Y) + K(X, \tau_Z Y), \\\
R(X, Z) Y & =  \sharp \langle (\nabla \tau)_Z X, Y \rangle_g -  (\nabla_Y \tau)_Z X \\
& \qquad + \frac{1}{2} K_Z K_X Y   - \frac{1}{2} (K_Z K_Y)^\dagger X  - \frac{1}{2} (K_Z K_X)^\dagger Y.
\end{align*}
\end{corollary}
\begin{proof}
We note that the torsion of $\nabla$ equals
$$T(u,v) = K(u,v) + \tau_u v - \tau_v u, \qquad u,v \in TM.$$
Hence for any $Z \in \Gamma(\calA)$, $X,Y \in \Gamma(\calE)$, we have
\begin{align*}
B_Z(X) Y& = (\nabla_X \tau)_Z Y - (\nabla_Y \tau)_Z X  - (\nabla_{Z} K)(X,Y) \\
& \qquad + K_Z K_XY - K(\tau_Z X,  Y) - K(X,\tau_Z Y) .\end{align*}
The result follows.
\end{proof}

We also have the following result regarding the curvature of $\nabla$.
\begin{lemma} \label{cor:CurvV}
For any $X \in \Gamma(\calE)$, $Z_1, Z_2 \in \Gamma(\calA)$, we have
\begin{align*}
R(X,Z_1)Z_2 & = - \frac{1}{2} K(X, K(Z_1, Z_2)) + \frac{1}{2} (K_X K_{Z_1})^\dagger Z_2 + \frac{1}{2} (K_X K_{Z_2})^\dagger Z_1\\
& \qquad  +\frac{1}{2} \bar{\sharp} (R(X, Z_1) \bar{g})(Z_2, \, \cdot \,) +\frac{1}{2} \bar{\sharp} (R(X, Z_2) \bar{g})(Z_1, \, \cdot \,).
\end{align*}
\end{lemma}
\begin{proof}
Let $X \in \Gamma(\calE)$ and $Z, W \in \Gamma(\calA)$ be arbitrary. We again look at the anti-symmetric part
\begin{align*}
& R(X, Z_1) Z_2 - R(X, Z_2) Z_1 \\
& = \pr_{\calA} \left( \circlearrowright (\nabla_X T)(Z_1, Z_2) + \circlearrowright T(T(X,Z_1), Z_2) \right)  =  K(K(Z_1, Z_2), X),
\end{align*}
and the symmetric part
\begin{align*}
&\langle R(X, Z) Z, W \rangle_{\bar{g}} = (R(X, Z)\bar{g})(Z,W) - \langle Z, R(X, Z) W \rangle_g \\
& = (R(X, Z)\bar{g})(Z,W) + \langle Z, K_X K_Z W \rangle_g,
\end{align*}
giving us the result.
\end{proof}

\subsection{Non-affine connections on vector bundles} \label{sec:NonlinearC}
The following formalism can be applied to any vector bundle, but we will focus on the specific case of the cotangent bundle.
Let $\pi: T^*M \to M$ be the canonical projection with vertical bundle $\calV = \ker \pi_*$. Let $\calH$ be an Ehresmann connection on $\pi$, i.e. a subbundle of $T(T^*M)$ satisfying $T(T^*M) = \calH \oplus \calV$. Let $X \mapsto hX$ be the horizontal lift of a vector field on $M$ with respect to $\calH$. Since $hX$ and $\vl \alpha$ are $\pi$-related to respectively $X$ and $0$ for $X \in \Gamma(TM)$, $\alpha \in \Gamma(\pi^* T^*M)$, we know that $[hX, \vl \alpha]$ is a section of $\calV$. We define $\nabla_X \alpha \in \Gamma(\pi^* T^*M)$ by
$$[hX, \vl \alpha] = \vl \nabla_X \alpha.$$
If $f \in C^\infty(T^*M )$, then we define $\nabla f \in (\pi^*T^*M)$ by $(\nabla f)|_a (v) = df(h_a v)$ and note that
$$\nabla_X \alpha = (\nabla f)(X) \alpha+ f \nabla_X \alpha. $$
We define the curvature $\calR(X, Y) \in \Gamma(\pi^* T^*M)$ by
$$[hX, hY] = h[X,Y] - \vl \calR(X,Y).$$

The connection $\calH$ is called affine if for the maps ${\boldsymbol \cdot}_c : T^*M \to T^*M$ and ${\boldsymbol +}:T^*M \oplus T^*M \to T^*M$,
$${\boldsymbol \cdot}_c(p) = c p, \qquad {\boldsymbol +}(p \oplus  p_2) = p + p_2,\qquad p, p_2 \in \calA, c \in \mathbb{R}.$$
we have $({\boldsymbol \cdot}_c)_* \calH_p \subseteq \calH_{cp}$ and ${\boldsymbol +}_*  (\calH_{p} \oplus \calH_{p_2}) \subseteq \calH_{p+p_2}$. If $X \in \Gamma(TM)$, $\alpha \in \Gamma(T^*M)$, then for affine connections, the corresponding covariant derivative $\nabla_X \alpha$, is well-defined as a section of $\Gamma(T^*M)$. Furthermore, we have $\calR(X,Y)|_p = R(X,Y) p$.

Now, let $\calH$ be an affine connection corresponding to covariant derivative $\nabla$. We parametrize all Ehresmann connections on $\pi$ by sections $\psi \in \Gamma(T^*M \otimes \pi^* T^*M)$ and we write
$$\calH^\psi := \{ h_p v - \vl_p \psi|_p(v) \, : \, (p,v) \in T^*M \oplus TM \}, \qquad \psi \in \Gamma(T^*M \otimes \pi^* T^*M).$$
We note the corresponding covariant derivative is then
\begin{equation} \label{ConnNonlinear} \nabla_X^\psi \alpha = \nabla_X \alpha + (\vl \alpha)( \psi(X) ) -( \vl\psi(X)) \alpha ,\end{equation}
with curvature
\begin{align} \label{CurvNonlinear} \calR^\psi(X,Y) &= \calR(X, Y) + (\nabla_X \psi)(Y) - (\nabla_Y \psi)(X) \\ \nonumber
& \qquad + \psi(T(X,Y)) + (\vl \psi(X)) \psi(Y) - (\vl \psi(Y)) \psi(X),\end{align}
with $(\nabla_X \psi)(Y) = (\pi^* \nabla)_{hX} \psi(Y) - \psi(\nabla_X Y)$.

In the expression \eqref{ConnNonlinear} and \eqref{CurvNonlinear}, we have terms containing vertical derivatives of sections of $\Gamma(\pi^* T^*M)$. We explain why this is well defined. We can see any $E \in \Gamma(\pi^* T^*M)$ as a map $E: T^*M \to T^*M$ satisfying $E(T_x^* M) \subseteq T^*_xM$ for any $x$ in $M$. Hence, for any $p, \alpha \in T_x^*M$, the map $t \mapsto E(p+ t \alpha)$ is a curve in the vector space $T_x^*M$. As a consequence, $\vl_p \alpha E = \frac{d}{dt} E(p+t\alpha)|_{t=0}$ is well defined as an element in $T^*_xM$.